\documentclass[english,a4paper,11pt]{article}
\usepackage[a4paper, margin=25.4mm]{geometry}
\usepackage[utf8]{inputenc}
\usepackage{graphicx}
\usepackage[dvipsnames]{xcolor}
\usepackage{amsmath,amsthm}
\usepackage[colorlinks=true,allcolors=black]{hyperref}

\usepackage{amsopn}

\usepackage{mathtools,microtype,lmodern, amsmath,amsfonts,amssymb}
\mathtoolsset{centercolon}
\usepackage{stmaryrd} %
\SetSymbolFont{stmry}{bold}{U}{stmry}{m}{n}  %
\usepackage{xfrac} %

\numberwithin{equation}{section}

\usepackage{comment}

\usepackage{enumerate}
\usepackage{tikz}
\usetikzlibrary {arrows.meta}

\DeclarePairedDelimiter{\abs}{\lvert}{\rvert}
\DeclarePairedDelimiter{\norm}{\lVert}{\rVert}
\DeclarePairedDelimiter{\bra}{(}{)}
\DeclarePairedDelimiter{\pra}{[}{]}
\DeclarePairedDelimiter{\set}{\{}{\}}
\DeclarePairedDelimiter{\skp}{\langle}{\rangle}

\definecolor{rot}{rgb}{1.000,0.000,0.000}

\newcommand{\cD}{\mathcal{D}}
\newcommand{\cE}{\mathcal{E}}

\newcommand{\cR}{\mathcal{R}}
\newcommand{\cS}{\mathcal{S}}

\renewcommand{\L}{\operatorname{L}}
\newcommand{\W}{\operatorname{W}}
\newcommand{\BV}{\operatorname{BV}}
\newcommand{\TV}{\operatorname{TV}}

\newcommand{\wstarlim}{\stackrel{*}{\rightharpoonup}}

\newcommand{\D}{\mathrm{D}}
\renewcommand{\d}{\dx}

\definecolor{rot}{rgb}{1.000,0.000,0.000}

\newcommand{\CE}{\mathrm{CE}}
\newcommand{\one}{1\!\!1}

\usepackage{babel}

\newcommand{\N}{\mathbb{N}}
\newcommand{\Z}{\mathbb{Z}}
\newcommand{\R}{\mathbb{R}}

\newcommand{\Prob}{\mathcal{P}}

\newcommand{\dd}{\dx}
\DeclareMathAlphabet{\mathup}{OT1}{\familydefault}{m}{n}
\newcommand{\dx}[1]{\mathop{}\!\mathup{d} #1} %

\newcommand{\eps}{\varepsilon}

\newcommand{\AC}{\textup{AC}}

\DeclareMathOperator{\PME}{PME}

\newcommand{\s}{\mathrm{s}}
\newcommand{\CC}{\mathsf{C}}

\newcommand{\CCC}{\mathsf{C}\!\!\rule[0.05em]{0.1em}{0.6em}\;}
\newcommand{\bbT}{\mathbb{T}}

\def\dd{\:\mathrm{d}}
\def\bbT{\mathbb T}
\def\rmC{\mathrm C}
\def\rmH{\mathrm H}
\def\rmL{\mathrm L}
\def\rmW{\mathrm W}

\newcommand{\sigmaAver}{\overline{\sigma}}
\newcommand{\mAver}{\overline{m}}

\def\qcrit{{q_\mathrm{crit}}}

\providecommand{\tnorm}[1]{{\left\vert\kern-0.25ex\left\vert\kern-0.25ex\left\vert #1 
    \right\vert\kern-0.25ex\right\vert\kern-0.25ex\right\vert}}

\newtheorem{manualresultinner}{Result}
\newenvironment{manualresult}[1]{%
	\renewcommand\themanualresultinner{#1}%
	\manualresultinner
}{\endmanualresultinner}

\newtheorem{theorem}{Theorem}[section]
\newtheorem{proposition}[theorem]{Proposition}
\newtheorem{lemma}[theorem]{Lemma}
\newtheorem{corollary}[theorem]{Corollary}

\newtheorem{definition}[theorem]{Definition}
\newtheorem{assumption}[theorem]{Assumption}
\newtheorem{remark}[theorem]{Remark}

\usepackage{authblk}
\usepackage[hang,flushmargin]{footmisc}

\begin{document}

\title{Derivation of the fourth-order DLSS equation with\\
       nonlinear mobility via chemical reactions}
\author[1]{Alexander Mielke}
\author[2]{André Schlichting}
\author[3]{Artur Stephan}
\affil[1]{{
\small
Weierstra\ss{}-Institut, Berlin, Germany,
and Vienna Centre for Advanced Studies (ViCAS), \newline University of Vienna, Austria,
email: \texttt{alexander.mielke@wias-berlin.de}
}}
\affil[2]{{
\small
Universität Ulm, Germany, email:
\texttt{andre.schlichting@uni-ulm.de}
}}
\affil[3]{{
\small
Technische Universität Wien, Vienna, Austria, email:
\texttt{artur.stephan@tuwien.ac.at}
}}
\date{\today}

\maketitle

\begin{abstract}
We provide a derivation of the  {one-dimensional} fourth-order DLSS equation 
based on an interpretation
as a chemical reaction network. We consider the rate equation on the discretized
circle for a process in which pairs of particles occupying the same site
simultaneously jump to the two neighboring sites; the reverse process involves pairs
of particles at adjacent sites simultaneously jumping back to the site located between
them. Depending on the rates, in the vanishing-mesh-size limit we obtain either the
classical DLSS equation or a variant with nonlinear mobility of power type. Via EDP
convergence, we identify the limiting gradient structure to be driven by entropy
with respect to a generalization of diffusive transport with nonlinear mobility.
Interestingly, the DLSS equation with power-type mobility shares
qualitative similarities with the fast diffusion and porous medium equation,
since we find traveling wave solutions with algebraic tails or
compactly supported polynomials, respectively.

\end{abstract}

\renewcommand{\thefootnote}{\fnsymbol{footnote}}
\footnotetext{\emph{Acknowledgement:}
The authors thank Daniel Matthes and Herbert Spohn
for discussions on properties of the DLSS equation and its derivation.
\\[2pt]
\\[2pt]
\\[2pt]
\emph{Funding:}
The research of ASch is partially based upon work from COST Action 24122 mSPACE,
supported by COST (European Cooperation in Science and Technology), www.cost.eu.
The research of AM has been partially funded by DFG through the Berlin Mathematics
Research Center MATH+ (EXC-2046/1, grant no.\ 390685689) subproject
``DistFell''.}
\renewcommand{\thefootnote}{\arabic{footnote}}

\section{Introduction}

In this paper, we provide a microscopic derivation of a generali\-zation of the 
Derrida--Lebowitz--Speer--Spohn (DLSS)  {equation} with nonlinear mobility given by
\begin{equation}\label{eq:DLSS:alpha}
	\partial_t \rho = - \partial_{xx} \bra[\big]{ \rho^{\alpha} \partial_{xx} \log \rho }
\end{equation}
on the torus $\mathbb{T}^1 = \mathbb{R}/\mathbb{Z}$. We consider mobility exponents 
$\alpha>0$. The case with linear mobility $\alpha=1$ was derived
in~\cite{DerridaLebowitzSpeerSpohn1991a,DerridaLebowitzSpeerSpohn1991b} as the
law of fluctuations of the Toom interface model, but also occurs in the
quantum drift-diffusion~\cite{DegondMehatsRinghofer}. 
The generalization with mobility $\alpha\ne 1$ is so far not studied in the literature. 
We will derive~\eqref{eq:DLSS:alpha} as the macroscopic limit of rate 
equations for a binary reaction network. In this situation $\alpha=2$ is a natural 
choice, but our approach provides a derivation of~\eqref{eq:DLSS:alpha} for 
all $\alpha>0$. 

The work focuses on the derivation of~\eqref{eq:DLSS:alpha} from a microscopic 
model and provides a first analytic framework identifying possible weak solutions 
and the {thermodynamically} consistent a priori bound 
for~\eqref{eq:DLSS:alpha}. However, 
we expect that~\eqref{eq:DLSS:alpha} has a rich dynamical behaviour mimicking that 
of the fast diffusion $(\alpha<1)$ and porous medium equation $(\alpha>1)$. The
details are left for further studies and we provide at the end of the introduction
some first indication in this direction. 
In detail, we provide traveling wave solutions on $\R$ with algebraic tails 
($\alpha<1)$ or compact support $(\alpha>1)$ as well as numerical tests of 
source-type solutions mimicking the according behaviour of the tails.

Our derivation of~\eqref{eq:DLSS:alpha} is motivated by a physical interpretation 
of  the recently proposed numerical scheme~\cite{MRSS} in the case of $\alpha=1$ as a
chemical network of $N\in \mathbb{N}$ binary reactions of the type 
\begin{equation}\label{eq:reaction}
	X_{k-1}+X_{k+1}\xrightleftharpoons{\quad} 2X_{k}  \,, \qquad\text{ for } 
    k\in [N] := \set[\big]{1,\dots,N} \,.
\end{equation}
We consider the vector of concentrations $c=\bigl(c_{k}\bigr)_{k=1,\dots,N}  %
\in \mathcal{P}^N$, that is a discrete probability measure on $[N]$, where each 
entry provides the concentration of the species $X_k$ in~\eqref{eq:reaction} 
(with the convention $c_0=c_N$). 
They evolve according to the chemical reaction-rate equation by
\begin{subequations}\label{eq:RRE}
	\begin{align}
	\dot{c}_{k} &= N^2 \bra[\big]{ J_{\alpha,k-1}[c] - 2J_{\alpha,k}[c]+ J_{\alpha,k+1}[c]} \label{eq:RRE:CEN} \\
	&\quad\text{with}\quad J_{\alpha,k}[c] =  \sigma_\alpha(c_{k-1},c_k,c_{k+1}) N^2 \bigl( c_{k}^{2}-c_{k-1}c_{k+1} \bigr) \label{eq:RRE:flux} \,,
\end{align}
\end{subequations}
 {incorporating the stoichiometric vector $(\dots,1,{-}2,1,\dots)^\mathrm{T}$ of the 
chemical reaction \eqref{eq:reaction} and the reaction rate (also called \textit{reaction intensity}) 
$\sigma_\alpha\cdot(c_k^2-c_{k-1}c_{k+1})$. The function $\sigma_{\alpha}$, which in the following 
will be called \textit{activity function} plays a crucial role in our modeling as it 
provides a parameter-dependent tunable function to change the reaction rate. 
A possible choice for $\sigma_\alpha$ is motivated from the context of chemical reactions
by the mass action kinetics and is given as $\sigma_2 \equiv 1$.
In general, $\sigma_\alpha$ will be chosen to be homogeneous of degree $\alpha-2$, 
which will turn~\eqref{eq:RRE} into a discrete approximation of~\eqref{eq:DLSS:alpha} 
(see the exact Assumption~\ref{ass:Activity} below).}

Since the system~\eqref{eq:RRE} originates from a rate equation of the chemical 
network~\eqref{eq:reaction}, it has a good global existence theory and also positivity properties.
Crucially, the solution of the chemical reaction-rate equation~\eqref{eq:RRE} can 
be characterized variationally as a generalized gradient flow 
\cite{LieroMielkePeletierRenger2017, PeletierSchlichting2022, Miel23IAGS} 
in continuity-equation format for the entropy
\begin{equation}\label{eq:def:energy}
	E_{N}(c):=\frac{1}{N}\sum_{k=1}^{N}\bra[\big]{ c_{k} \log c_k - c_k + 1} \,.
\end{equation}
The role of the continuity equation is taken by the first equation 
in~\eqref{eq:RRE:CEN} and a general curve $(c,J): [0,T]\to \Prob^N \times \R^N$ 
solving~\eqref{eq:RRE:CEN} is denoted by $(c,J)\in\CE_N$. 
The specific choice of the flux $J_{\alpha}$ in~\eqref{eq:RRE:flux} is encoded 
by an energy-dissipation balance (EDB), which includes the time-integrated 
dissipation functional $D_{\alpha,N}(c,J)$ (see~\eqref{eq:def:dissipation} below), 
and is given by the full energy-dissipation functional
\begin{equation}\label{eq:def:ED-functional}
	L_{\alpha,N}(c,J) := E_{\alpha,N}(c(T))-  E_{\alpha,N}(c(0)) 
    + D_{\alpha,N}(c,J) \,.
\end{equation}
For positive and smooth curves $(c,J)\in\CE_N$ with $c:[0,T]\to \Prob^N_{>0}$, 
the construction in Section~\ref{ssec:discrete:EDB} ensures by the Young-Fenchel
inequality that $L_{\alpha,N}(c,J) \geq 0$. 
In our first result, we identify the classic solutions to~\eqref{eq:RRE} having 
$E_N$ as Lyapunov function as global minimizers of the energy-dissipation 
functional, that is $L_{\alpha,N}(c,J)=0$ and are called \emph{energy-dissipation 
balance (EDB) solutions}. This will be the starting point for the convergence 
analysis of the system. 

\begin{manualresult}{A}[Well-posedness and variational characterization of~\eqref{eq:RRE}]\label{result:discrete}
	Assume the  {activity} $\sigma_\alpha$ satisfies Assumption~\ref{ass:Activity}. 
	For all $c^0 \in \Prob^N$, the system~\eqref{eq:RRE} has a global differentiable 
    solution $c:[0,\infty) \to \Prob^N$ such that $c(t)\in \Prob^N_{>0}$ is strictly
    positive for all $t>0$. The constructed solution is an EDB solution satisfying
    $L_{\alpha,N}(c,J)=0$. 
\end{manualresult}

The Result~\ref{result:discrete} is proven in Proposition~\ref{prop:WellposednessODE} 
and Proposition~\ref{prop:GFDiscrete} in Section~\ref{sec:Discrete}.

We turn next to the limit $N\to \infty$ and note that the total scaling $N^4$ 
in~\eqref{eq:RRE} is chosen to obtain a macroscopic limit, if the concentration 
vector~$c^N$ is embedded into the space of continuous densities on the torus~$\mathbb{T}$. In this scaling, 
we arrive at the fourth-order equation~\eqref{eq:DLSS:alpha}. For the heuristic 
argument, we observe that the outer Laplacian in~\eqref{eq:DLSS:alpha} is already
explicit as formal limit of the discrete second-order continuity equation 
in~\eqref{eq:RRE:CEN}.
The diffusive flux term $\rho^2\partial_{xx}\log \rho$ is generated by the 
\emph{geometric} Laplacian in~\eqref{eq:RRE:flux} by a formal expansion 
\begin{equation}\label{eq:flux:rewrite:expand}
	\begin{split}
	N^2\bra[\big]{ c_k^2- c_{k-1}c_{k+1}} %
	&=-  N^2 c_k^2 \pra[\Big]{ \exp \bra[\big]{ \log c_{k-1}+\log c_{k+1} - 2 \log c_k} -1 } \\
	&\approx -\rho(x)^2 \partial_{xx} \log \rho + O(N^{-2}).
\end{split}
\end{equation}

The rigorous proof of the limit is based on the energy-dissipation principle, which 
is a thermodynamic formulation for the discrete \eqref{eq:RRE} as well as 
continuous~\eqref{eq:DLSS:alpha} gradient structure. In fact, our results shows 
that the limiting gradient structure for~\eqref{eq:DLSS:alpha} is driven by the
continuous entropy given by
\begin{equation}\label{eq:def:entropy}
{\cal E}(\rho) =\int_{\mathbb{T}} \bra[\big]{ \rho(x) \log \rho(x) - \rho(x) +1 }\dx x. %
\end{equation}
We formulate the limit~\eqref{eq:DLSS:alpha} in terms of a continuous second-order
continuity-type equation and a constitutive relation for the flux, that is
\begin{equation}\label{eq:DLSS:system}
		\partial_t \rho =  {\partial_{xx}} j_{\alpha} \qquad\text{ and }\qquad j_{\alpha} = - \rho^\alpha \partial_{xx} \log \rho \,.
\end{equation}
Again, as in the discrete case, a general solution pair 
$(\rho,j): [0,T]\to \L^1(\mathbb{T})\times \W^{2,1}(\mathbb{T})$ to the second-order 
continuity equation $\partial_t \rho =  {\partial_{xx}} j$ is denoted by 
$(\rho,j)\in \CE$, understood in a suitable weak sense. The specific flux 
$j_{\alpha}=-\rho^\alpha \partial_{xx} \log \rho$ for the solution in~\eqref{eq:DLSS:system} 
is encoded through a suitable total dissipation functional 
$\cD_{\alpha}$, see~\eqref{eq:ContinunousDissipationFunctional}.

\vspace{0.5em}
\indent \textit{Remark} (Multi-dimensional formulation). While our analysis is confined to space dimension $d=1$, the model has a formal
analogue on the torus $\bbT^d$. The natural generalization of~\eqref{eq:DLSS:alpha}
preserves its fourth-order structure by reading the outer operator $\partial_{xx}$
as a double divergence and the inner operator $\partial_{xx}\log\rho$ as a Hessian,
\[
  \partial_t \rho = - \sum_{i,j=1}^d \partial_i\partial_j\bigl(\rho^\alpha\, \partial_i\partial_j\log\rho\bigr)
  = - D^2 \!:\! \bigl(\rho^\alpha\, D^2\log\rho\bigr),
\]
where $D^2 f=(\partial_i\partial_j f)_{i,j=1}^d$ denotes the Hessian and $A\!:\!B=\sum_{i,j}A_{ij}B_{ij}$.
On the discrete level, the reaction network~\eqref{eq:reaction} generalizes to the
lattice $(\Z/N\Z)^d$ by allowing reactions across all double edges, i.e.\ for every
ordered pair of coordinate directions,
which reassemble into
the discrete Hessian as $N\to\infty$. 
The functional-analytic
difficulties, however, become even more severe than for $d=1$: the compactness
argument underlying Corollary~\ref{cor:StaticBounds} and Lemma~\ref{lem:ImprovedRegularity}
relies on the embedding $\W^{1,4}(\bbT)\subset\L^\infty(\bbT)$, which fails for
$d\geq2$, and the slope representations of Lemma~\ref{lem:RelaxedSlope} no longer
possess a canonical vectorial counterpart. We therefore restrict the present work to~$d=1$.

\vspace{0.5em}

These ingredients provide the limiting gradient structure by defining
the full energy-dissipation functional for a curve $(\rho,j)\in\CE$ by
\begin{align}\label{eq:EDfunctional}
	\mathcal{L}_\alpha(\rho,j):=\cE(\rho(T))-\cE(\rho(0)) + \mathcal{D}_\alpha(\rho,j). 
\end{align}
The total dissipation $\cD_{\alpha}$ consists of the time integral of a primal 
dissipation $\cR_\alpha(\rho,j)$ and the slope $\cS_\alpha(\rho)$. The primal 
dissipation is given for $j\in \L^1(\mathbb{T})$ and $j \d x \ll \rho \d x$ by 
\begin{equation}\label{eq:def:primal:dissipation}
	{\cal R}_\alpha(\rho,j) :=
	\frac 1 2\int_{\mathbb{T}}\frac{j^{2}}{\rho^\alpha}\d x, \qquad\text{ for } j\in\L^1(\bbT) \text{ such that } j \dx{x}\ll \rho \dx{x} \,.
\end{equation}
The slope term $\cS_{\alpha}$ is formally defined by inserting the driving force 
$- {\partial_{xx}} \D\cE(\rho) = -  {\partial_{xx}} \log \rho$ into the dual dissipation 
potential~\eqref{eq:def:dual:dissipation:continuous}, obtained as the 
Legendre-Fenchel dual of $\cR_\alpha$ from~\eqref{eq:def:primal:dissipation}, 
that is $\mathcal{S}_{\alpha,+}(\rho) = \cR_\alpha^*(\rho,- {\partial_{xx}}\D\cE(\rho))$. 
Due to the presence of the logarithm and Laplacian, this is only justified for 
positive and smooth concentrations, where we indeed get
\begin{align}\label{eq:def:slope+}
	{\cal S}_{\alpha,+}(\rho) &:= \frac12  
    \int_{\mathbb{T}}\rho^{\alpha }\bigl( {\partial_{xx}}\log\rho\bigr)^{2}\d x.
\end{align}
The main ingredient of the analysis is its relaxed lower-semicontinuous envelope 
$\cS_\alpha$ defined on a suitable family of non-negative Sobolev functions 
(see Lemma~\ref{lem:RelaxedSlope}).

For this reason, the non-negativity of $\mathcal{L}_{\alpha}$ is only ensured 
for sufficiently regular curves $(\rho,j)\in \CE$. Such regularity is a priori not 
known for the obtained solutions to~\eqref{eq:DLSS:alpha}.
Hence, we distinguish two  {notions} of solutions, those satisfying the 
\emph{energy-dissipation balance} (EDB) $\mathcal{L}_\alpha(\rho,j)=0$ and the 
weaker form satisfying the \emph{energy-dissipation inequality} (EDI) 
$\mathcal{L}_\alpha(\rho,j)\leq 0$.

As we will see, the convergence result in Section~\ref{sec:EDPconvergence}
constructs a curve $\rho$ that is a priori only an EDI solution.
 {Precise definitions of the EDI and EDB concepts at the discrete level are 
given in Section~\ref{ssec:discrete:EDB}, and at the continuum level in the 
beginning of Section~\ref{sec:EDPconvergence}.}

\begin{manualresult}{B}[EDP convergence]\label{result:EDP}
	Assume the {activity} $\sigma_\alpha$ satisfies Assumption~\ref{ass:Activity}.\\
	Let $(c^N,J^N)\in \CE_N$ be such that $\sup_{N} L_{\alpha,N}(c^N,J^N)<\infty$ 
    and $\sup_{N} E_N(c^N(0)) <\infty$. Then there exist suitable embeddings 
    $(\iota_N c^N, {\cal I}_N J^N)\in \CE$ such that the sequence 
    $(\iota_N c^N, {\cal I}_N J^N)$ converges to $(\rho,j)\in \CE$ and it holds
	\begin{align*}
		\forall\, t>0: \ \liminf_{N\to\infty}E_{N}(c^{N}(t)) \geq{\cal E}(\rho(t)) 
        \qquad\text{and}\qquad 
		\liminf_{N\to\infty}D_{\alpha,N}(c^{N},J^{N})  \geq{\cal D}_\alpha(\rho,j).
	\end{align*}
	If, in addition the curves $c^N$ are EDB solutions to~\eqref{eq:RRE}, that 
    is $L_{\alpha,N}(c^N,J^N)=0$, and have well-prepared initial data, i.e.\  
    $\iota_Nc^N(0)\to\rho(0)$ and $E_N(c^N(0))\to\cE(\rho(0))$, then the 
    limit curve $(\rho,j)$ is an EDI solution to~\eqref{eq:DLSS:alpha}, that 
    is $\mathcal{L}_{\alpha}(\rho,j)\leq 0$.
\end{manualresult}
The full statement is contained in Theorem~\ref{thm:ConvergenceResult} and proven in Section~\ref{sec:EDPconvergence}.

Under suitable assumptions, we can show that the so-obtained EDI solution is indeed 
an EDB solution by showing that the energy-dissipation functional 
$\mathcal{L}_{\alpha}$ is non-negative on its domain. 
For doing so, we show a \emph{chain-rule} in Proposition \ref{pr:ChainRule}, which 
allows us to conclude that any EDI solution $(\rho,j)$ is already an EDB solution, 
and at the same time identifies a suitable weak solution to~\eqref{eq:DLSS:alpha}.
 {The well-posedness of the weak formulation relies on the algebraic identity 
for smooth positive densities $\rho$
\begin{equation}\label{eq:flux-rewriting}
  \rho^\alpha \partial_{xx} \log \rho = \frac{1}{\alpha}\Bigl(\partial_{xx}\rho^\alpha - 4|\partial_x\rho^{\alpha/2}|^2\Bigr),
\end{equation}
which rewrites the diffusive flux purely in terms of Sobolev-regular quantities 
$\rho^{\alpha/2}\in\W^{2,2}$ and $\rho^{\alpha/4}\in\W^{1,4}$, thereby bypassing the logarithm.}

\begin{manualresult}{C}[EDB and weak solutions]
\label{result:EDB-weaksol} Consider $\alpha>0$ and an EDI solution $(\rho,j)$ 
satisfying one of the following conditions:
\begin{subequations}
	\label{eq:CR.AddiCond}
	\begin{align}
		\label{eq:CR.AddiCond.a}
		& \alpha=1;
		\\
		\label{eq:CR.AddiCond.b}
		&\alpha\in {]0,2]} \ \text{ and } \  \rho \in \rmL^\infty([0,T]{\times} \bbT);
		\\
		\label{eq:CR.AddiCond.c}
		& \alpha>0,  \ \ \rho \in \rmL^\infty([0,T]{\times} \bbT), \text{ and } \
		\exists\,\delta>0: \ \rho(t,x) \geq \delta \text{ a.e.}
	\end{align}
\end{subequations}
Then, $(\rho,j)$ is an EDB solution. Moreover, these EDB solutions are also weak solutions 
of $ \partial_t\rho + \frac1\alpha \partial_{xx}\bigl(\partial_{xx} \rho^\alpha - 
4|\partial_x \rho^{\alpha/2}|^2\bigr)=0$, namely  for all 
$\psi \in \rmC^2_\mathrm{c}({[0,T)}{\times}\bbT)$ it holds
\begin{equation*}
	\int_\bbT \rho(0)\,\psi(0) \dx x + \int_0^T \!\!\int_\bbT 
	\rho \,\partial_t \psi \dx x \dx t \\
	= \int_0^T \!\!\int_\bbT \frac1\alpha \bigl( \partial_{xx} \rho^{\alpha} -
	4|\partial_x \rho^{\alpha/2}|^2\bigr) \;\partial_{xx} \psi \dx x \dx t  \,,
\end{equation*}
where, setting  $p_\alpha=\max\bigl\{(4{+}\alpha)/(2{+}\alpha), 4/3\bigr\}$,  
the flux $j$ satisfies 
\begin{align*}
j= & - \frac1\alpha \bigl( \partial_{xx} \rho^{\alpha} -
   4|\partial_x \rho^{\alpha/2}|^2\bigr) =-\frac2\alpha  \rho^{\alpha/2} 
   \bigl( \partial_{xx} \rho^{\alpha/2} - 4|\partial_x \rho^{\alpha/4}|^2\bigr) 
 \in \rmL^{p_\alpha}([0,T]{\times} \bbT) \,.
\end{align*}
If $ \rho \in \rmL^\infty([0,T]{\times} \bbT) $, then we always have $ j \in \rmL^2([0,T]{\times} \bbT)$.
\end{manualresult}
The well-posedness of the term on the right-hand side is ensured by finiteness 
of the relaxed slope~$\cS_{\alpha}$ for variational solutions and a posteriori 
for the identified weak solutions demanding to satisfy the energy dissipation balance, that is
\begin{equation*}
	\frac{\d}{\dx t} \cE(\rho) = - \frac{4}{\alpha^2} \int_{\mathbb{T}} \bigl( \partial_{xx} \rho^{\alpha/2} -
	4|\partial_x \rho^{\alpha/4}|^2\bigr)^2 \dx{x} \qquad\text{ for a.e. } t\in [0,T]\,.
\end{equation*}
 {The three cases in~\eqref{eq:CR.AddiCond} correspond to different dynamical 
regimes. The case~\eqref{eq:CR.AddiCond.a} covers the classical DLSS case $\alpha=1$ 
without further assumptions, here we expect bounded positive solutions for any positive time $t>0$. 
The case~\eqref{eq:CR.AddiCond.b} covers the range $\alpha\in{]0,2]}$ with bounded solutions, 
corresponding to the fast-diffusion and moderate-mobility regimes.
The case~\eqref{eq:CR.AddiCond.c} requires solutions that are bounded and, in addition, 
strictly positive. While this assumption is crucial for our existence analysis, we do not 
expect it to be essential in general. Indeed, the compactly supported profiles observed 
numerically (see Figure \ref{fig:SelfSimi}) are incompatible with this assumption.
In this regime, the identification of EDB solution as weak solution remains an open problem.}

 {The fact that three cases must be distinguished stems from the qualitatively different 
behavior of the dissipation functional $\cal{D}_\alpha$ for the various parameters $\alpha>0$. 
The primal dissipation density $R(\rho,j):=j^2/(2\rho^\alpha)$ is jointly convex in $(\rho,j)$ 
if and only if $\alpha\in(0,1]$. For $\alpha>1$ the lack of joint convexity of $R$ is the key 
obstruction to proving the chain-rule inequality for general EDI solutions: the convexity argument 
that works for quadratic (classical) gradient flows is not available. The additional 
conditions~\eqref{eq:CR.AddiCond.b}--\eqref{eq:CR.AddiCond.c} compensate for this by ensuring 
the integrability and uniform positivity needed for a direct computation.
Similarly, the slope $\cS_\alpha$ is jointly convex (via formula~\eqref{eq:RelaxedSlope.c}) for 
$\alpha\in[3/2,2]$, but not outside this range, as the coefficient $\frac{2\alpha-3}{3}$ changes 
sign at $\alpha=3/2$.}

 {The uniqueness of EDB and weak solutions to~\eqref{eq:DLSS:alpha} is not established 
in this work and is largely open. For $\alpha=1$, uniqueness of positive smooth weak 
solutions is known \cite{JuengelMatthes2008},  but more delicate for data with vanishing 
density~\cite{Fischer-uniqueness}. For $\alpha\ne1$, the equation is new and uniqueness 
is entirely open. The standard De~Giorgi-type argument for uniqueness of EDB solutions 
does not apply, since neither $\cR_\alpha$ nor $\cS_\alpha$ (via~\eqref{eq:RelaxedSlope.c}) 
are jointly convex in $(\rho,j)$ in a common regime for~$\alpha>0$.}

\subsection*{Relation to the literature}

We note that our Results~\ref{result:discrete}--\ref{result:EDB-weaksol} lead to 
several generalizations of the work~\cite{MRSS}, even in the case $\alpha=1$. 
We provide a complete variational EDP convergence and identification of the same 
weak solution constructed as in~\cite{MRSS}. Moreover, we allow for a greater 
range of semi-discretizations based on the choice of the activity function 
$\sigma_{\alpha}$, where our Assumption~\ref{ass:Activity} covers the specific 
choice from~\cite{MRSS}, however at the drawback that our scheme in general 
preserves only the entropy as Lyapunov function, whereas~\eqref{eq:DLSS:alpha} 
for $\alpha=1$ has a rich family of further Lyapunov functions. At this point, 
we refer to~\cite[Sec.\,1]{MRSS} for an extensive discussion of the origin, 
the structural properties and various numerical schemes for the DLSS 
equation ($\alpha=1$) as well as to~\cite{MRS2025} for the discussion of the 
induced distance from the primal dissipation~\eqref{eq:def:primal:dissipation}.

In this work, we pioneer the derivation of~\eqref{eq:DLSS:alpha} for the full 
range~$\alpha>0$ via EDP convergence and provide the existence of variational 
solutions satisfying an energy-dissipation inequality or a conditional 
energy-dissipation balance. Both concepts are a generalization of curves of 
maximal slope introduced by De Giorgi~\cite{DeGiorgiMarinoTosques80} 
in the form developed in \cite{AGS2005, RoMiSa08MACD, Miel23IAGS} relying  
on the energy dissipation balance based on a primal and a dual dissipation potential.
The notion of EDP convergence was introduced (informally and without name) in 
\cite{LieroMielkePeletierRenger2017} and then conceptually studied in 
\cite{DoFrMi19GSWE, MiMoPe21EFED, PeletierSchlichting2022} and \cite[Sec.\,5.4]{Miel23IAGS}. It is a 
refinement of the Sandier-Serfaty approach to $\Gamma$-convergence for gradient 
flows~\cite{SandierSerfaty2004, Serfaty2011} and allows to study general multiscale
limits like homogenization~\cite{Mielke2016}, layer-to-membrane limits 
\cite{LieroMielkePeletierRenger2017,FreLie21EDTS,Miel23NESS}, fast-slow reaction 
systems~\cite{FreLie21EDTS, MielkePeletierStephan2021, Stephan2021a, Miel23NESS}, 
or discrete-to-continuum limits~\cite{DisserLiero, HraivoronskaSchlichtingTse2023, 
EspositoHeinzeSchlichting2025, HeinzePietschmannHeinze2024arXiv, HeinzeMielkeStephan2025} 
as in the present paper. 

The derivation of thermodynamically consistent continuum models from stochastic 
or discrete dynamics is a recent undertaking and we comment on some recent 
literature. In the context of linear regular Markov jump processes the 
work~\cite{PRST22} provides an extensive framework. The justification of 
the macroscopic, exponential kinetic relation (Marcellin-De Donder kinetic) 
from stochastic jump processes was derived via large-deviation theory in 
\cite{MielkePeletierRenger2014,MPPR17NETP} which leads to the so-called
\emph{cosh-gradient structure} used for our discrete model as well. In the context 
of discrete coagulation-fragmentation equations, which show a similar quadratic 
structure as~\eqref{eq:RRE}, the $\cosh$-gradient structure was also identified 
in~\cite{HoeksemaLamSchlichting2025}.

The obtained gradient structure for our continuum model~\eqref{eq:DLSS:alpha} can be 
seen as a second order generalization of the recent novel gradient structure obtained 
for the porous medium equation in~\cite{GesHey25PMEL,FehGes23NELD}, that is
\begin{equation}\label{eq:PME}
	\partial_t \rho = \tfrac{1}{\alpha} \partial_{xx}\bra[\big]{ \rho^\alpha}
    = \partial_x \bra[\big]{ \rho^\alpha \partial_x \log \rho} \,.
\end{equation}
The authors derive~\eqref{eq:PME} as the continuum limit of a suitably rescaled 
zero-range process in the \emph{thermodynamic scaling} limit of infinitely many 
particles and large volume such that the density stays order one. In our situation, 
the ODE system~\eqref{eq:RRE} can be already seen as the infinite particle limit of 
a stochastic model, which in the context of chemical reaction corresponds to the 
chemical master equation~\cite{vanKampen2007,MaaMie20MCRS}. The considered limit 
$N\to \infty$ of the manuscript corresponds to the infinite volume limit in that 
language, see~\cite[Figure 1]{GesHey25PMEL}.

The authors in~\cite{GesHey25PMEL} identify the driving energy as the entropy $\cE$ 
as in our case~\eqref{eq:def:entropy}, see~\cite[Chapter 11]{GesHey25PMEL}. Their 
variational formulation is obtained in terms of curves $(\rho,g): [0,T]\to 
\Prob \times \L^2$ solving the skeleton equation 
\begin{equation*}
	\partial_t \rho = \tfrac{1}{\alpha} \partial_{xx} \rho^\alpha + \partial_x(\rho^{\alpha/2} g) \,,
\end{equation*}
with the control $g\in \L^2([0,T];\L^2)$ and a large-deviation rate functional given 
by $\mathcal{J}_\alpha(\rho,g)= \tfrac{1}{2}\norm{g}_{\L^2\L^2}^2$, where 
$\mathcal{J}_\alpha \equiv 0$ characterizes suitable weak solutions 
to~\eqref{eq:PME}. Their result~\cite[Thm.\,4]{GesHey25PMEL} implies that the 
functional $\mathcal{J}_\alpha$ has indeed the structure of an energy-dissipation 
functional similar to $\mathcal{L}_\alpha$ in~\eqref{eq:EDfunctional}. Indeed, 
along suitable curves $(\rho(t),j(t))_{t\in [0,T]}$ solving the (first-order) 
continuity equation $\partial_t\rho + \partial_x j=0$ in the weak sense, it holds
\begin{equation*}
	\mathcal{J}_\alpha(\rho,g) = \cE(\rho(T)) - \cE(\rho(0)) + \int_0^T \bra[\big]{ \cR_{\alpha}(\rho,j) + \cS_{\alpha}^{\PME}(\rho)} \dx{t} \,,
\end{equation*}
where $\cR_{\alpha}$ is as in~\eqref{eq:def:primal:dissipation} and 
$ \cS_{\alpha}^{\PME} = \frac{2}{\alpha^2} \int \abs[big]{\partial_x 
\rho^{\alpha/2}}^2 \dx{x} $. Again as in our case $\cS_{\alpha}^{\PME}$ can be 
seen as the lower semicontinuous relaxation of the functional $\int \rho^\alpha 
\abs[\big]{ \partial_x \log \rho}^2 \dx{x}$. In this sense, the DLSS equation with
mobility $\alpha>0$ in~\eqref{eq:DLSS:alpha} is a generalization of the porous 
medium equation with respect to a second-order thermodynamic (degenerate) metric 
induced by $\cR_{\alpha}$ from~\eqref{eq:def:primal:dissipation} together with 
the second-order continuity equation given in terms of the first equation 
in~\eqref{eq:DLSS:system}.

The classical DLSS equation ($\alpha=1$) has, besides the gradient structure based on
diffusive transport and driving functional given by the entropy, another gradient structure
based on classical Otto--Wasserstein tensors driven by the Fisher
information \cite{GianazzaSavareToscani2009}.  {For $\alpha\ne1$, extending this
formulation faces a fundamental obstruction: the natural candidate, a weighted
Fisher-information functional with suitably chosen mobility of power-type, 
does not directly
yield the second-order flux coupling $j_\alpha=-\rho^\alpha\partial_{xx}\log\rho$
in~\eqref{eq:DLSS:system} from an Otto--Wasserstein gradient flow in the standard
sense \cite{Otto1998,GiacomelliOtto2001,MatthesMcCannSavare2009} as we show in 
Appendix~\ref{appendix:OttoFisher} (except for special $H^{-1}$ gradient flow 
corresponding to the case $\alpha=-1$, not covered in this work). Instead, these references
provide Otto--Wasserstein gradient-flow structures for the thin-film equation
$\partial_t\rho=-\partial_{x}\bigl(m(\rho)\partial_{xxx}\rho\bigr)$ with a general mobility.
Hence, it seems that the introduced class of equations~\eqref{eq:DLSS:alpha} for 
$\alpha\ne 1$ in this manuscript is fundamentally different.
Moreover, the microscopic cosh-gradient structure developed here, which
arises naturally from the chemical reaction interpretation, provides a
variational formulation valid for all $\alpha>0$ that is thermodynamically consistent
and well-adapted to the discrete-to-continuum limit with driving energy given by the entropy.}

\subsection*{Traveling fronts and similarity profiles}

We close the introduction with an
illustration of some of the expected novel features of~\eqref{eq:DLSS:alpha} due 
to the nonlinear mobility by investigating traveling front solutions and doing 
some numerical experiments. 

First, we provide on the real line $\R$ solutions to the equation 
\[
\partial_t \rho =  -\partial_{xx} \bigl( \rho^{\alpha}\partial_{xx} \log \rho \bigr)
 = - \partial_{xx} \Bigl( \rho^{\alpha-2}\bigl(\rho \partial_{xx} \rho - (\partial_x \rho)^2 \bigr) \Bigr) , \qquad t>0,\ x\in \R,
\]
We observe that for $\alpha>1$ there are explicit solutions that have a moving support. 
We consider the ansatz
\[
\rho(t,x)= \kappa \bigl(ct-x\bigr)^\delta \text{ for } x< ct \quad \text{and}
\quad \rho(t,x)=0 \text{ for }x >ct,
\]
with $\kappa >0$. 
We obtain an explicit solution for $x<ct$ if we choose
\[
\delta=\frac3{\alpha{-}1} \quad \text{and} \quad c=\kappa^{\alpha-1} \,\frac{3(\alpha{+}2)}{(\alpha{-}1)^2}>0. 
\]
For $\alpha\in {]1,7/4[}$ the solution is a classical solution lying in 
$\mathrm C^4(\R)$. For $\alpha\geq 7/4$ the solution is still a weak solution 
if properly defined. Note that $\rho^{\alpha-1}\partial_{xx} \rho$ and 
$\rho^{\alpha-2}(\partial_x \rho)^2$ both behave like $(ct-x)^\sigma$ with 
$\sigma=\alpha \delta{-}2=(\alpha+2)/(\alpha-1)>1$.

\begin{figure}
	\centering
	$\alpha=-1$ \hspace{0.4\linewidth} $\alpha=0$\\[-1.6em]
	\includegraphics[width=0.45\linewidth]{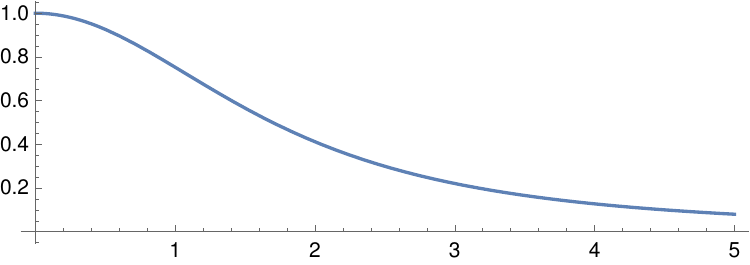}
	\includegraphics[width=0.45\linewidth]{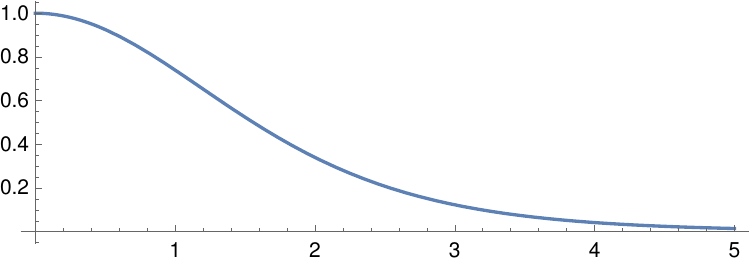}
	\\[1em]
	$\alpha=0.5$ \hspace{0.4\linewidth} $\alpha=1$\\[-1.6em]
	\includegraphics[width=0.45\linewidth]{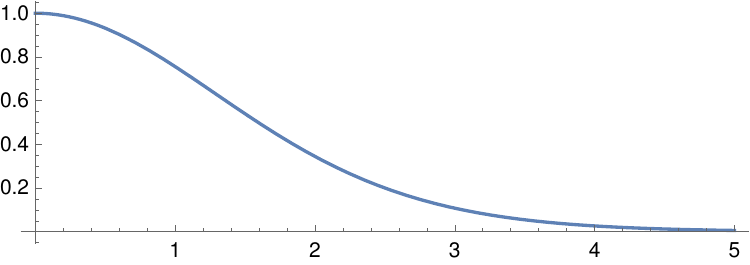}
	\includegraphics[width=0.45\linewidth]{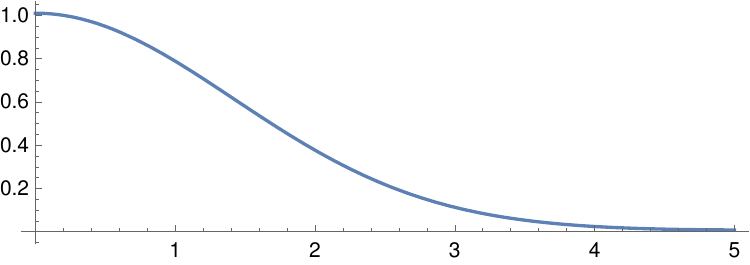}
	\\[1em]
	$\alpha=2$ \hspace{0.4\linewidth} $\alpha=4$\\[-1.6em]
	\includegraphics[width=0.45\linewidth]{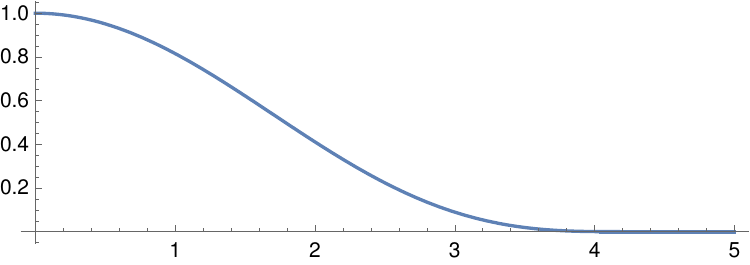}
	\includegraphics[width=0.45\linewidth]{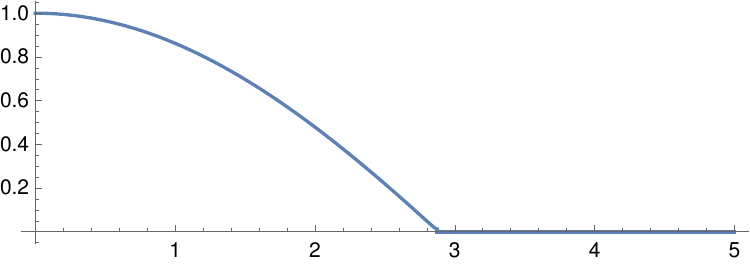}
	\\[1em]
	$\alpha=5$ \hspace{0.4\linewidth} $\alpha=7$\\[-1.6em]
	\includegraphics[width=0.45\linewidth]{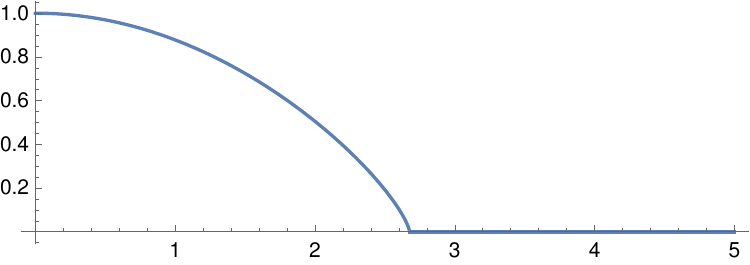}
	\includegraphics[width=0.45\linewidth]{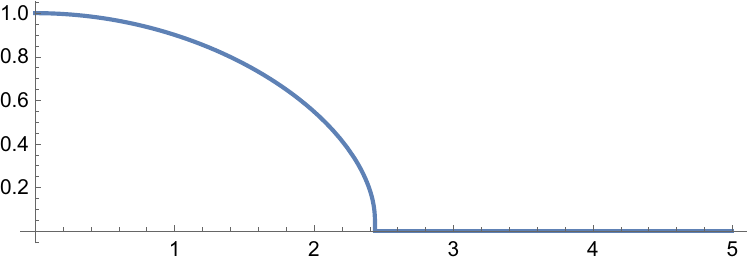}
	
	\caption{Numerically obtained similarity profiles $\Phi_\alpha$ for $\alpha\in \{-1,0,0.5,1,
         {2,4,5,7 } \}$ normalized by $\Phi_\alpha(0)=1$. 
		For $\alpha<1$ the solutions have algebraic decay like $|y|^{-3/(1{-}\alpha)}$; for $\alpha=1$ 
        we have $\Phi_1(y)=\mathrm e^{-y^2/4}$, for $\alpha>1$ the solutions have compact support and 
        behave like $(y_\alpha{-}y)^{3/(\alpha-1)}$ to the left of the right boundary of the support 
        $[-y_\alpha,y_\alpha]$. }
	\label{fig:SelfSimi}
\end{figure}

Moreover, we are able to look for self-similar solutions of having the dynamical 
scaling form $\rho(t,x)= t^{-\gamma} \Phi(x/t^\gamma)$ for a profile function 
$\Phi:\R\to\R$. As the right-hand  side is homogeneous of degree $\alpha$ in 
$\rho$ we find $ \gamma= \frac1{3+\alpha}$. Using the similarity variable 
$y=x t^{-\gamma}$, we find the ODE for the profile function
\[
-\gamma \Phi(y) - \gamma y\Phi'(y) = - \Bigl(\Phi^{\alpha-2} \bigl( \Phi \Phi'' -(\Phi')^2 \bigr) \Bigr)''. 
\]
The left-hand side is a derivative, so we can integrate once and are left with
\begin{equation*}
	\frac1{3+\alpha}y\,\Phi= \Bigl( \Phi^{\alpha-2} \bigl( \Phi \Phi'' -(\Phi')^2 \bigr) \Bigr)'.
\end{equation*}
 {Note, that the integration constant vanishes, since for symmetric solutions $\Phi(-y)=\Phi(y)$ the left-hand side $\frac{1}{3+\alpha}y\,\Phi(y)$ is odd in $y$, and the right-hand side (the derivative of an even function) is likewise odd; hence both sides vanish at $y=0$, which forces the constant to be $0$.}
For the DLSS case with $\alpha=1$ there is the explicit solution $\Phi(y)=\mathrm e^{-y^2/4}$, see~\cite{Bleher}. 

In general, we expect symmetric solutions (i.e.\ $\Phi(-y)=\Phi(y)$) and thus can produce solutions by a shooting method starting with $\Phi(0)=1$, $\Phi'(0)=0$, and $\Phi''(0)=b$, where $b$ needs to be varied to find a sufficiently smooth, non-negative solution in $L^1(\R)$. Figure \ref{fig:SelfSimi} displays the corresponding solutions $\Phi_\alpha$ for $\alpha\in\{-1,0,0.5,1,2,4,5,7\}$. Clearly, for $\alpha>1$ the behavior close to the moving boundary of the support is given by the traveling fronts as constructed above.  {We note that the rigorous existence and uniqueness of the symmetric similarity profile $\Phi_\alpha$ for $\alpha\ne1$ (e.g., via a phase-plane shooting argument as used for the porous medium equation~\cite[Chapter 16]{Vazquez2006}) remains an open problem.}

For comparison we also provide numerical experiments for the equation~\eqref{eq:RRE}, which provides already a spatial discrete approximation to~\eqref{eq:DLSS:alpha} once the activity $\sigma_\alpha$ is specified.
We use the specific choice
\begin{equation*}%
	\sigma_{\alpha}(c_{k-1},c_{k},c_{k+1})=\frac{4}{\alpha^2} \bra[\bigg]{ \frac{c_k^{\frac{\alpha}{2}}- \sqrt{c_{k-1}c_{k+1}}^{\frac{\alpha}{2}}}{ c_k - \sqrt{c_{k-1} c_{k+1}}}}^{\!2} \,, %
\end{equation*}
which satisfies Assumption~\ref{ass:Activity}{; this can be verified by noting that it coincides with the choice $\sigmaAver_\alpha$ from Lemma~\ref{lem:activ}(2) after a straightforward algebraic simplification}.
We use an implicit Euler scheme for the time integration, which is solved using a 
Newton method.  The resulting scheme is implemented in the Julia 
language~\cite{Julia-2017} and the obtained solutions are depicted in 
Figure~\ref{fig:NumSol}. 
\begin{figure}
	\includegraphics[width=0.4\linewidth]{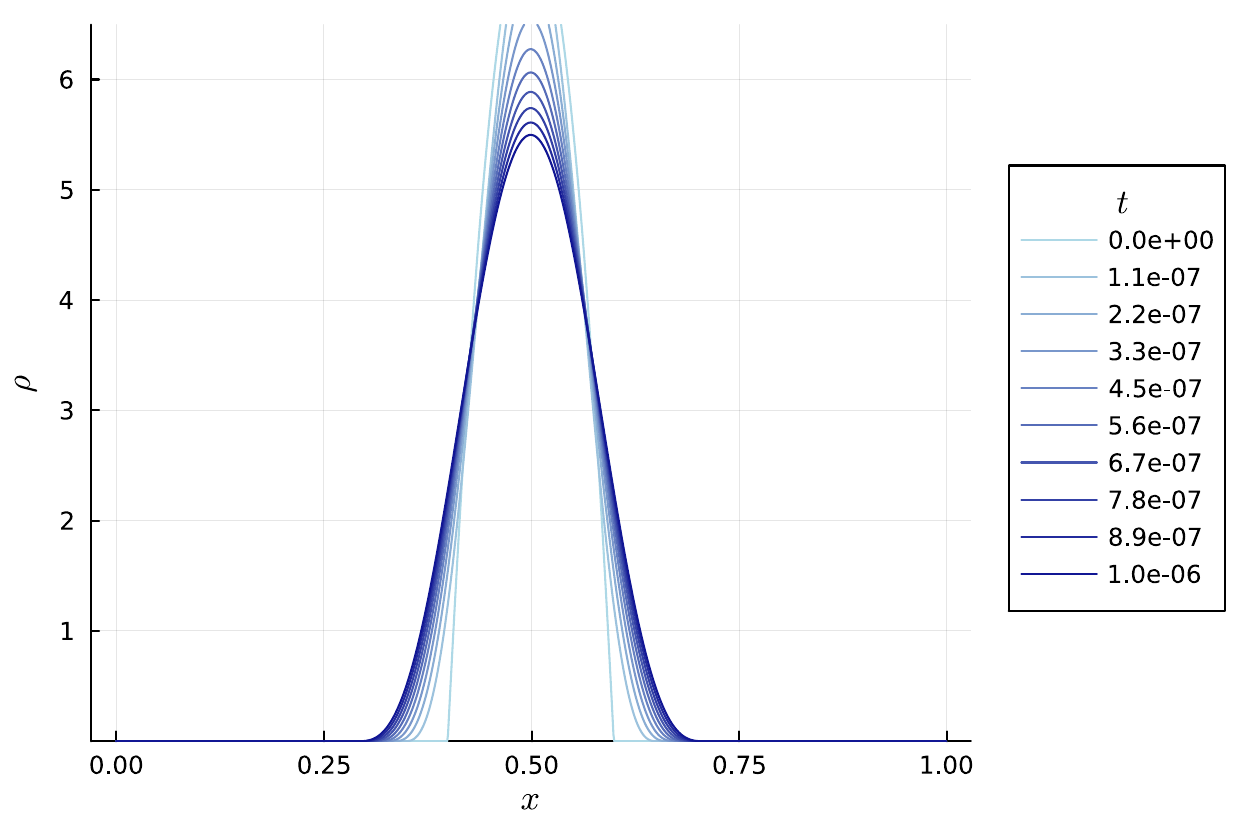}
    \hspace{-2cm}
	\includegraphics[width=0.4\linewidth]{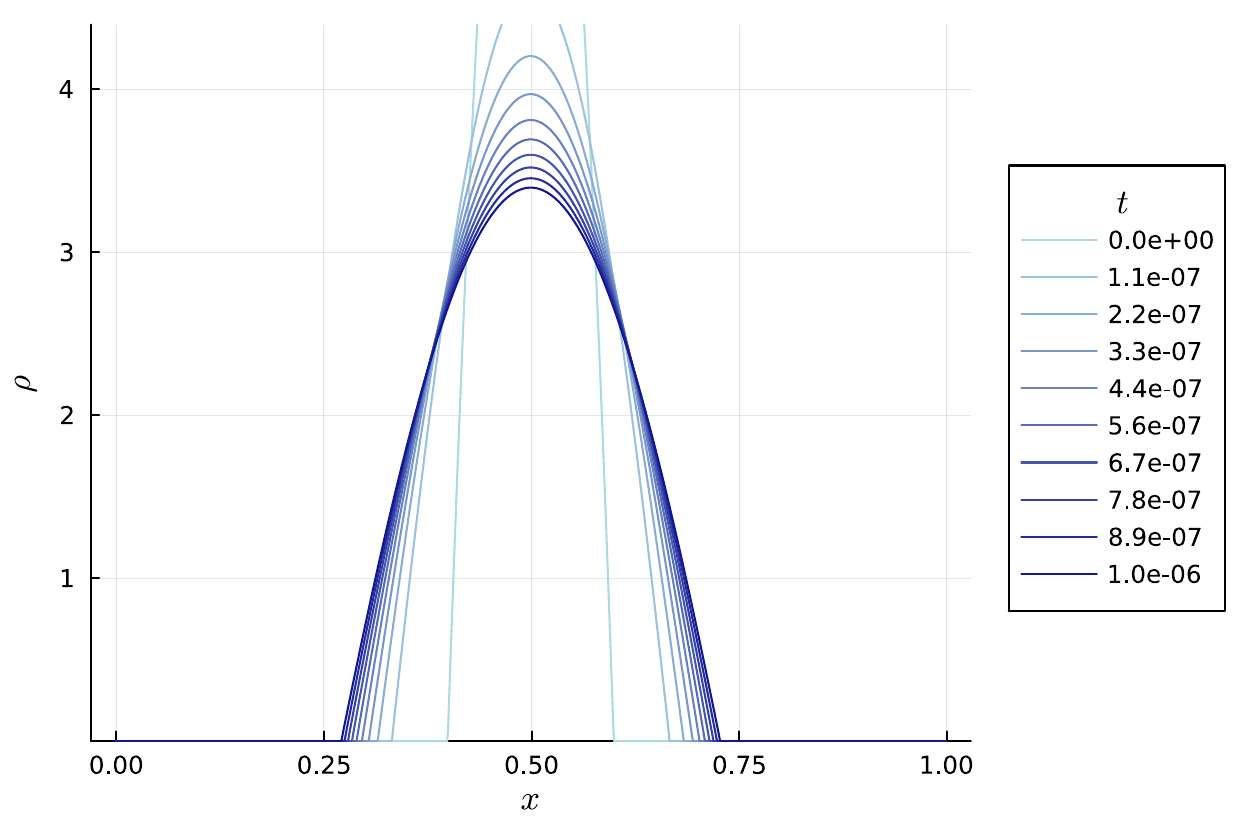}
    \hspace{-2cm}
	\includegraphics[width=0.4\linewidth]{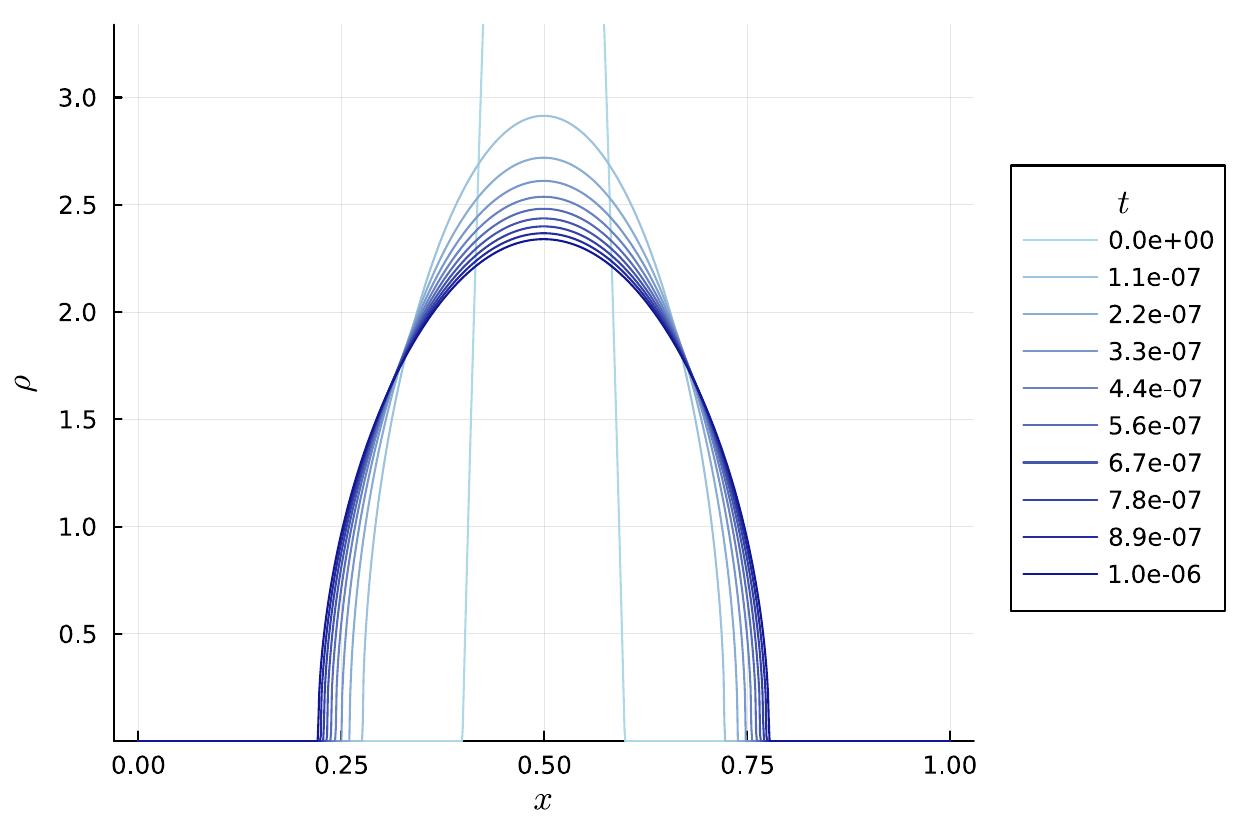}\\
    \includegraphics[width=0.40\linewidth]{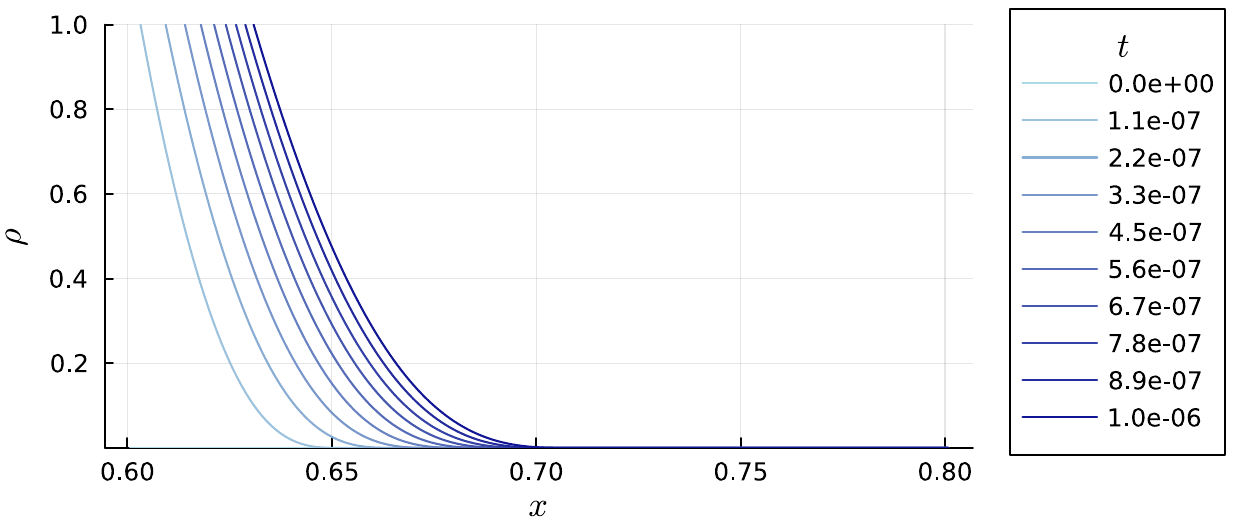}
    \hspace{-2cm}
    \includegraphics[width=0.40\linewidth]{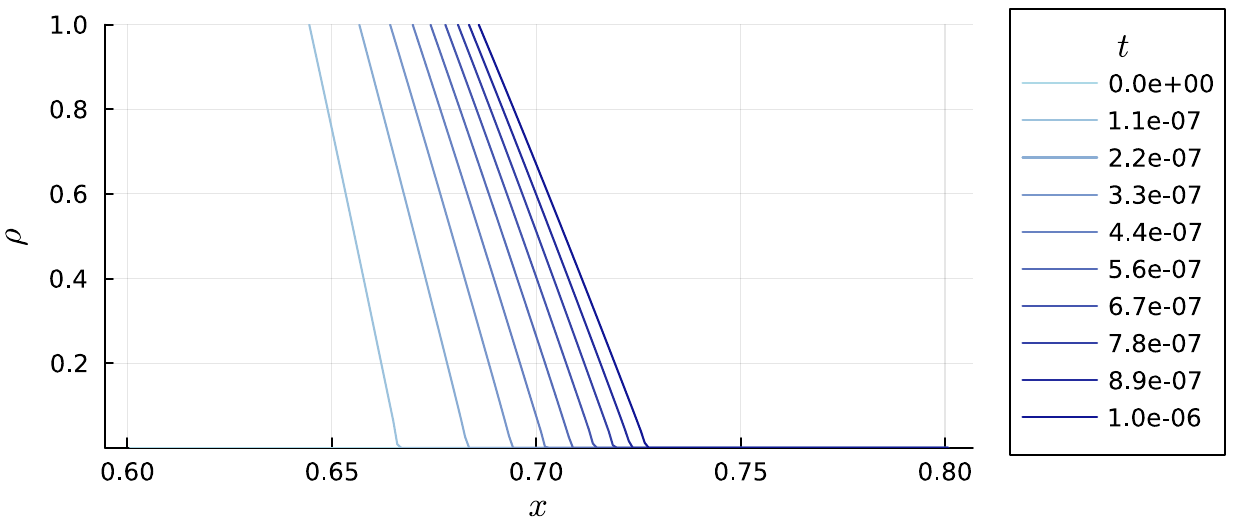}
    \hspace{-2cm}
    \includegraphics[width=0.40\linewidth]{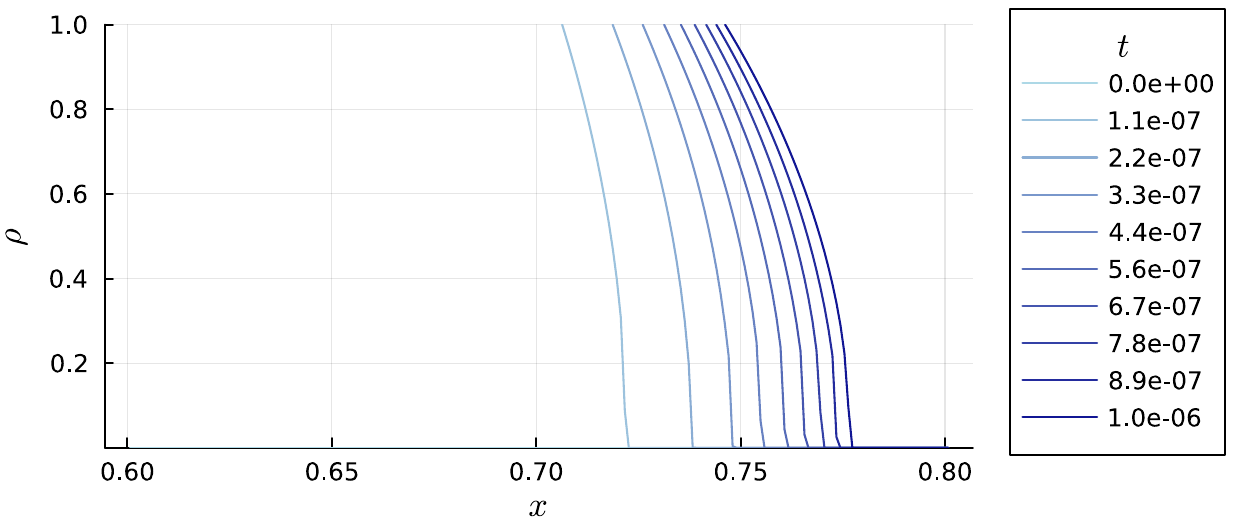}
\caption{Numerically obtained solutions to~\eqref{eq:RRE} for $\alpha\in \{2,4,7\}$ (from left to right).
Starting from a discrete bump function $c^0_k = \max\set[\big]{ 0 ,1-((N/2-k)/(\ell\,N))^2}$ with 
$\ell=0.1$ and $N=2^{10}$.  {The first row shows} the overall evolution and propagation of fronts, 
whereas the second  {row contains} zooms towards the tip of the support. 
We emphasize that in accordance to our result about positivity of solutions to the discrete system, 
the obtained numerical  {solutions} are also positive.}
\label{fig:NumSol}
    
\end{figure}

\section{The discrete model}\label{sec:Discrete}

\subsection{Discrete spaces and operators}

Let $N\in\N$ be fixed in this section. On $\L^{1}(\mathbb{T}_{N})\simeq \R^N$ we consider the canonical scalar product
\begin{equation}\label{eq:def:product:XN}
\langle v,\psi\rangle_{N}=\frac{1}{N}\sum_{k=1}^Nv_{k}\psi_{k}.
\end{equation}
We write $f_+,f_-:[N]\to\R$ for the left/right translates of $f:[N]\to \R$ given by $(f_{\pm})_k = f_{k\pm 1}$ for all $k\in [N]$.
We also have the forward and backward differential operators as well as discrete Laplace operator for such $f:[N]\to \R$ defined by
\begin{equation*}%
	\partial_+^N f = N(f_+ - f) \,,
	\quad 
	\partial_-^N f = N(f - f_-) \,
	\quad 
	 \quad\text{and}\quad
	 \Delta^N f = N^2( f_- - 2 f + f_+ ) \,.
\end{equation*}
We also note, that $\Delta^N = \partial_+^N \partial_-^N =\partial_-^N \partial_+^N$. The operator $\Delta^N$ is a symmetric linear operator with respect to the product~\eqref{eq:def:product:XN}, that is the integration by parts formulas hold
\begin{equation}\label{eq:def:IntByParts}
	\skp{\Delta^N \varphi,\psi}_N= 
	\skp{\partial_+^N \partial_-^N \varphi,\psi}_N=
	-\skp{ \partial_-^N \varphi,\partial_-^N\psi}_N=
	\skp{ \varphi,\partial_+^N \partial_-^N\psi}_N=
	\skp{\varphi,\Delta^N \psi}_N \,.
\end{equation}
Next, we introduce the set of probability densities on the discrete torus by
\begin{equation*}%
	\Prob^N = \set[\Big]{ c\in \L^1(\mathbb{T}_N)\ :\  \forall\, k\in [N]:\ c_k \geq 0,
    \ \ N^{-1} \sum_{k=1}^N c_k = 1 } \subset \L^1(\mathbb{T}_N) \,,
\end{equation*}
as well as the subset of positive probability densities by
\begin{equation*}%
	\Prob^N_{>0} = \set[\big]{ c\in \Prob^N:  c_k > 0 \ \forall\, k\in [N]} \,.
\end{equation*}
For later reference, we define the variational derivative of a functional $E_N: 
\Prob^N\to \R$ at $c^* \in \Prob^N$ as the dual function $E_N'(c^*)$ with respect 
to the product~\eqref{eq:def:product:XN}; more explicitly
\begin{equation}\label{eq:def:VarDeriv}
	(E'_N(c^*))_k = N\,\Bigl.\frac{\partial E_N(c)}{\partial c_k}\Bigr|_{c=c^*} \,.
\end{equation}

\subsection{Well-posedness of the ODE system}

We specify the assumptions on the activity function $\sigma_\alpha$ that describes 
the jump rates. As it turns out, we have some flexibility for the choice, which 
ensures both well-posedness and the convergence to the limit 
system~\eqref{eq:DLSS:alpha}. 

 {%
We first fix the structural \emph{shape} of an admissible activity, deferring the quantitative bounds to Assumption~\ref{ass:Activity}.
\begin{definition}[Activity function]\label{def:activity}
  An \emph{activity function} of mobility exponent $\alpha>0$ is a map
  $\sigma_\alpha:{[0,\infty[}^{3}\to[0,\infty[$ that
  \begin{enumerate}[(i)]
    \item is symmetric in its first and third argument, that is
          $\sigma_\alpha(a,c,b)=\sigma_\alpha(b,c,a)$ for all $a,b,c\geq0$;
    \item depends on the first and third argument only through their geometric
          mean, that is, there exists a function
          $\sigmaAver_\alpha:{[0,\infty[}^{2}\to[0,\infty[$ such that
          \begin{equation}\label{eq:def:activity:reduced}
            \sigma_\alpha(c_{k-1},c_k,c_{k+1})
              =\sigmaAver_\alpha\bigl(c_k,\sqrt{c_{k-1}c_{k+1}}\bigr)\,;
          \end{equation}
    \item is homogeneous of degree $\alpha-2$, that is
          $\sigma_\alpha(\lambda a,\lambda c,\lambda b)=\lambda^{\alpha-2}\,\sigma_\alpha(a,c,b)$
          for all $\lambda>0$ and $a,b,c\geq0$.
  \end{enumerate}
\end{definition}
The homogeneity~(iii) is what makes the discrete rate equation~\eqref{eq:RRE} a consistent 
approximation of~\eqref{eq:DLSS:alpha}, reproducing the degree-$\alpha$ scaling of the limiting 
flux $\rho^\alpha\partial_{xx}\log\rho$; in particular, the flux function $\bar\jmath_\alpha$ 
defined below is homogeneous of degree~$\alpha$. The following assumption selects, among all 
activity functions, the admissible subclass by imposing quantitative bounds that guarantee 
well-posedness of the discrete system and convergence to~\eqref{eq:DLSS:alpha}.}
 {A prototypical example for $\alpha=1$ is the activity function $\sigmaAver_1(x,y)=\frac{2}{x+y}$
from~\cite{MRSS}, which is the inverse of the arithmetic mean of $x$ and $y$. The following conditions 
are designed to include this and related choices for $\alpha>0$.}

\begin{assumption}\label{ass:Activity} For $\alpha>0$ the
\emph{activity} \[\sigma_\alpha = \sigma_\alpha(c_{k-1},c_k,c_{k+1}) =:
\sigmaAver_\alpha(c_k,\sqrt{c_{k-1}c_{k+1}})\]  {(in the sense of Definition~\ref{def:activity})} satisfies:
	\begin{enumerate}[ ({A}1) ]
	\item \label{ass:item:sigma-Aver:continuous}
	  The function $\bar\jmath$ defined via $]0,\infty[^2 \,\ni (x,y)\mapsto \bar\sigma_{\alpha}(x,y) (x^2{-}y^2)=: \bar\jmath_{\alpha}(x,y)$ has 
      a continuous extension to ${[0,\infty[}^2$.
	\item \label{ass:item:sigma-Aver:definit}
			The extended function $\bar\jmath_{\alpha}: {[0,\infty[}^2\to \R$ 
            satisfies $\bar\jmath_\alpha(x,0)>0$ for all $x >0$.
	\item For all $x,y> 0$ the following bounds hold:
	\begin{align}\label{eq:def:sigma-Aver}
		\max\{x^\alpha,y^\alpha\}&\ \geq  \sigmaAver_\alpha(x,y) \ x y\quad 
        \text{and} \notag\\  \sigmaAver_\alpha(x,y) &\ \geq \ 
		\biggl(\frac{ (x^\alpha-y^\alpha)}{\alpha(x-y)}\biggr)^{\!2} 
        \min\{1/x^{\alpha},1/y^{\alpha}\}.
	\end{align}
	\end{enumerate}
\end{assumption}

 {A key feature of Assumption~\ref{ass:Activity} is its \emph{universality}: as shown in 
Theorem~\ref{thm:ConvergenceResult}, any choice of activity function $\sigma_\alpha$ satisfying 
these conditions leads to the same macroscopic limit equation~\eqref{eq:DLSS:alpha}, independently 
of the specific choice. For this, note also that the upper and lower bounds in 
\eqref{eq:def:sigma-Aver} imply  $\sigmaAver_\alpha (x,x)=x^{\alpha-2}$.}

It will be beneficial to rewrite the bound~\eqref{eq:def:sigma-Aver} in terms of the
Stolarsky mean, which is a generalization of the logarithmic mean, to make use of 
many of its properties, see~\cite{Stolarsky1975}. 
\begin{definition}[Stolarsky mean]
	For $p\in \R$, the
	\emph{Stolarsky mean} $\s_{p}:[0,\infty)\times [0,\infty) \to [0,\infty)$ is given by
	\begin{equation*}%
		\s_{p}(x,y)= \Bigl(\frac{x^{p}-y^{p}}{p\,(x{-}y)}\Bigr)^{1/(p-1)} \,,
	\end{equation*}
	with the cases $p=0,1$ and $\s_p(x,x)=x$ understood as a limit. 
\end{definition}
Hence, we can equivalently rewrite the lower bound~\eqref{eq:def:sigma-Aver} in terms of the Stolarsky mean as 
\begin{equation*}%
\sigmaAver_\alpha(x,y)\quad\geq\quad  \s_{\alpha}(x,y)^{2\alpha-2} \min\{1/x^{\alpha},1/y^{\alpha}\}.
\end{equation*}
In particular, we observe that the activity $\sigma_\alpha$ is thus 
$(\alpha{-}2)$-homogeneous and non-negative. 

We also define the so-called \emph{mobility function} by
\begin{equation}\label{eq:def:m-alpha}
	m_{\alpha}(c_{k-1},c_{k},c_{k+1}) = \mAver_\alpha(c_k, \sqrt{c_{k-1}c_{k+1}}):=\sigmaAver_{\alpha}(c_k, \sqrt{c_{k-1}c_{k+1}})\cdot\bigl(c_{k}\sqrt{c_{k-1}c_{k+1}}\bigr),
\end{equation}
we observe that it satisfies thanks to Assumption~\ref{ass:Activity} the inequality
\begin{align}\label{eq:MobilityInequality}
\forall\, x,y\in {[0,\infty[}:\quad \max\{x^\alpha,y^\alpha\} \ \geq \ \mAver_\alpha(x,y) \ \geq \  \s_{\alpha}(x,y)^{2\alpha-2} \, x y\, \min\{1/x^{\alpha},1/y^{\alpha}\}.    
\end{align}

\begin{remark}[Discussion on the assumptions and examples]
The continuity in Assumption \ref{ass:Activity} is used in the proof of existence 
of the ODE system~\eqref{eq:RRE}. The upper bound on $\sigmaAver_\alpha$ provides 
that the mobility is of order $\alpha$ and is crucial for compactness to pass to 
the limit. The lower bound is needed for the proof of a liminf-estimate for the 
slope in the gradient-flow formulation.

Some important cases are $\alpha=1$ and $\alpha=2$. For $\alpha=1$, the classical 
DLSS equation, we observe that $\overline\sigma_\alpha$ has to satisfy 
\[
    \max\{1/ x, 1 / y\}\quad\geq\quad \overline\sigma_\alpha(x,y) \quad\geq\quad \min\{1 /  x,  1 / y\}.
\]
In particular, the scheme introduced in~\cite{MRSS} with the activity function
$\frac{2}{x+y}$ is included in our analysis.
For $\alpha=2$, we have the inequality
\[
\forall\, x,y\in[0,\infty[:\quad \frac{\max\{x^2,y^2\}}{x y}\quad\geq\quad 
\sigmaAver_2(x,y) \quad\geq\quad   \frac{(x+y)^2}{4\max\{x^2,y^2\}},
\]
where we have used that $\s_2(x,y)=\frac{x+y}{2}$ is the arithmetic mean. 
In particular, $\sigmaAver_2=1$, corresponding to the mass-action law for 
chemical reactions, is admissible.
\end{remark}

The next lemma shows that there is indeed for all $\alpha>0$ an activity function 
$\sigma_\alpha$ that satisfies Assumption~\ref{ass:Activity}.

\begin{lemma}[Existence of activity $\sigma_\alpha$]\label{lem:activ}~
\begin{enumerate}
 \item
	\label{lem:StolarskyInequality}
	For all $\alpha\geq 0$, $x,y\geq0$ we have the inequality $
	\s_{\alpha}(x,y)^{\alpha-1}\sqrt{xy}\, \leq \, \frac{x^{\alpha}+y^{\alpha}}{2} \leq \max\{x^\alpha,y^\alpha\}$.
\item The choices $\overline \sigma_\alpha(x,y)= \s_{\alpha}(x,y)^{2\alpha-2} 
\frac{2}{x^\alpha + y^\alpha}$ satisfy Assumption~\ref{ass:Activity}.
\end{enumerate}
\end{lemma}
\begin{proof}
	We prove the inequality of the first part by distinguishing two cases $\alpha\leq 1$ and $\alpha\geq 1$.
	
	First, let $\alpha\in]0,1]$. Because the Stolarsky means $\s_p$ are monotone increasing in their parameter $p\in \R$ and $\s_{-1}$ is the geometric mean~\cite{Stolarsky1975}, we have that 
	\[
	\s_{\alpha}(x,y)\geq \s_{-1}(x,y) = \sqrt{xy} \quad \Rightarrow \quad \s_{\alpha}(x,y)^{\alpha-1}\leq \sqrt{xy}^{\alpha-1}.
	\]
	Multiplying with $\sqrt{xy}$, we get with the AM-GM inequality the claim.
	
	Now, let $\alpha\geq 1$. Observe that the generalized Stolarsky mean\footnote{defined for $\alpha,\beta\in \R$ as $\s_{\alpha,\beta}(x,y) = \bra[\Big]{\frac{\beta(x^\alpha-y^\alpha)}{\alpha(x^\beta-y^\beta)}}^{\frac{1}{\alpha-\beta}}$ with the undefined cases understood as limits.} $\s_{r,2r}$ \cite[(10)]{Stolarsky1975} is just the power mean $\bigl(\frac{x^r+y^r}{2}\bigr)^{1/r}$. 
	Again, by using the monotonicity in both parameters of the generalized Stolarsky mean~\cite[Theorem p.~89]{Stolarsky1975} and using that $\alpha\geq 1$, we have
	\[
	\s_{\alpha}(x,y)=\s_{\alpha,1}(x,y)\leq \s_{\alpha,2\alpha}(x,y)  \ \Rightarrow \ \s_{\alpha}(x,y)^{\alpha-1}\leq \s_{\alpha,2\alpha}(x,y)^{\alpha-1} = \bigl(\frac{x^\alpha+y^\alpha}{2}\bigr)^{\frac{\alpha-1}{\alpha}} \,.
	\]
	Multiplying with $\sqrt{xy}$ and using that the power means are monotone, i.e. $\sqrt{xy}\leq \bigl(\frac{x^\alpha+y^\alpha}{2}\bigr)^{1/\alpha}$ for $\alpha\geq 1$ we get
	\[
	\s_{\alpha}(x,y)^{\alpha-1} \sqrt{xy} \leq  \Bigl(\frac{x^\alpha+y^\alpha}{2}\Bigr)^{\frac{\alpha-1}{\alpha}} \Bigl(\frac{x^\alpha+y^\alpha}{2}\Bigr)^{\frac{1}{\alpha}} = \frac{x^\alpha+y^\alpha}{2} \,.
	\]
	Finally, we observe that by continuity the inequality also holds for $\alpha=0$: both sides of the inequality are continuous in $\alpha$ at $\alpha=0$, and the inequality is strict for all $\alpha>0$ and $x\ne y$, so the limiting case $\alpha=0$ follows by taking $\alpha\downarrow0$.

    By the first part it follows immediately, that $\overline\sigma_\alpha$ satisfies \eqref{eq:def:sigma-Aver}. Hence, it suffices to show the properties of $\overline\jmath_\alpha$, where we now have 
    $\overline\jmath_\alpha(x,y) = \s_{\alpha}(x,y)^{2\alpha-2} \frac{2}{x^\alpha + y^\alpha} (x^2-y^2)$ for $x,y>0$. The only singular limits are for $x=y$, $x=0$, or $y=0$, where the latter can be excluded by symmetry. For $y=0$ and $x>0$, we easily compute the continuous extension $\overline\jmath_\alpha(x,0) = \frac{2}{\alpha^2} x^\alpha$. For $x=y$, we use the monotonicity of the Stolarsky mean, which provides $\min\{x,y\}\leq \s_\alpha(x,y) \leq \max\{x,y\}$. Hence, for both cases $\alpha\geq 1$ and $\alpha<1$, we compute the extension $\overline\jmath_\alpha(x,x)=0$. This shows, that $\overline\jmath_\alpha$ can be extended continuously on the whole $[0,\infty)^2$. The definiteness in Assumption~\ref{ass:item:sigma-Aver:definit} is also clear.
\end{proof}
Our first result regards the well-posedness of~\eqref{eq:RRE} under Assumption~\ref{ass:Activity}.
\begin{proposition}[Well-posedness of~\eqref{eq:RRE}]\label{prop:WellposednessODE}
		For $\alpha>0$, let $\sigma_\alpha$ satisfy Assumption \ref{ass:Activity}. Then for all $c^0\in \Prob^N$ there exists a differentiable global solution $c:[0,\infty)\to \Prob^N$ solving~\eqref{eq:RRE} such that $c(t)\in \Prob^N_{>0}$ for all $t>0$.
		For those solutions, the discrete entropy~\eqref{eq:def:energy} is a Lyapunov function.
\end{proposition}
\begin{proof}
	\noindent\underline{\emph{Step 1. Local existence of positive solutions}:} We start with a positive initial datum $c^0\in \Prob^N_{>0}$. By Assumption~\ref{ass:Activity}, the diffusive flux 
	\[
		J_{\alpha,k}[c]=\sigma_\alpha(c_{k-1},c_k,c_{k+1})N^2(c_k^2 - c_{k-1}c_{k+1}) = \bar\sigma_\alpha(c_k,\sqrt{c_{k-1}c_{k+1}}) N^2(c_k^2 - c_{k-1}c_{k+1})
	\]
	in~\eqref{eq:RRE} is continuous for $c\in \Prob_{>0}^N$ and hence we obtain for some $\tau>0$ a local solution $c: [0,\tau)\to \Prob^N_{>0}$ by the Peano existence theorem. We note that any such solution is mass conserving due to the presence of the discrete Laplacian in~\eqref{eq:RRE}. Hence, the constructed solution is also bounded by $N$. 
	
	\smallskip
	\noindent
	\underline{\emph{Step 2. Preservation of positivity}:}
	We now consider a local maximal positive solution on $[0,\tau)$. Since it is bounded, we find a sequence $t_n \to \tau$ such that the limit $\lim_{n\to \infty} c(t_n)=:c(\tau)\in \Prob^N$ exists. Suppose that $c(\tau)\in\Prob^N\setminus \Prob^N_{>0}$, then we find some $k\in [N]$ such that 
	\[
		c_k(\tau) = 0  , \quad \dot c_k(\tau) \leq 0 \quad \text{and}\quad c_{k+1}(\tau) > 0 \,
	\] 
	where we used mass conservation and continuity of the fluxes from Assumption~\ref{ass:Activity}  ~(A\ref{ass:item:sigma-Aver:continuous}).
	From the form of the flux and Assumption~\ref{ass:Activity}~(A\ref{ass:item:sigma-Aver:definit}), we conclude
	\[
		J_{\alpha,k}[c] \leq 0 ,\quad J_{\alpha,k-1}[c] \geq 0 \quad\text{and}\quad J_{\alpha,k+1}[c] > 0 \,.
	\]
	Hence, from the discrete Laplacian in~\eqref{eq:RRE}, we conclude
	\[
		\dot c_k = N^2 \bra[\big]{ J_{\alpha,k-1}[c] - 2J_{\alpha,k}[c]+ J_{\alpha,k+1}[c]} > 0 \,
	\]
	which provides a contradiction to $\dot c_k\leq 0$. Hence, we get $c(\tau)\in \Prob^N_{>0}$ and the constructed solution exists globally in time on $[0,\infty)$.
	
	\smallskip
	\noindent
	\underline{\emph{Step 3. Extension to $\Prob^N$ and generation of positivity}:}
	For $c^0\in \Prob^N$, we consider a sequence $c^{0,n}\in \Prob^N_{>0}$ such that $c^{0,n}\to c^0$. Consider the according solutions $c^n: [0,\infty)\to \Prob^N_{>0}$. 
	Again by the continuity of the diffusive flux from Assumption~\ref{ass:Activity}~(A\ref{ass:item:sigma-Aver:continuous}), we find a pointwise limit $c(t) := \lim_{n\to \infty} c^n(t)$ starting from $c^0$, which is differentiable satisfying~\eqref{eq:RRE}. 
	Suppose $c(t)\notin \Prob^N_{>0}$ for $t>0$, then continuity implies that there exists $k\in [N]$ such that $c_k(0)=0=\dot c_k(0)$ and without loss of generality $c_{k+1}(0)>0$, which by the same argument as in Step 2 provides a contradiction. 
	
	\smallskip
	\noindent
	\underline{\emph{Step 4. Lyapunov functional}:}
	By the positivity and differentiability of the constructed solutions, we can calculate for all $t>0$ using the integration by parts~\eqref{eq:def:IntByParts} and exponential rewriting of the diffusive flux~\eqref{eq:flux:rewrite:expand}
	\begin{align*}
		&\frac{\dx{}}{\dx{t}} E_{N}(c(t)) = \frac{1}{N} \sum_{k\in [N]} \log c_k(t)  \Delta^N J_{\alpha,k}(c(t)) \\
		&= - \frac{1}{N} \!\! \sum_{k\in [N]} \!\!\Delta^N \!\log c_k(t) \sigma_{\alpha}(c_{k-1}(t),c_{k}(t),c_{k+1}(t)) c_k(t)^2  N^2 \pra[\Big]{\exp\bra[\big]{ \tfrac{1}{N^2} \Delta^N\! \log c_k(t)} {-}1 } \!\leq\! 0
	\end{align*}
	where the last inequality follows from the fact that $\R \ni x\mapsto x\bra[\big]{\exp(x)-1} \geq 0$.
	Since $E_N: \Prob^N \to [0,\infty)$ is continuous, we conclude that $t \mapsto E_N(c(t))$ is decreasing on $[0,\infty)$.
\end{proof}

\subsection{Discrete gradient structure in continuity-equation format}%

We now describe the gradient structure for ~\eqref{eq:RRE} on $\Prob^N$ in more detail. For this we use gradient systems in continuity-equation format as in \cite{PeletierSchlichting2022,HeinzePietschmannHeinze2024arXiv,HeinzeMielkeStephan2025}. Note that here the situation simplifies because the spaces for the concentrations and for the fluxes have the same dimension, because we are in one dimension. The first ingredient is the abstract gradient map which is given by the discrete Laplace operator~$\Delta^N$. Secondly, the discrete energy is given in terms of the entropy~\eqref{eq:def:energy}.
By using the notation~\eqref{eq:def:VarDeriv}, we have the identity
\begin{equation*}
 (E_{N}'(c))_{k\in [N]} = \bigl( \log c_{k} \bigr)_{k\in [N]}.
\end{equation*}
The final element is the (dual) dissipation potential $R^*_{\alpha,N}$. For fixed $\alpha>0$, let  the activity $\sigma_\alpha:[0,\infty[^3\to [0,\infty[$ be given and recall the mobility $m_\alpha$ from \eqref{eq:def:m-alpha}:
\begin{equation*}
m_{\alpha}(c_{k-1},c_{k},c_{k+1}) :=\sigma_{\alpha}(c_k, c_{k-1}, c_{k+1})\cdot\bigl(c_{k}\sqrt{c_{k-1}c_{k+1}}\bigr),
\end{equation*}
we observe that provided $\sigma_\alpha$ is homogeneous of degree $\alpha-2$, then $m_\alpha$ is homogeneous of degree $\alpha$, justifying its name.
Now, we define for $c\in \Prob^N$  and any $\xi: [N]\to \R$ the dual dissipation potential by
\begin{equation*}%
R_{\alpha,N}^{*}(c,\xi)=\frac{1}{N}\sum_{k=1}^NN^{4}m_{\alpha}(c_{k-1},c_{k},c_{k+1})\CC^{*}\bigl(N^{-2}\xi_{k}\bigr),
\end{equation*}
where we have $\CC^{*}(r)=4\bigl(\cosh(r/2)-1\bigr)$.
Here, the superscript~$*$ refers to the Legendre duality. Since $\CC^*$ is convex, we get its convex primal function in terms of
\begin{equation*}%
	\CC(s) = \sup_{r\in \R} \bra[\big]{ rs - \CC^*(r)} \,.
\end{equation*}
Instead of providing the explicit form, it is more convenient to recall the basic identity
\begin{equation*}
	\CC^{*,\prime}(r) = 2 \sinh(r/2) \qquad\text{and}\qquad \CC'(s) = (\CC^{*,\prime})^{-1}(s) = 2\operatorname{arsinh}(s/2) \,,
\end{equation*}
that is $\CC$ is the primitive of $2\operatorname{arsinh}(s/2)$ with $\CC(0)=0$.
Besides $\CC$, we define its perspective function $\CCC: \R \times [0,\infty]$ defined by
\begin{equation}\label{eq:def:CCC}
    \CCC(s|w) = \begin{cases}
        w \CC\bra[\Big]{\frac{s}{w}} , & w>0 \,; \\
        0 , & s=0, w=0  \,; \\
        +\infty , & s\ne 0 , w=0 \,.
    \end{cases}
\end{equation}
The function $\CC$ and its perspective function $\CCC$ satisfy the following crucial estimate 
\begin{align}\label{eq:MagicalEstimate}
\forall\, s\in\R,w\in[0,\infty[\quad \forall\, q>1:\quad \CC(s)\leq\frac{q}{q-1}\CCC(s|w) + \frac{4 w^q}{q-1}
\end{align}
and the monotonicity property
\begin{align}\label{eq:Monotonicity}
    \forall\, s\in\R , \forall\, w\in[0,\infty[:\quad ]0,\infty[\ni\lambda\mapsto \CCC(\lambda s|\lambda ^2 w) \quad \text{is increasing},
\end{align}
For the proofs of \eqref{eq:MagicalEstimate} and \eqref{eq:Monotonicity}, we refer to \cite{HeinzeMielkeStephan2025}.

Fixing a time horizon $T>0$, and an initial datum $c_0\in \Prob^N_{>0} $ the gradient system with components $(\Prob^N,\Delta^N,E_N,R_{\alpha,N}^*)$
induces on $[0,T]$ the \emph{gradient-flow} equation
\begin{subequations}
 \label{eq:DiscreteGFE}
\begin{alignat}{2}
    \dot c & = \Delta^N J^N, && \qquad\text{ on } ]0,T] \label{eq:CEN}\\
    J^N & = \partial_\xi R_{\alpha,N}^*(c,-\Delta^NE_N'(c)), && \qquad\text{ on } ]0,T]  \label{eq:DiscreteConstitutiveRelation}\\
    c(0) & = c_0 .\nonumber %
\end{alignat}
\end{subequations}
The first equation \eqref{eq:CEN} is the second-order continuity equation involving 
the discrete Laplacian. Note that for all curves $c\in \AC([0,T];\Prob^N)$, thanks 
to the conservation of mass,  $\dot c$ has mean-zero, that is $\sum_{k\in [N]} 
\dot c_k=0$. Since $\Delta^N$ is invertible on mean-zero densities and hence for 
a.e.~$t\in [0,T]$  {there} exists a curve of fluxes $J^N(t):[N]\to \R$ such that
$\Delta^N J^N = \dot c$. In the following, the set of curves $(c^N,J^N)$ 
satisfying \eqref{eq:CEN} is denoted by $\CE_N$.

To make the equation \eqref{eq:DiscreteConstitutiveRelation} more explicit, we observe the identity
$\partial_{\xi_k}R^*_{\alpha,N} = N^2m_\alpha \CC^{*\prime}(N^{-2}\xi_k)$. Inserting $\xi_k=-(\Delta^NE_N'(c))_k$, using the logarithmic rule, i.e. $\CC^{*\prime}(2\log c_k-\log c_{k-1}-\log c_{k+1}) = \frac{c_k^2-c_{k-1}c_{k+1}}{c_k\sqrt{c_{k-1}c_{k+1}}}$, and the explicit choice of the mobility $m_\alpha$, we obtain
\[
J^N_k = N^2 \sigma_{\alpha}(c_{k-1},c_k,c_{k+1})(c_k^2-c_{k-1}c_{k+1}),
\]
which provides exactly the equation~\eqref{eq:RRE}.

\subsection{Variational characterization via the energy dissipation principle}\label{ssec:discrete:EDB}

One important tool between the gradient system and the  evolution equation  is the so-called \emph{energy-dissipation principle (EDP)}, which provides a variational formulation for the system~\eqref{eq:DiscreteGFE} in which we pass to the limit $N\to \infty$.
For this, we define by Legendre transform the primal dissipation potential $R_{\alpha,N}$, which for $c\in \Prob^N$ and $J:[N]\to \R$ is given by
\begin{align*}
R_{\alpha, N}(c,J)=& \frac{1}{N}\sum_{k=1}^NN^2\CCC\bigl(J_{k}|N^{2}m_{\alpha}(c_{k-1},c_{k},c_{k+1})\bigr) \\
= & \frac{1}{N}\sum_{k=1}^N\CCC\bigl(N^2J_{k}|N^{4}m_{\alpha}(c_{k-1},c_{k},c_{k+1})\bigr),
\end{align*}
where the so-called \textit{perspective function} $\CCC$ is defined in~\eqref{eq:def:CCC}. 
In addition, we define the \emph{discrete relaxed slope} for 
$c\in \Prob^N_{\geq 0}$ by 
\begin{equation}\label{eq:DiscreteSlope}
S_{\alpha, N}(c):=
R_{\alpha, N}^{*}(c,-\Delta^NE'_{N}(c)) 
=\frac{1}{N}\sum_{k=1}^N2 N^{4}\sigma_{\alpha}(c_{k-1},c_{k},c_{k+1})\bigl(c_{k}-\sqrt{c_{k-1}c_{k+1}}\bigr)^{2} \,,
\end{equation}
where we use for $a,b>0$ the identity
\[
	\sqrt{ab} \, \CC^{*}(\log a -\log b) = 2  \bra[\big]{\sqrt{a} - \sqrt{b}}^2 \,.
\]
Equation \eqref{eq:DiscreteSlope} and the rewriting in terms of $\overline\jmath_\alpha$ with Assumption~\ref{ass:Activity} thus shows, that the  definition of $S_{\alpha,N}:\Prob^N_{>0}\to \R$ can be continuously extended to $\Prob^N_{\geq 0}$, which is what we use from now on.

The primal dissipation functional and the slope-term  give rise to the total dissipation functional defined for curve of a pair of concentration $c\in \AC([0,T],\Prob^N)$ and flux $J^N:[0,T]\to\L^1(\mathbb{T}_N)$ by
\begin{equation}\label{eq:def:dissipation}
D_{\alpha, N}(c,J)  = \begin{cases}
\int_{0}^{T} \bigl( R_{\alpha, N}(c,J)+S_{\alpha, N}(c)\bigr) \d t &\text{ for } (c,J)\in\CE_N \,; \\
+\infty &\text{ for } (c,J)\not \in\CE_N \;.
\end{cases}\end{equation}
In total, we defined now all ingredients for the full energy-dissipation functional $L_{\alpha,N}$ in~\eqref{eq:def:ED-functional} and are now in the position to prove the second part of Result~\ref{result:discrete}.
\begin{proposition}\label{prop:GFDiscrete}
For $\alpha>0$, let $\sigma_\alpha$ satisfy Assumption~\ref{ass:Activity}. 
The solution $c\in\AC([0,T],\Prob^N)$ of equation~\eqref{eq:RRE} minimizes the full energy-dissipation functional, i.e. $L_{\alpha,N}(c,J) =0$.	
\end{proposition}
\begin{proof}
	By Proposition \ref{prop:WellposednessODE}, we know that $[0,T]\ni t \mapsto c(t)$ is a classical solution of \eqref{eq:RRE} and we have $c(t)_k>0$ for all $k\in[N]$ and $t>0$.
    In particular, we have $E_N'(c)_k=\log c_k$. The energy is differentiable, and we obtain on any subinterval $[s,t]\subset[0,T]$ that
    \[
    E_N(c(t))-E_N(c(s)) = \int_s^t\frac{\d}{\d r} E_N(c(r))\d r = \int_s^t   \frac{1}{N} \sum_{k\in [N]} \Delta^N \log c_k(r)   J_{\alpha,k}(c(r)) \d r,
    \]
    where we have used that $\Delta^N$ is self-adjoint. By 
    \eqref{eq:DiscreteConstitutiveRelation} we have that $J_{\alpha,N} = \partial_\xi 
    R_{\alpha,N}^*(c,-\Delta^N \log c)$, which provides by Young--Fenchel duality that
    \begin{align*}
    E_N(c(t))-E_N(c(s)) =& -\!\int_s^t \!  \tfrac{1}{N} \sum_{k\in [N]} \! 
    (-\Delta^N \log c_k(r))   J_{\alpha,k}(c(r)) \d r \\
    =& -\!\int_s^t \!\!
    \bigl(S_{\alpha,N}(c(r)) + R_{\alpha,N}(c(r),J(r))\bigr) \d r.
    \end{align*}
    Hence, $L_{\alpha,N}(c,J)=0$.
\end{proof}
\begin{remark}
    Although not needed in the main existence result, we briefly remark that also the converse statement of Proposition \ref{prop:GFDiscrete} holds true, i.e., curves  $(c,J)\in\CE_N$, $c\in\AC([0,T],\Prob^N)$ with $L_{\alpha,N}(c,J)=0$ are solutions of the evolution equation \eqref{eq:RRE} satisfying an entropy dissipation balance for $E_N$. For this, one uses that $\overline\sigma_\alpha(x,y)=0$ iff $x=y=0$ (by the lower bound) and Prop. 6.1 and equation (6.3) from \cite{HeinzeMielkeStephan2025}, which provides the identification of the flux $J$. 
\end{remark}

\section{EDP convergence to continuous system}
\label{sec:EDPconvergence}

\subsection{Continuous spaces and functionals}

The continuous space is $\mathbb{T}$, identified with $[0,1]$ and periodic 
boundary conditions. We consider the state space 
$\L^{1} (\mathbb{T})_{\geq0}$ with the usual dual pairing
by $\langle v,\xi\rangle=\int_{\mathbb{T}}v(x)\xi(x)\dx x$. We will also use the Bochner space on the time interval $[0,T]$, denoted by $\L^p([0,T],\L^q(\bbT))=:\L^p\L^q$, and similarly $\L^p([0,T],\W^{1,q}(\bbT))=:\L^p\W^{1,q}$.
The differential operators are denoted as usual by
\[
\partial_x: \W^{1,1}(\bbT)\to \L^1(\bbT), \qquad\partial_{xx}:\W^{2,1}(\bbT)\to\L^1(\bbT).
\]
For the gradient system in continuity-equation format $(X, \partial_{xx},{\cal E},{\cal R})$ 
with $X:=\L^{1}(\mathbb{T})_{\geq0}$, we consider the Laplacian $\partial_{xx}$ as the 
abstract gradient map. Moreover, we define the (quadratic) dual dissipation 
potential as
\begin{align}\label{eq:def:dual:dissipation:continuous}
  {\cal R}_\alpha^{*}(\rho,\eta) & =\frac{1}{2}\int_{\mathbb{T}} 
 \rho(x)^{\alpha }\eta(x)^{2}\dx x .
\end{align}
For a fixed time horizon $T>0$ the gradient system $(X,\partial_{xx},\cE,\cR^*_\alpha)$ induces on $[0,T]$ a \textit{gradient-flow equation}
\begin{subequations}%
    \begin{align}
        \dot \rho &= \partial_{xx} j \label{eq:CE}\\
        j &= \partial_\eta\cR^*_\alpha(\rho,-\partial_{xx}\D\cE(\rho)) \label{eq:ConstitutiveRelationContinuousGF}\\
        \rho(0) &= \rho_0. \nonumber
    \end{align}
\end{subequations}
The first equation \eqref{eq:CE} is the second-order continuity equation involving the Laplacian, which is understood in the weak-sense, i.e.
\begin{equation}\label{eq:CE:weak}
\forall\, \phi\in \mathrm{C}^2(\bbT): \qquad \frac{\d}{\d t} \int_\bbT\phi\, \rho\dx x = \int_\bbT \partial_{xx}\phi \, j \dx x.
\end{equation}
Note, that restricting the domain of the Laplacian to functions with average
zero, it becomes invertible and self-adjoint, i.e., 
\[
\partial_{xx}: H_{\#}^{2}(\bbT):=H^{2}(\bbT)\cap\Bigl\{ f\text{ periodic }:\int_{\mathbb{T}}f=0\Bigr\} \to\L^{2}(\bbT).
\]
Similarly, we use $\L^p_\#(\bbT)$, etc.\@ for other function spaces with average zero. 
We note that  {double primitive} $\partial_{xx}^{-1}:\L^p_\#(\bbT)\to\W^{2,p}_\#(\bbT)$ is bounded for all $p\in[1,\infty]$.
In particular, for a curve $\rho\in\W^{1,1}(0,T;\bbT)$ it follows that there exists for a.e. $t\in[0,T]$ a flux $j(t)\in \W^{2,1}(\bbT)$ such that $\dot \rho = \partial_{xx} j$. The set of curves $(\rho,j)$ is denoted by $\CE$.

For the second equation \eqref{eq:ConstitutiveRelationContinuousGF}, we consider positive concentrations $\rho$ and observe $\D\cE(\rho)=\log\rho$. Using $\partial_\eta\cR^*_\alpha(\rho,\eta)=\rho^\alpha \eta$, we derive formally the flux $j = - \rho^\alpha \partial_{xx}(\log\rho)$ which provides exactly the equation \eqref{eq:DLSS:alpha}. 
We highlight the two cases $\alpha=1$, corresponding to the usual DLSS equation $\dot{\rho}=-\partial_{xx}\bigl(\partial_{xx}\rho - \frac{(\partial_x \rho)^2}{\rho}\bigr)$, and $\alpha=2$, which originates on the discrete level from the mass-action law in chemical kinetics and the according space-continuous equation is of the form $\dot{\rho}=-\partial_{xx}\bigl(\rho\partial_{xx}\rho-(\partial_x \rho)^{2}\bigr)$.

\subsection{Variational characterization of \texorpdfstring{DLSS$_\alpha$}{DLSS with mobility}}

Analogously to the discrete setting, we now derive the energy-dissipation balance. The primal dissipation potential $\cR_{\alpha}$ is again quadratic and, by duality to the dual~\eqref{eq:def:dual:dissipation:continuous}, given by~\eqref{eq:def:primal:dissipation}.

The next lemma shows that the tentative definition of the slope given by the functional $\mathcal{S}_{\alpha,+}$ in~\eqref{eq:def:slope+} has a relaxed lower semicontinuous envelope for non-negative Sobolev functions, which we call the~\emph{relaxed slope} $\cS_{\alpha}$ and it is the main ingredient in our analysis.
\begin{lemma}[Relaxed slope]\label{lem:RelaxedSlope}
	For all smooth positive $\rho$ it holds that $\cS_{\alpha,+}(\rho) = \cS_{\alpha}(\rho)$. 
	Hereby, the relaxed slope $\cS_{\alpha}$ is equivalently defined for all non-negative  {densities} $\rho:\bbT\to[0,\infty[$ with $\rho^{\alpha/2}\in\W^{2,2}(\bbT)$ and $\rho^{\alpha/4}\in\W^{1,4}(\bbT)$ by
\begin{subequations}
 \label{eq:RelaxedSlopes}
  \begin{align}
   \label{eq:RelaxedSlope.a}
   \cS_\alpha(\rho) 
   :=& \frac{2}{\alpha^2} \int_\bbT \Bigl(\partial_{xx}\rho^{\alpha/2} 
    -4 (\partial_{x}\rho^{\alpha/4})^{2} \Bigr)^{2} \!\dx x
\\
    \label{eq:RelaxedSlope.b}
   =&  \frac{2}{\alpha^2} \int_\bbT \Bigl(\bigl(\partial_{xx} \rho^{\alpha/2}\bigr)^2 
        + \frac{16}{3}\bigl(\partial_x\rho^{\alpha/4}\bigr)^4 \Bigr) \dx x
\\
\label{eq:RelaxedSlope.BenamouBrenier}
 = &  \frac{2}{\alpha^2}
  \int_{\mathbb{T}}\frac{\bigl(\rho^{\alpha/2}\partial_{xx}\rho^{\alpha/2} 
  -\bigl(\partial_{x}\rho^{\alpha/2}\bigr)^{2}\bigr)^{2}}{\rho^\alpha}\dx x
\\
  \label{eq:RelaxedSlope.c} 
   =& \frac12 \int_\bbT  \Bigl(\frac{(\partial_{xx} \rho)^2}{\rho^{2-\alpha}} +
   \frac{2\alpha{-}3}{3}\, \frac{(\partial_x\rho)^4}{\rho^{4-\alpha}} \Bigr) \dx x \,.
    \end{align}
\end{subequations}    
\end{lemma}

\begin{remark}
	The several different forms of the slope will have 
	their different advantages in various situations: 
	\begin{itemize}
		\item[\eqref{eq:RelaxedSlope.a}] shows that the integrand of $\cS_\alpha$ equals 
	$\frac12 \Sigma^2$ with $\Sigma = -\frac2\alpha\bigl( 
	\partial_{xx}\rho^{\alpha/2} -4 (\partial_{x}\rho^{\alpha/4})^{2}\bigr)$ which
	plays a prominent role in the proof of the chain-rule (Proposition \ref{pr:ChainRule});
	\item[\eqref{eq:RelaxedSlope.b}] shows that $\cS_\alpha$ controls the $\L^2$ norms of 
	$\partial_{xx}\rho^{\alpha/2} $ and $(\partial_{x}\rho^{\alpha/4})^{2}$ justifying the assumption of the Lemma on $\rho$;
    \item[\eqref{eq:RelaxedSlope.BenamouBrenier}] characterizes the integrand of the slope $\cS_\alpha$  by a jointly convex function $u^2/v$, which is used in the proof of the liminf estimate;
	\item[\eqref{eq:RelaxedSlope.c}] shows that $\cS_\alpha$ is convex for the range $\alpha\in [3/2,2]$: the coefficient $\frac{2\alpha-3}{3}$ of the term $\frac{(\partial_x\rho)^4}{\rho^{4-\alpha}}$ is non-negative precisely for $\alpha\geq3/2$ and changes sign to negative for $\alpha<3/2$. For $\alpha<3/2$, the non-negativity of $\cS_\alpha$ follows instead from the manifestly non-negative form~\eqref{eq:RelaxedSlope.a}.
	\end{itemize}
\end{remark}

The proof of~\eqref{eq:RelaxedSlopes} relies on the classical chain rule for Sobolev functions and the integration-by-parts rule 
	\begin{equation}
		\label{eq:UsefulRelation}
		0=\int_\bbT \partial_x \bigl( u^\gamma \,(\partial_x u)^3\bigr) \dx x 
		=  \gamma \int_\bbT u^{\gamma-1} (\partial_x u)^4 \dx x + 3 \int_\bbT
		u^\gamma \,(\partial_x u)^2 \partial_{xx} u \dd x ,
	\end{equation}
	which holds for smooth and positive $u$ and all $\gamma \in \R$.
\begin{proof}
For a positive smooth function $\rho:\bbT\to ]0,\infty[$, a direct computations shows
\[
  \rho^{\alpha}\bigl(\partial_{xx}\log\rho\bigr)^{2} = \Bigl(\rho^{\alpha/2-1}
  \partial_{xx}\rho-\rho^{\alpha/2-2}\bigl(\partial_{x}\rho\bigr)^{2}\Bigr)^{2}.
\]
 Using that $\rho^{\alpha/2-1}\partial_{xx}\rho=\frac{2}{\alpha}\partial_{xx}\rho^{\alpha/2}
 +(1{-}\alpha/2)\rho^{\alpha/2-2}\bigl(\partial_{x}\rho\bigr)^{2}$ and
 $\partial_{x}\rho^{\alpha/4} =\frac{\alpha}{4}\rho^{\alpha/4-1}\partial_{x}\rho$,
 we obtain
  \begin{equation*}
   \rho^{\alpha/2-1}\partial_{xx}\rho - \rho^{\alpha/2-2}(\partial_x\rho)^2
   = \frac{2}{\alpha}\Bigl(\partial_{xx}\rho^{\alpha/2} - \bigl(2\partial_x\rho^{\alpha/4}\bigr)^2\Bigr),
 \end{equation*}
 where $\frac{2}{\alpha}$ is a global prefactor of the entire bracket. Hence
 \[\rho^{\alpha} \bigl(\partial_{xx}\log\rho\bigr)^{2}
 = \biggl(\frac{2}{\alpha} \Bigl(\partial_{xx}\rho^{\alpha/2} -  (2\partial_{x}
 \rho^{\alpha/4})^{2} \Bigr)\biggr)^{2},\] which proves the first formula~\eqref{eq:RelaxedSlope.a}.

Moreover, we have 
\[
  \biggl(\frac{2}{\alpha}\Bigl(\partial_{xx}\rho^{\alpha/2} -
     4(\partial_{x}\rho^{\alpha/4})^{2}\Bigr)\biggr)^{2}
  = \frac{4}{\alpha^2}\biggl((\partial_{xx}\rho^{\alpha/2})^2-8 \partial_{xx}\rho^{\alpha/2}
   (\partial_x\rho^{\alpha/4})^2 + 16 (\partial_x\rho^{\alpha/4})^4  \biggr).
\]
The integral of the middle term in the last expression can be simplified 
by using \eqref{eq:UsefulRelation} with $\gamma=1$ and $u=\rho^{\alpha/4}$ and 
exploiting the identity $\partial_{xx}(u^2) = 2u\partial_{xx} u + 2(\partial_x u)^2$, namely 
\begin{align*}%
 \int_\bbT \bigl(\partial_x \rho^{\alpha/4}\bigr)^2\partial_{xx} \rho^{\alpha/2} \dx x 
  =&  \int_\bbT (\partial_x u)^2 \partial_{xx} u^2  \dx x 
  = 2 \int_\bbT \bigl( u(\partial_x u)^2\partial_{xx} u + (\partial_x u)^4\bigr) \dx x \\
  =& \frac43 \int_\bbT \bigl(\partial_x\rho^{\alpha/4}\bigr)^4 \dx x. 
\end{align*}
With this, \eqref{eq:RelaxedSlope.b} is established. 

Using  
$\frac{(\partial_x\rho^{\alpha/2})^2}{\rho^{\alpha/2}} = 
4 (\partial_x\rho^{\alpha/4})^2$, we obtain \eqref{eq:RelaxedSlope.BenamouBrenier} from \eqref{eq:RelaxedSlope.a}. To obtain the last form~\eqref{eq:RelaxedSlope.c}, we re-express the integrand in \eqref{eq:RelaxedSlope.b} 
by powers of $\rho$, $\partial_x \rho$ and $\partial_x \rho$. From $(\partial_{xx} 
\rho^{\alpha/2})^2$ we obtain a mixed term $\rho^{\alpha-3}(\partial_x \rho)^2 
\partial_{xx} \rho$ which is integrated by parts via \eqref{eq:UsefulRelation} with 
$\gamma=\alpha{-}3$ and $u=\rho$. After some cancellations we obtain the desired 
form \eqref{eq:RelaxedSlope.c}. 

Thus, \eqref{eq:RelaxedSlope.a}--\eqref{eq:RelaxedSlope.c} hold for positive 
and smooth $\rho$. 
Since the right-hand sides can be continuously extended 
to all non-negative $\rho$ with $\rho^{\alpha/2}\in\W^{2,2}(\bbT)$ and 
$\rho^{\alpha/4}\in\W^{1,4}(\bbT)$, the claim follows.
\end{proof}

\begin{remark}[Critical spaces]\label{rem:CriticalSpace}
    Equation \eqref{eq:RelaxedSlope.b} shows that the slope is given as a sum of two terms involving a second and first order differential operator, respectively. However, the term $(\partial_x\rho^{\alpha/4})^4$ cannot be treated as a lower order term w.r.t. $(\partial_{xx}\rho^{\alpha/2})^2$ in the proof of the convergence result in Section~\ref{sec:EDPconvergence}. 
    Indeed, one can show that the mapping $v \mapsto w=\sqrt{v}$ is bounded and weakly continuous from 
	$\rmH^2_{\geq0}(\bbT)$ $:=$ $ 
	\bigl\{ v\in \rmH^2(\bbT)\,\big|\, v\geq 0 \text{ a.e. }\bigr\}$ into
	$\rmW^{1,4}(\bbT)$, but not compact. \\
	The boundedness follows by observing that
	from $\partial_{xx} v= \partial_{xx} w^2 = 2 w\partial_{xx} w + 2(\partial_x w)^2$ combined in
	\eqref{eq:UsefulRelation} with $\gamma=1$ and $u=w$, we obtain 
	\[
		\int_\bbT (\partial_x w)^4 \dd x = \frac34 \int_\bbT (\partial_x w)^2 \partial_{xx} v \dd x 
		\quad \text{implying} \quad \int_\bbT (\partial_x w)^4 \dd x
 		\leq \frac9{16} \int_\bbT (\partial_{xx} v)^2 \dx x \,.
	\]
	From this, we easily see that $v\in \rmH^2_{\geq 0}(\bbT)$ implies 
	$w=\sqrt{v}\in \rmW^{1,4}(\bbT)$ with $\| w\|_{\rmW^{1,4}}^4 \leq
	\| v\|_{\rmH^2}^2$. 

	For the non-compactness, we consider a sequence $v_m \rightharpoonup v$ in $\rmH^2_{\geq 0}(\bbT)$ 
	and define $w_m = \sqrt{v_m}$ and $w=\sqrt{v}$. Using $|\sqrt{a}{-}\sqrt{b}| 
	\leq \sqrt{|a{-}b|}$ we find $\| w_m{-}w\|_{\rmL^4}^4 \leq \| v_m{-}v 
	\|_{\rmL^2}^2 \to 0$ and conclude $ w_m \rightharpoonup w $ in $\rmW^{1,4}(\bbT)$. 
	Hence, to show that the mapping is not compact, we construct a weakly converging sequence 
	$(v_m)_m$ in $\rmH^2(\bbT)$ such that $w_m=\sqrt{v_m}$ does not converge strongly. 
	For $m\in \N$, we set 
	\[
		v_m(x) = \frac1{m^{1/2}} | \sin(2\pi x)|^{3/2+2/m} \quad \text{and} 
		\quad w_m(x) = \sqrt{v_m(x)} = \frac1{m^{1/4}} | \sin(2\pi x)|^{3/4+1/m} 
	\]
	We easily obtain $\| w_m\|_{\rmL^4}^4 = \| v_m\|_{\rmL^2}^2 \leq C/m$. 
	Moreover, an explicit calculation gives 
	\[
		\| \partial_x w_m\|_{\rmL^4}^4 \to C_1 >0 \quad \text{and} \quad 
		\| \partial_{xx} v_m\|_{\rmL^2}^2 \to C_2 >0.
	\]
	Thus, $w_m\rightharpoonup 0$ in $\rmW^{1,4}(\bbT)$ and $ v_m \rightharpoonup  0$ in $\rmH^2(\bbT)$, but strong convergence fails in both cases.   
\end{remark}

With the relaxed slope and the primal dissipation potential we now define the total dissipation functional for a curve of a pair of density $\rho$ and flux $j$ by
\begin{align}\label{eq:ContinunousDissipationFunctional}
{\cal D}_\alpha(\rho, j):=\begin{cases}
\int_{0}^{T}\bigl({\cal R}_\alpha(\rho,j)+\cS_\alpha(\rho)\bigr)\dx t & \text{ for } (\rho,j)\in\CE,\\
+\infty & \text{ for } (\rho,j)\not\in\CE,
\end{cases}
\end{align}
where $\CE$ denotes the set of solutions to the second-order continuity equation~\eqref{eq:CE:weak}.
In this way, we have introduced all ingredients to define energy dissipation functional~\eqref{eq:EDfunctional} to formulate our notions of solutions. 
For this, we distinguish curves that satisfy the \emph{energy-dissipation inequality} (EDI) and the \emph{energy-dissipation balance} (EDB).

\begin{definition}%
	A curve $\rho:[0,T]\to\L^1(\bbT)$ is an EDI solution of \eqref{eq:DLSS:alpha} if $\sup_{t\in[0,T]}\mathcal{E}(\rho(t))<\infty $ and if there exists a curve of fluxes $j:[0,T]\to\L^1(\bbT)$ such that $(\rho,j)\in\CE$ and $\mathcal{L}_\alpha(\rho,j)\leq 0$. If, in addition, $(\rho,j)$ satisfies for all $[s,t]\subset[0,T]$ the balance equation
	\[
	\cE(\rho(t))- \cE(\rho(s)) + \int_s^t \bigl( \cR_\alpha(\rho(r),j(r)) + \cS_\alpha(\rho(r))\bigr)\dx r = 0,
	\]
	then $\rho:[0,T]\to\L^1(\bbT)$ is called EDB solution.
\end{definition}
 {The EDI is the natural outcome of the EDP convergence limit (see Result~\ref{result:EDP} via lower semicontinuity of $\cE$ and liminf estimates for $\cR_\alpha$ and $\cS_\alpha$ proven in Section~\ref{s:ConvergenceProof}); the EDB additionally requires the chain-rule of Proposition~\ref{pr:ChainRule} to close the reverse inequality $\mathcal{L}_\alpha(\rho,j)\geq0$.}

\subsection{Embeddings}\label{ss:Embeddings}

The convergence result of the discrete energy-dissipation functional $L_{\alpha,N}$ towards its continuous analog $\mathcal{L}_{\alpha}$ will be formulated with the help of suitable embeddings of the discrete quantities into the continuum.

\subsubsection*{Embedding of densities}
A discrete concentration vector $c\in\L^{1}(\mathbb{T}_{N})$ is turned into 
a density on $\mathbb{T}=[0,1]_{0\sim1}$ by a piecewise constant interpolation $\iota : \L^1(\mathbb{T}_N)\to \L^1(\mathbb{T})$ defined by
\begin{equation}\label{eq:def:EmbeddingDensity}
\bra[\big]{\iota_{N}c^{N}}(x):=\rho^{N}(x):=\sum_{k= {1}}^{N}c_{k}^{N}\one_{[ {(k-1)/N},k/N[}(x) \,.
\end{equation}
The dual embedding is an integral average given by
\[
\bigl(\iota_{N}^*\phi\bigr)_{k}=N\int_{ {(k-1)/N}}^{k/N}\phi(x)\dx x \,.
\]
Indeed, we have the identity
$
\langle\iota_{N}c^{N},\phi\rangle=\frac{1}{N}\sum_{k=1}^Nc_{k}^{N}\  N\int_{ {(k-1)/N}}^{k/N}\phi\dx x=\langle c^{N},\iota_{N}^{*}\phi\rangle_{N}$.
In particular, we have for the embedded density from~\eqref{eq:def:EmbeddingDensity} the identity
\[
\int_{\mathbb{T}}\rho^{N}\dx x=\sum_{k=1}^N\int_{ {(k-1)/N}}^{k/N}c_{k}^{N}\dx x=\frac{1}{N}\sum_{k}c_{k}^{N}.
\]
For the discrete spaces we use for $p\geq1$ the following norms
\[
\|c^{N}\|_{\L^{p}_N}:=\biggl(\frac{1}{N}\sum_{k=1}^N|c_{k}^{N}|^{p}\biggr)^{1/p}
\]
and observe that the embedding $\iota_{N}:(\L^{1}(\mathbb{T}_{N}),\|\cdot\|_{\L^p_N})\to(\L^{1}(\mathbb{T}),\|\cdot\|_{\L^{p}(\mathbb{T})})$
is norm-preserving
\begin{equation*}%
\|\iota_{N}c\|_{\L^{p}(\mathbb{T})}^{p}=\int_{\mathbb{T}}|\iota_{N}c(x)|^{p}\dx x = \frac{1}{N} \sum_{k=1}^N|c_{k}|^{p}=\|c\|_{\L^p_N}^{p} \,.
\end{equation*}
\subsubsection*{Embedding of fluxes}
Using the discrete Laplacian $\Delta^N$ and the usual Laplacian $\partial_{xx}$, we define for the fluxes the following embedding operator 
\begin{equation}\label{eq:def:FluxEmbedding}
{\cal I}_{N}J^{N}(x):=\bigl(\partial_{xx}^{-1}\iota_{N}\Delta^{N}J^{N}\bigr)(x).
\end{equation}
A simple calculation shows that this embedding is consistent with both continuity equations, the discrete~\eqref{eq:CEN} and the continuous~\eqref{eq:CE}. Indeed, we have the equivalent formulations
\begin{align}\label{eq:ContinuityEquations}
\bigl(\iota_{N}{c}^{N},{\cal I}_{N}J^{N}\bigr)\in\CE  \quad \Leftrightarrow  \quad \bigl({c}^{N},J^{N}\bigr)\in\CE_{N},
\end{align}
because a simple computation shows for all $\phi\in{\cal D}(\mathbb{T})$ and $\xi=\iota_{N}^{*}\phi\in{\cal D}(\mathbb{T}_N)$ that
\begin{alignat*}{2}
\bigl(\iota_{N}{c}^{N},{\cal I}_{N}J^{N}\bigr)\in\CE & \quad \Leftrightarrow\quad \forall\,\phi\in{\cal D}(\mathbb{T}):& \langle\iota_{N}\dot{c}^{N},\phi\rangle&=\!\langle{\cal I}_{N}J^{N},\partial_{xx}\phi\rangle =\!\langle\iota_{N}\Delta^{N}J^{N},\phi\rangle\\
\text{and }\, \bigl({c}^{N},J^{N}\bigr)\in\CE_{N}&\quad \Leftrightarrow\quad \forall\,\xi\in{\cal D}(\mathbb{T}_{N}):& \langle\dot{c}^{N},\xi\rangle_{N}&=\!\langle\Delta^{N}J^{N},\xi\rangle_{N}.
\end{alignat*}
Moreover, we introduce the discrete analogue of the classical (homogeneous) Sobolev spaces by the norms
\[
\|\phi\|_{\W_{N}^{2,p}}^{p}:=\frac{1}{N}\sum_{k}\bigl\lvert(\Delta^{N}\phi)_{k}\bigr\rvert^{p}=\|\Delta^{N}\phi\|_{\L_{N}^{p}}^{p},
\]
In particular, we have that 
\begin{equation*}%
	\Delta^{N}:\W_{N}^{2,p}\to\L_{N}^{p}  \quad \text{ is an isomorphism}.
\end{equation*}
Since $\iota_{N}:\L_{N}^{p}\to\L^{p}$ is bounded, and $\partial_{xx}^{-1}:\L^{p}\to\W^{2,p}$
is bounded, we conclude for the flux embedding in~\eqref{eq:def:FluxEmbedding} that
$ {\cal I}_N = \partial_{xx}^{-1}\iota_{N}\Delta^{N}:\W_{N}^{2,p}\to\W^{2,p}$ is also bounded.
\begin{lemma}[Flux embedding]\label{lem:FluxEmbedding}
 {For} all discrete flux $J^N\in \L^1_N$ it holds that
$\|{\cal I}_{N}J^{N}\|_{\L^{1}(\mathbb{T})}\leq2\|J^{N}\|_{\L^1_N}$. 
If the sequence $(J^N)_{N\in\N}$ is uniformly bounded, i.e. $\sup_N\|J^N\|_{\L^1_N}<\infty$, then there exists a subsequence such that both embedded fluxes $\iota_NJ^N\in\L^1$ and $\mathcal{I}_NJ^N\in\L^1$ converge as finite measures in $\mathcal{M}(\bbT)$ and their limits coincide.
\end{lemma}
\begin{proof}
We first derive the estimate for the fluxes.
Let us fix $N\in\N$. First, take $\phi\in\W^{2,\infty}(\bbT)\subset C^{1,\gamma}(\bbT)$
for all $\gamma\in[0,1[$. 
By the mean-value theorem, for all $x\in[0,1]$
there are $x_{-},x_{+}\in[0,\frac{1}{N}]$ such that 
\[
\phi(x-N^{-1})-\phi(x)=-N^{-1}\phi'(x-x_{-}),\quad\phi(x)-\phi(x+N^{-1})=-N^{-1}\phi'(x+x_{+}).
\]
Hence, we have that 
\[
\phi(x-N^{-1})-2\phi(x)+\phi(x+N^{-1})=\frac{1}{N}\bigl( \phi'(x+x_{+})-\phi'(x-x_{-})\bigr) =\frac{1}{N}\int_{-x_{-}}^{x_{+}}\phi''(x+t)\dx t,
\]
where we have used that $\phi'$ is absolutely continuous. In particular,
we have for all $x\in[0,1]$ that 
\[
\bigl\lvert\phi(x-N^{-1})-2\phi(x)+\phi(x+N^{-1})\bigr\rvert \leq\frac{2}{N^{2}}\|\phi''\|_{\infty}.
\]
Hence, we get for all $k$ that 
\begin{align*}
\bigl\lvert\bigl(\Delta^{N}\iota_{N}^{*}\phi\bigr)_{k}\bigr| & =N\biggl\lvert \int_{k/N}^{(k+1)/N}N^2\bigl(\phi(x-N^{-1})-2\phi(x)+\phi(x+N^{-1})\bigl) \dx x \biggr\rvert \leq2\|\phi''\|_{\infty},
\end{align*}
which implies that
\begin{align*}
\langle{\cal I}_{N}J^{N},\partial_{xx}\phi\rangle=&\langle\iota_{N}\Delta^{N}J^{N},\phi\rangle=\langle J^{N},\Delta^{N}\iota_{N}^{*}\phi\rangle_{N}\\
=&\frac{1}{N}\sum_{k}J_{k}^{N}\bigl(\Delta^{N}\iota_{N}^{*}\phi\bigr)_{k} \leq 2 \|\phi''\|_{\infty}\ \frac{1}{N}\sum_{k}|J_{k}^{N}|.
\end{align*}
Using that $\partial_{xx}^{-1}:\L^\infty(\bbT)\to\W^{2,\infty}(\bbT)$ is an isomorphism, we hence obtain
\begin{align*}
\|{\cal I}_{N}J^{N}\|_{\L^{1}(\mathbb{T})}  =&\sup_{f\in\L^{\infty}(\mathbb{T}):\|f\|_{\L^\infty\leq 1}}\langle{\cal I}_{N}J^{N},f\rangle \\
= &\sup_{\phi\in\W^{2,\infty}(\mathbb{T}):\|\phi\|_{\W^{2,\infty}}\leq1}\langle{\cal I}_{N}J^{N},\partial_{xx}\phi\rangle\leq 2 \|J^{N}\|_{\L^1_N} \,.
\end{align*}
This shows the desired estimate.

Assume now that $\sup_N\|J^N\|_{\L^1_N}<\infty$. Then, the induced measures $\iota_NJ^N$ and $\mathcal{I}_NJ^N$ are bounded in the space of finite measures. Hence, by the Banach--Alaoglu theorem, there exists a subsequence that converges in $\mathcal{M}(\bbT)$. 
To identify the limit, let $\xi\in C^{\infty}(\mathbb{T})$ be fixed.
Similar to the expansion in the beginning of the proof, we can do a third-order Taylor expansion around $x\in[0,1]$ to find
$x_{-}\in[x-\frac{1}{N},x]$ and $x_{+}\in[x,x+\frac{1}{N}]$ such
that 
\[
\xi\bigl(x-\tfrac{1}{N}\bigr)-2\xi(x)+\xi\bigl(x+\tfrac{1}{N}\bigr)
=\tfrac{1}{N^{2}}\xi''(x)+\tfrac{1}{6N^{3}}\bigl( \xi'''(x_{+})-\xi'''(x_{-})\bigr) \,,
\]
which implies that 
\[
(\iota_N^*\partial_{xx} \xi)_k - (\Delta^N\iota_N^*\xi)_k = N \int_{k/N}^{(k+1)/N} \frac{1}{6N}(\xi'''(x_-)-\xi'''(x_+))\d x.
\]
In particular, we get the following commutator estimate
\begin{align}\label{eq:CommutatorEstimate}
\bigl|(\iota_{N}^{*}\partial_{xx}\xi)_{k}-\bigl(\Delta^{N}\iota_{N}^{*}\xi\bigr)_{k}\bigr|
\leq \frac{1}{3N}\|\xi'''\|_{\L^\infty} \,,
\end{align}
which allows us to identify the limit, because of
\begin{align*}
|\langle \iota_NJ^N,\partial_{xx}\xi\rangle &- \langle \mathcal{I}_NJ^N,\partial_{xx}\xi\rangle| = |\langle J^N,\iota_N^*\partial_{xx}\xi -\Delta^N\iota_N^*\xi\rangle_N| \\
&\leq \frac{1}{N}\sum_{k=1}^N |J^N_k| \frac{1}{3N}\|\xi'''\|_{\L^\infty} \to 0 \,. %
\end{align*}
\end{proof}

\subsection{EDP convergence}\label{s:ConvergenceProof}

We are now in the position to state the detailed version of Result~\ref{result:EDP}.
\begin{theorem}\label{thm:ConvergenceResult}
	Assume the  {activity} $\sigma_\alpha$ satisfies Assumption~\ref{ass:Activity}.
	Let $(c^{N},J^{N})\in \CE_N$ be solutions of the discrete continuity equation with $\sup_{N} L_{\alpha,N}(c^N,J^N)<\infty$ and $\sup_N E_N(c^N(0))<\infty$. 
	Let $\rho^{N}:=\iota_{N}c^{N},j^{N}:={\cal I}_{N}J^{N}$
	be their embeddings.
	
	Then there exists an absolutely continuous curve
	of densities $\rho\in\W^{1,1}([0,T],\L^{1}(\mathbb{T}))$ and fluxes
	$j\in\L^{1}([0,T],\L^{1}(\mathbb{T}))$ such that $({\rho},j)\in\CE$,
	$\rho^{N}\rightharpoonup\rho\in\W^{1,1}([0,T],\L^{1}(\mathbb{T}))$
	and $j^{N}\wstarlim  j$ in $\mathcal{M}([0,T]\times \mathbb{T})$ up
	to a subsequence.
	Moreover, it holds
	\begin{align*}
		\forall\, t>0:\liminf_{N\to\infty}E_{N}(c^{N}(t))  \geq{\cal E}(\rho(t))
		\qquad\text{and}\qquad 
		\liminf_{N\to\infty}D_{\alpha,N}(c^{N},J^{N})  \geq{\cal D}_\alpha(\rho,j).
	\end{align*}
	If, in addition the curves $c^N$ are EDB solutions to~\eqref{eq:RRE} and have well-prepared initial data, that is $\iota_Nc^N(0)\to\rho(0)$ and $E_N(c^N(0))\to\cE(\rho(0))$, then the limit curve $(\rho,j)$ is an EDI solution to~\eqref{eq:DLSS:alpha}.
\end{theorem}
In the rest of the section,
we prove Theorem \ref{thm:ConvergenceResult} by deriving first strong compactness for the embedded curves $\rho^N$ and showing then the liminf estimates.

\subsubsection*{Proof of EDP convergence: compactness}

The aim of the section is to derive enough compactness for the family of solutions the discrete gradient-flow equation in order to prove Theorem \ref{thm:ConvergenceResult} and so Result~\ref{result:EDP}. Let $\alpha>0$ and an activity~$\sigma_\alpha$ be fixed such that Assumption~\ref{ass:Activity} is satisfied. Throughout this section, we fix a time horizon~$T>0$. 

Let $[0,T]\ni t\mapsto c^{N}(t)\in\L^1_N(\mathbb{T})$ and $[0,T]\ni t\mapsto J^{N}(t)$
be two trajectories with $(c^N,J^N)\in\CE_N$ such that 
\begin{align}\label{eq:AprioriBounds}
\sup_{N\in\N}D_{\alpha,N}(c^{N},J^{N})\leq C_{\mathrm{diss}},\quad\sup_{N\in\N}\sup_{t\in[0,T]}E_{N}(c^{N}(t))\leq C_{\mathrm{en}}.    
\end{align}
In particular, we have from the definition of $D_{\alpha,N}$ in~\eqref{eq:def:dissipation} the two priori bounds
\begin{align*}
\int_{0}^{T}\frac{1}{N}\sum_{k=1}^NN^2\CCC\bigl(J_{k}^{N}|N^2m_{\alpha}(c_{k-1},c_{k},c_{k+1})\bigr)\dx t & \leq C_{\mathrm{diss}} \;; \\
\int_{0}^{T}\frac{1}{N}\sum_{k=1}^NN^{4}\sigma_\alpha(c_{k-1},c_{k},c_{k+1})\bigl(c_{k}-\sqrt{c_{k-1}c_{k+1}}\bigr)^{2}\dx t & \leq C_{\mathrm{diss}} \;.
\end{align*}
Using the embedding operators from Section \ref{ss:Embeddings}, we define pointwise in time the embedded sequences 
\[
\bigl(\rho^{N}(t):=\iota_{N}c^{N}(t)\bigr)_{N\in\N},\quad\bigl(j^{N}(t):={\cal I}_{N}J^{N}(t)\bigr)_{N\in\N}.
\]
To derive strong relative compactness of $(\rho^N)_{N\in\N}$ in $\L^1(0,T;\bbT)$ we rely on an Aubin-Lions argument. For this, we first derive spatial regularity.

\begin{lemma}[Discrete spatial regularity]
\label{lem:SpatReg} 
For all $c\in\L^1(\bbT_N)$ it holds
\[
S_{\alpha,N}(c) \geq \frac{1}{24\alpha^2} \ \frac1N \sum_{k=1}^N\biggl( \bigl(\Delta^{N}c^{\alpha/2}\bigr)_{k}^{2} +  \bigl(\partial_-^N c^{\alpha/4}\bigr)_{k}^{4} + \bigl(\partial_+^N c^{\alpha/4}\bigr)_{k}^{4} \biggr).
\]
\end{lemma}
\begin{proof}
The assumption on $\sigmaAver_\alpha$ from  \eqref{eq:def:sigma-Aver} allows us to obtain a lower bound on the slope. Indeed, we have
\begin{align*}%
    S_{\alpha, N}(c) 
 \geq \frac{1}{N}\sum_{k=1}^N  \frac{N^{4}}{\alpha^2} \frac{(c_{k}^\alpha-\sqrt{c_{k-1}c_{k+1}}^\alpha)^2}{\max\{c_{k}^\alpha,\sqrt{c_{k-1}c_{k+1}}^\alpha\}} \geq \frac{1}{N}\sum_{k=1}^N  \frac{N^{4}}{\alpha^2} \bigl(c_{k}^{\alpha/2}-\sqrt{c_{k-1}c_{k+1}}^{\alpha/2}\bigr)^{\!2} \,,
\end{align*}
where in the second inequality, we have used that $\frac{\bigl(c_k^{\alpha/2} + \sqrt{c_{k-1}^{\alpha/2} c_{k+1}^{\alpha/2}}\bigr)^2}{\max\{c_k^\alpha,\sqrt{c_{k-1}^\alpha c_{k+1}^\alpha}\}}\geq 1$.
Introducing $w_{k}:=c_{k}^{\alpha/2}$, we need to show that 
\[
\sum_{k=1}^NN^4\bigl(w_{k}-\sqrt{w_{k-1}w_{k+1}}\bigr)^{2}\geq \frac{1}{24}\sum_{k=1}^N\biggl((\Delta^{N}w)_{k}^{2} + \bigl(\partial_-^N c^{\alpha/4}\bigr)_{k}^{4} + \bigl(\partial_+^N c^{\alpha/4}\bigr)_{k}^{4}  \biggr).
\]
We use discrete partial integration to observe that
\[
\sum_{k=1}^N \pra[\big]{ (\sqrt{w_{k+1}}-\sqrt{w_{k}})^{3}(\sqrt{w_{k+1}}+\sqrt{w_{k}})-(\sqrt{w_{k}}-\sqrt{w_{k-1}})^{3}(\sqrt{w_{k}}+\sqrt{w_{k-1}})} =0.
\]
Adding this to the left-hand side, writing $u_{k}=\sqrt{w_{k}}$, it hence suffices to show that for all $k\in [N] $ we have 
\begin{multline*}
(u_{k}^{2}-u_{k+1}u_{k-1})^{2}+\frac{1}{4}\bigl\{ (u_{k+1}-u_{k})^{3}(u_{k+1}+u_{k})
-(u_{k}-u_{k-1})^{3}(u_{k-1}+u_{k})\bigr\}\\
\geq \frac{1}{24}\biggl((u_{k-1}^{2}-2u_{k}^{2}+u_{k+1}^{2})^{2}+(u_{k-1}
     -u_{k})^{4}+(u_{k+1}-u_k)^{4}\biggr).
\end{multline*}

First, we use that both sides of the inequality are 1-homogeneous and so we may 
set $u_{k}=1$. Moreover, introducing the parametrization  $u_{k-1}=x+\sqrt{t}$ 
and $u_{k+1}=x-\sqrt{t}$  with $x\geq0$ and $t\in[0,x^{2}]$, putting all terms on 
one side, and rearranging, a straight-forward computation shows that it suffices 
to show that 
\[
\forall\, x\geq 0\ \forall\, t\in[0,x^2]: \ \ S(x,t):=15 t^2 + 2 t (x{-}11) 
 (x{-}1 ) + (x{-}1)^2 (3 {+} 22 x {+} 15 x^2) \geq 0.
\]

We observe $S(x,0)\geq 0$ for all $x\geq 0$. 
For fixed $x>0$ we minimize over
$t\in [0,x^2]$ and obtain $S(x,t)\geq S(x,0)$ for $x \in [0,1]\cup [11,\infty)$,
as the prefactor of the term for $t^1$ is non-negative for these $x$ {: indeed,
for $x\in[0,1]$ both factors $(x{-}11)$ and $(x{-}1)$ are non-positive, while for
$x\in[11,\infty)$ both are non-negative; hence $2(x{-}11)(x{-}1)\geq0$ in both
cases, and the minimum of the quadratic $t\mapsto S(x,t)$ over $t\geq0$ is therefore
attained at $t=0$}. 

For $x\in (1,11)$ we use that $t\mapsto S(x,t)$ is quadratic,  hence
its minimum is given at $t_{x}=\frac{1}{15}(11{-}x)(x{-}1) \in [0,x^2]$.
Inserting the critical value into $S$ we find
\[
S(x,t_{x})=\frac{4}{15}(x-1)^{2}(56x^{2}+88x-19)\geq0,
\]
for $x\geq1$. Hence, we conclude that $S(x,t)\geq0$, which proves the inequality. 
\end{proof}
The discrete spatial regularity has important consequences for the embedded sequence $(\rho^N)_{N\in\N}$ as it provides a uniform spatial $\L^\infty$-bound and spatial regularity. We summarize static properties next, which are later also needed for time-integrated bounds.

\begin{corollary}\label{cor:StaticBounds}
Let $(c^N)_{N\in\N}$ satisfy $\sup_{N\in\N}\bigl\{S_{\alpha,N}(c^N)+ \|c^N\|_{\L^\alpha_N}\bigr\} <\infty$. Then:

(A) The sequence $\rho^N=\iota_Nc^N$ is uniformly bounded in $\L^\infty(\bbT)$;

(B) The sequence $ \iota_N\partial_\pm^N(c^{N})^{\alpha/2}$ is compact in $\L^2(\bbT)$;

(C) We have $\|(\rho^N)^{\alpha/4}\|_{\mathrm{BV}(\bbT)}\leq \|\partial_-^N(c^N)^{\alpha/4}\|_{\L^4_N}$.
\end{corollary}

\begin{proof}
    Let $u^N:=(c^N)^{\alpha/4}$ and $w^N:=(c^N)^{\alpha/2}$. The bound on $c^N\in\L^\alpha_N$ implies the bounds $w^N\in\L^2_N$ and $u_N\in\L^4_N$. From the assumptions in combination with Lemma \ref{lem:SpatReg}, we conclude that 
    \[
    \sup_{N\in\N}\|u^N\|_{\W^{1,4}_N}=\sup_{N\in\N}\bigl\{\|\iota_N\partial_\pm u^N\|_{\L^4} + \|\iota_Nu^N\|_{\L^4}\bigr\}<\infty, %
    \]
    which implies that all sequences $\iota_Nu^N,\iota_Nw^N,\rho^N$ are uniformly bounded in $\L^\infty(\bbT)$ by $\W^{1,4}_N\subset\L^\infty_N$, the discrete analog of the continuous embedding $\W^{1,4}(\bbT)\subset\L^\infty(\bbT)$. This proves (A).
    
    The bound on the slope also provides the uniform bound $\iota_N\Delta^N w^N\in\L^2$. Since $\partial_\pm^Nw^N$ has mean-average zero, we conclude by a discrete Poincaré inequality the uniform bound $\iota_N\partial_\pm^Nw^N\in\L^2$. Compactness of that sequence in $\L^2$ now follows from a discrete Rellich theorem.
    Indeed, since $(\partial_+^Nw^N)_k=(\partial_-^Nw^N)_{k+1}$, it suffices to consider $v^N:=\partial_-^Nw^N$.
    By the Arzelà--Ascoli theorem, in its version for asymptotically equicontinuous sequences, it suffices to show that $\omega^N=\iota_Nv^N$ satisfies $\lvert\omega^N(x)-\omega^N(y)\rvert\leq C(|x-y|+1/N)^{1/2}$ if $\sup_N\|\Delta^N w^N\|_{\L^2_N}\leq C$.
    To show this, we let $x,y$ with $x\leq y$ such 
    that $x\in[(k-1)/N,k/N[$, $y\in[(l-1)/N,l/N[$ and estimate
    \[
    \bigl\lvert\omega^N(x)-\omega^N(y)\bigr\rvert
   	= |v^N_k-v^N_l| \leq \sum_{j=k}^{l-1}|v^N_j-v^N_{j+1}| \,.
   	\]
    By the Cauchy--Schwarz inequality, introducing factors of $\sqrt{N}$, and using $N(v^N_{j+1}-v^N_j)=(\Delta^Nw^N)_j$, we conclude
    \begin{align*}
    	\bigl\lvert\omega^N(x)-\omega^N(y)\bigr\rvert
    	&\leq \biggl(\sum_{j=k}^{l-1}N|v^N_j-v^N_{j+1}|^2 \biggr)^{\!1/2} \biggl(\sum_{j=k}^{l-1}\frac 1 N \biggr)^{\!1/2}\\
    	& \leq \lVert\Delta^Nw^N\rVert_{\L^2_N} \biggl(\frac{l-k}{N}\biggr)^{1/2}\leq\|\Delta^Nw^N\|_{\L^2_N}\Bigl(|x-y|+\frac1N\Bigr)^{1/2},
    \end{align*}
    which proves (B).

    For (C), we recall that the BV-norm is defined in duality with functions $\phi\in C^1(\bbT)$ with $\|\phi\|_{\L^\infty(\bbT)}\leq 1$. We observe that for all $h\in]0,1/N[$ we have
\begin{align}\label{eq:31}
    \int_{\bbT} \iota_N u(x) \frac{\phi(x+h)-\phi(x)}{h}\dx x \leq \frac 1 h \|\phi\|_{\L^\infty(\bbT)} \int_{\bbT} |\iota_Nu(x-h)-\iota_Nu(x)|\dx x. 
\end{align}
Since 
\begin{align*}
|\iota_Nu(x-h)-\iota_Nu(x)| &= \biggl|\sum_{k=1}^Nu^N_k\one_{[(k-1)/N,k/N[}(x-h)-u^N_k\one_{[(k-1)/N,k/N[}(x)\biggr| \\
&\leq \sum_{k=1}^N|u^N_k-u^N_{k-1}|\one_{[(k-1)/N,(k-1)/N + h[}(x) \,,
\end{align*}
we get from \eqref{eq:31} that
\begin{align*}
\MoveEqLeft\int_{\bbT} \iota_N u^N(x)\frac{\phi(x+h)-\phi(x)}{h}\dx x \leq  \frac 1 h \|\phi\|_{\L^\infty(\bbT)} \int_{\bbT} \sum_{k=1}^N|u^N_k-u^N_{k-1}|\one_{[(k-1)/N,(k-1)/N + h[} \dx x \\
\leq &  \|\phi\|_{\L^\infty(\bbT)}  \sum_{k=1}^N|u^N_k-u^N_{k-1}| 
=    \|\phi\|_{\L^\infty(\bbT)} \|\partial_-^Nu^N\|_{\L^1_N} \leq \|\phi\|_{\L^\infty(\bbT)} \|\partial_-^Nu^N\|_{\L^4_N}.     
\end{align*}
Taking the limit $h\to 0$, we conclude by dominated convergence that
\[
\int_{\bbT} \iota_Nu^N(x) \phi'(x) \dx x \leq   \|\phi\|_{\L^\infty(\bbT)} \|\partial_-^Nu^N\|_{\L^4_N}.
\]
Taking the supremum over $\phi\in C^1(\bbT)$ with $\|\phi\|_{\L^\infty(\bbT)}\leq 1$ and noting that $(\rho^N)^{\alpha/4}=(\iota_Nc^N)^{\alpha/4} =\iota_N(c^N)^{\alpha/4}  = \iota_Nu_N$, we conclude that 
\[
\|(\rho^N)^{\alpha/4}\|_{\BV(\bbT)} =  \|\iota_Nu^N\|_{\BV(\bbT)} \leq \|\partial_-^Nu^N\|_{\L^4_N}.%
\]
\end{proof}

\begin{lemma}[Improved integrability]\label{lem:ImprovedRegularity}
    Let the family of curves $c^N\in\L^1(0,T;\L^1(\bbT))$ satisfy the a priori bounds \eqref{eq:AprioriBounds}. Then, there exists an exponent $\qcrit:=\alpha+1$ such that
    \[
    \sup_{N\in\N}\|\rho^N\|_{\L^{\qcrit}([0,T]\times\bbT)}<\infty.
    \]
\end{lemma}
\begin{proof}
\smallskip
\noindent
\underline{\emph{Step 1.}} We set $u^N:=(c^N)^{\alpha/4}$ and $\upsilon^N:=\iota_Nu^N$. 
We first show, that $\upsilon^N$ is uniformly bounded in $\L^4\L^\infty$.
To arrive at the claim, we observe the static estimate
\begin{align}\label{eq:StaticHoelder}
   \|u^N\|_{\L^\infty_N} \leq \|u^N\|_{\L^1_N} + \|\partial_\pm^Nu^N\|_{\L^1_N},    
\end{align}
which is the discrete analog of the continuous embedding $\W^{1,1}\subset\L^\infty$.
Now, let $\gamma>0$ sufficiently small. Using Hölder's inequality with exponents 
$q,q^*\geq 1$, i.e. $1/q+1/q^*=1$, we have
\[
    \|\upsilon^N\|_{\L^1} = \int_\bbT (\upsilon^N)^{1-\gamma}(\upsilon^N)^{\gamma}\dx x \leq \biggl(\int_\bbT (\upsilon^N)^{(1-\gamma)q^*}\d x\biggr)^{\! 1/q^*}\biggl(\int_\bbT(\upsilon^N)^{\gamma q}\d x\biggr)^{\! 1/q}.
\]
Choosing, $\gamma q=4/\alpha$, which is possible for $\gamma<4/\alpha$, the second 
term is just given by $\bigl(\int_\bbT(\upsilon^N)^{\gamma q}\d x\bigr)^{1/q}
=\|\rho^N\|_{\L^1}^{\alpha\gamma/4}$. For this $q$ we have 
$q^*=4/(4-\alpha\gamma)$, and hence we conclude
    \[
    \|\upsilon^N\|_{\L^1} \leq \|\upsilon^N\|^{1-\gamma}_{\L^p}\cdot  \|\rho^N\|_{\L^1}^{\alpha\gamma/4}, \quad \text{with} \quad p>\max\Bigl\{\frac{4-4\gamma}{4-\alpha\gamma},1\Bigr\}.
    \]
    Note that for all $\alpha>0$, there exists $\gamma\in]0,4/\alpha[$ such that such a $p>1$ exists.
    The last estimate implies with Young's inequality that
    \[
    \|\upsilon^N\|_{\L^1} \leq \|\upsilon^N\|^{1-\gamma}_{\L^\infty}\cdot  \|\rho^N\|_{\L^1}^{\alpha\gamma/4} \leq (1-\gamma)\|\upsilon^N\|_{\L^\infty} + \gamma\|\rho^N\|_{\L^1}^{\alpha/4}. 
    \]
    Inserting this into \eqref{eq:StaticHoelder}, we hence obtain
\[
\|\upsilon^N\|_{\L^\infty_N} \leq \|\upsilon^N\|_{\L^1_N}
  + \|\partial_\pm^Nu^N\|_{\L^1_N} \leq (1-\gamma)\|\upsilon^N\|_{\L^\infty}
  + \gamma\|\rho^N\|_{\L^1}^{\alpha/4}+ \|\partial_\pm^Nu^N\|_{\L^1_N}.
\]
Hence, we conclude that $\|\upsilon^N\|_{\L^\infty_N} 
\leq \|\rho^N\|_{\L^1}^{\alpha/4}+ \frac 1 \gamma \|\partial_\pm^Nu^N\|_{\L^1_N}
\leq  \|\rho^N\|_{\L^1}^{\alpha/4}+ \frac 1 \gamma \|\partial_\pm^Nu^N\|_{\L^4_N}$,
where in the last inequality we have used Jensen's inequality. Since $\rho^N$ is 
uniformly bounded in $\L^\infty\L^1$ by the bound on the energy  
\eqref{eq:AprioriBounds}, and, in addition, the discrete spatial regularity 
in Lemma \ref{lem:SpatReg} which provides $\|\partial_\pm^N u^N\|_{\L^4\L^4_N} 
\leq C_\mathrm{diss}$, we conclude that $\upsilon^N$ is uniformly bounded 
in $\L^4\L^\infty$.
\smallskip

\underline{\emph{Step 2.}} To prove the statement, we now recall the classical interpolation result for Lebesgue spaces, i.e. for $\L^{q_1}(\bbT),\L^{q_2}(\bbT)$ with $q_1,q_2\geq 1$ and $p_1,p_2\geq 1$ and $\theta\in[0,1]$ we have
    \[
    \bigl[\L^{p_1}\!([0,T],\!\L^{q_1}\!(\bbT)),\!\L^{p_2}\!([0,T],\!\L^{q_2}\!(\bbT))\bigr]_\theta \!\simeq\! \L^{p_\theta}\!([0,T],\![\L^{q_1}\!(\bbT),\!\L^{q_2}\!(\bbT)]_\theta)\! \simeq\! \L^{p_\theta}\!([0,T],\!\L^{q_\theta}\!(\bbT))
    \]
    where $\frac{1}{p_\theta} = \frac{1-\theta}{p_1} + \frac{\theta}{p_2}$ and  $\frac{1}{q_\theta} = \frac{1-\theta}{q_1} + \frac{\theta}{q_2}$.
    Let first $\alpha\in]0,4]$. By the time-uniform bound on $\rho^N$, we conclude the uniform bound $\upsilon^N\in\L^\infty\L^{4/\alpha}$. Hence, we conclude that $u^N\in\L^\infty\L^{4/\alpha}\cap\L^4\L^\infty\subset\L^{4+4/\alpha}([0,T],\bbT)$, which implies that $\rho^N = (u^N)^{4/\alpha}\in\L^{1+\alpha}([0,T]\times\bbT)$ is uniformly bounded. For $\alpha\geq 4$, we observe that $\upsilon^N\in\L^4\L^\infty$ is equivalent to $\rho^N=(\upsilon^N)^{4/\alpha}\in\L^\alpha\L^\infty$. Hence, in addition with the uniform energy bound, we have $\rho^N\in\L^\infty\L^1\cap\L^\alpha\L^\infty\subset \L^{\alpha+1}([0,T]\times\bbT)$. This proves the claim.
\end{proof}
The time regularity of the curve will be implied by suitable integrability of the fluxes. 
Hence, we first show compactness for the
fluxes by uniform integrability. Here the improved regularity from Lemma~\ref{lem:ImprovedRegularity} and the estimate~\eqref{eq:MagicalEstimate} are crucial. 
The proof is analogous to~\cite{HeinzeMielkeStephan2025}.
\begin{lemma}[Uniform integrability of fluxes]\label{lem:UniIntFluxes}
For the admissible curves $(c^{N},J^{N})$,
there is a constant $C_{\mathrm{flux}}>0$, such that for all $N\in\N$ we have 
\[
\frac 1 N \sum_{k=1}^N \int_0^T\CC(J^N_k) \d t \leq C_\mathrm{flux}.
\]
Moreover, there exists a curve of fluxes $j\in\L^1(0,T;\L^1(\bbT))$ with $\int_0^T\int_\bbT\CC(j)\d x \d t \leq C_\mathrm{flux}$ and $\iota_NJ^N \rightharpoonup j$ weakly in $\L^1(0,T;\L^1(\bbT))$ (up to a subsequence). In addition, $\mathcal{I}_NJ^N=j^N$ converges in the space $\mathcal{M}([0,T]\times\bbT)$ (up to a subsubsequence) also to $j\in\L^1([0,T]\times\bbT)$.
\end{lemma}
\begin{proof}
By the estimate \eqref{eq:MagicalEstimate}, we have 
\[
\CC(J^N) \leq \frac{q}{q-1} \CCC\bigl(J^N \big|m_\alpha(c_{k-1},c_k,c_{k+1})\bigr)+ \frac{4 m_\alpha^q}{q-1}.
\]
Using~\eqref{eq:MobilityInequality}, we estimate
\[
m_\alpha =\! \sigma_{\alpha}(c_{k-1}^{N},c_{k}^{N},c_{k+1}^{N})\Bigl(c_{k}^{N}\sqrt{c_{k-1}^{N}c_{k+1}^{N}}\Bigr)\leq \max\bigl\{c_{k}^{\alpha},\sqrt{c_{k-1}c_{k+1}}^{\alpha}\bigr\} \leq c_{k-1}^{\alpha}+ c_{k}^{\alpha}+c_{k+1}^{\alpha}.
\]
Choosing $q:=(1{+}\alpha)/\alpha>1$, i.e.\ $\alpha q=q_\mathrm{crit}$ from Lemma~\ref{lem:ImprovedRegularity}, together with a discrete
Hölder inequality, we get that $m_\alpha^q \in\L^1(0,T;\L^1_N)$.
Moreover, we have by monotonicity \eqref{eq:Monotonicity}
\[
\CCC\bigl(J^N|m_\alpha(c_{k-1},c_k,c_{k+1})\bigr) \leq \CCC(N^2J^N_k|N^4 m_\alpha).
\]
Together this implies that
\begin{align*}
\frac 1 N \sum_{k=1}^N \int_0^T\CC(J^N) \d t & \leq \frac 1 N \sum_{k=1}^N \int_0^T \CCC(N^2J^N_k|N^4 m_\alpha) \d t + \|m_\alpha^q\|_{\L^1\L^1_N} \\
& \leq C_\mathrm{diss} + C \|c^N\|_{\L^{\alpha q}\L^{\alpha q}}<\infty \,,
\end{align*}
and proves the first claim. 

For the embedded fluxes we estimate
\[
\int_0^T\int_{\bbT}\CC(\iota_NJ^N)\dx{x}\d t \leq \frac 1 N \sum_{k=1}^N \int_0^T\CC(J^N) \d t <\infty.
\]
Hence, by the criterion of de la Vallée Poussin, we conclude that there exists $j\in\L^1(0,T;\L^1(\bbT))$ with $\iota_N J^N\rightharpoonup j$ in $\L^1([0,T]\times\bbT)$ for a subsequence. By Lemma \ref{lem:FluxEmbedding}, we also conclude that there is a subsubsequence such that  the induced measures $\mathcal{I}_NJ^N=j^N$  converge in $\mathcal{M}([0,T],\bbT)$ to the same limit $j\in\L^1([0,T],\bbT)$ finishing the proof. 
\end{proof}
\begin{lemma}[Time regularity]\label{lem:TimeRegularity}
Let $(c^{N},J^{N})$ be admissible curves. Then, we have that $\rho^{N}\in \BV([0,T],\bigl(\W^{2,\infty}(\bbT)\bigr)^{*})$
uniformly in $N\in\N$, i.e. we have a uniform bound on 
\[
\|\rho^{N}\|_{\TV}=\sup\biggl\{ \sum_{\ell=1}^{L}\|\rho^{N}(t_{\ell})-\rho^{N}(t_{\ell-1})\|_{*}:\ \ 0=t_{0}\leq t_{1}\leq\dots\leq t_{\ell}\leq\dots\dots\leq t_{L}=T\!\biggr\} ,
\]
where the dual norm for functions in $f\in\L^{1}(\mathbb{T})$ with respect to $\W^{2,\infty}$ is
defined by 
\[
\|f\|_{*}:=\sup_{\phi\in\W^{2,\infty}}\Bigl\{ \int_{\mathbb{T}}f\,\phi\,\dx x:\|\phi\|_{\L^{\infty}(\mathbb{T})}+\|\phi'\|_{\L^{\infty}(\mathbb{T})}+\|\phi''\|_{\L^{\infty}(\mathbb{T})}\leq1\Bigr\} .
\]
\end{lemma}

\begin{proof}

Fixing two time values, say $t_{1},t_{2}\in[0,T]$, and a test function
$\phi\in\W^{2,\infty}(\bbT)$, we observe by the use of $({c}^{N},J^{N})\in\CE_{N}$ and Lemma \ref{lem:FluxEmbedding}
that
\begin{align*}
\MoveEqLeft\bigl\langle\phi,\rho^{N}(t_{2})-\rho^{N}(t_{1})\bigr\rangle  =\bigl\langle\phi,\iota_{N}(c^{N}(t_{2})-c^{N}(t_{1}))\bigr\rangle=\bigl\langle\iota_{N}^{*}\phi,c^{N}(t_{2})-c^{N}(t_{1})\bigr\rangle\\
 &=\int_{t_{1}}^{t_{2}}\bigl\langle\iota_{N}^{*}\phi,\dot{c}^{N}(t)\bigr\rangle\d t  \ \stackrel{\mathclap{\CE_{N}}}{=}\ \int_{t_{1}}^{t_{2}}\bigl\langle\iota_{N}^{*}\phi,\Delta^{N}J^{N}(t)\bigr\rangle\d t=\int_{t_{1}}^{t_{2}}\bigl\langle\phi,\partial_{xx}{\cal I}_{N}J^{N}(t)\bigr\rangle\dx t\\
 & \leq\int_{t_{1}}^{t_{2}}\!\|\partial_{xx}\phi\|_{\L^{\infty}}\times\|{\cal I}_{N}J^{N}(t)\|_{1}\d t\leq2\int_{t_{1}}^{t_{2}}\!\|{\iota}_{N}J^{N}(t)\|_{1}\d t =
 2 |\iota_NJ^N|([t_1,t_2]\times\mathbb{T}).
\end{align*}
In particular, this implies that
\begin{align}\label{eq:Equicontinuity}
    \forall\, t_1,t_2\in[0,T]:\quad \|\rho^{N}(t_{2})-\rho^{N}(t_{1})\|_*\leq
    2 |\iota_NJ^N|([t_1,t_2]\times\mathbb{T}).
\end{align}
Summing up, using the bound on the embedded fluxes from Lemma \ref{lem:FluxEmbedding} and the boundedness of $J^N\in\L^1([0,T],\L^1_N)$ from Lemma \ref{lem:UniIntFluxes}, we arrive at the uniform TV-bound 
\[
\sum_{\ell=1}^{L}\|\rho^{N}(t_{\ell})-\rho^{N}(t_{\ell-1})\|_{*}\leq2 \sum_{\ell}\int_{t_{\ell-1}}^{t_{\ell}}\|J^{N}(t)\|_{1}\d t = 2 |J^{N}|([0,T]\times\bbT_N)<\infty \,. %
\]
\end{proof}
\begin{proposition}[Strong compactness]\label{p:StrongCompactness}
There exists $\rho\in\L^{\alpha+1}([0,T],\L^{1}(\mathbb{T}))$ such that the strong convergence $\rho^{N}\to\rho$ in $\L^1([0,T],\L^{1}(\mathbb{T}))$ holds up to a subsequence.
\end{proposition}
\begin{proof}
The proof relies on the Aubin-Lions type result from \cite{RoSa03} (more precisely, Thm.~2 in addition with Prop.~1.10). As usual it combines spatial regularity, which we will deduce from Lemma \ref{lem:SpatReg} together with temporal regularity deduced from Lemma \ref{lem:TimeRegularity}. By Lemma \ref{lem:ImprovedRegularity}, we already now that $\rho^N$ is uniformly bounded in $\L^{\alpha+1}([0,T]\times\bbT)$. Hence, it suffices to show that $\rho^N\to\rho$ in $\L^1$. We perform the proof in three steps.

\smallskip
\noindent
\underline{\emph{Step 1.}}
On $\L^1(\mathbb{T})$, we define the functional
\begin{align*}%
    \mathcal{F}(\rho):=\cE(\rho) + \|\rho^{\alpha/4}\|_{\L^1(\bbT)} + \|\rho^{\alpha/4}\|^4_{\BV(\mathbb{T})}.
\end{align*}
From the bound on the energy, the bound on the dissipation functional with Lemma \ref{lem:SpatReg} and the improved integrability from Lemma \ref{lem:ImprovedRegularity}, we conclude that the sequence $\rho^N$ satisfies 
\[
\int_0^T\cE(\rho^N(t))\dx t < \infty , \quad \int_0^T \|(\rho^N)^{\alpha/4}\|_{\L^1(\bbT)}\dx t <\infty, \quad  \int_0^T \|\partial^N_\pm(c^N)^{\alpha/4}\|_{\L^4_N(\bbT)}^4\dx t <\infty.
\]
By  {Corollary}~\ref{cor:StaticBounds}, the last term provides a bound on $\int_0^T\|(\rho^N)^{\alpha/4}\|^4_{\BV(\mathbb{T})}\dx t$, which shows that the sequence $\rho^N$ is tight w.r.t. $\mathcal{F}$, i.e. $\sup_{N\in\N}\int_0^T\mathcal{F}(\rho^N(t))\dx t<\infty$. 

\smallskip
\noindent
\underline{\emph{Step 2.}}
We show that $\mathcal{F}:\L^1(\bbT)\to[0,\infty]$ is a normal coercive integrand 
(in the sense of~\cite{RoSa03}). In particular, we have to show 
\[
\forall\, c>0: \{\rho\in\L^1(\bbT):\mathcal{F}(\rho)\leq c\} \quad \text{is compact in }\L^1(\bbT).
\]
To see this we recall the classical Helly's selection criteria, which is stated as 
following. For a given sequence $f_n:\bbT\to\R$ with $\sup_{n\in\N}\bigl(\|f_n 
\|_{\L^1(\bbT)} + \|f_n\|_{\BV(\bbT)}\bigr)<\infty$ there exists a subsequence 
$f_{n_k}$ and a function $f\in\BV(\bbT)$, such that $f_{n_k}\to f$ pointwise 
almost everywhere, $\|f_{n_k}-f\|_{\L^1(\bbT)}\to 0$, and $\|f\|_{\BV(\bbT)} 
\leq \liminf_{k\to \infty} \|f_{n_k}\|_{\BV(\bbT)}$. 

Now, to prove that $\mathcal{F}$ has compact sublevels, we take $c>0$ and any 
sequence $\rho^N$ in the sublevel set of $\mathcal{F}$. Defining pointwise 
$\upsilon^N:=(\rho^N)^{\alpha/4}$ (possible because $\rho^N\geq 0$), we conclude 
that $\sup_{N\in\N}\bigl(\|\upsilon^N\|_{\L^1(\bbT)} + \| \upsilon^N 
\|_{\BV(\bbT)}\bigr)<\infty$, where we used that $x\mapsto x^4$ is monotone. 
Hence, by Helly's selection theorem, there exists $ \upsilon \in
\L^1(\bbT)$ such that, up to a subsequence, $\upsilon^N 
\to \upsilon$ in $\L^1$. In particular, we have that $\upsilon^N\to \upsilon$ 
in measure. The power is continuous and hence, $\rho^N=(\upsilon^N)^{4/\alpha} 
\to \upsilon^{4/\alpha}$ in measure\footnote{We recall that on a set with finite 
measure we have that a sequence $f_n$ converges to $f$ in measure if and only 
if every subsequence has in turn a subsequence that converges to $f$ almost 
everywhere. In particular, applying a continuous function does not change 
convergence in measure.}. Because the energy $\mathcal E$ is superlinear, we 
know that the sequence $\rho^N$ is uniformly integrable by the de la 
Vallée-Poussin theorem. Convergence in measure together with uniform 
integrability implies the strong convergence $\rho^N\to\rho$ in $\L^1(\bbT)$. 
In particular, we also observe that $\mathcal{F}$ is lower semicontinuous.

\smallskip\noindent
\underline{\emph{Step 3.}}
Finally we show the weak integral equicontinuity. We have to show that 
\[
\lim_{h\to 0} \sup_{N\in\N}\int_0^{T-h} \|\rho^N(t+h)-\rho^N(t)\|_*\d t=0. 
\]
By the bound~\eqref{eq:Equicontinuity}, it is sufficient to show
\[
\forall\, \eps>0 \exists\, h>0:\quad \sup_{N\in\N}\int_0^{T-h} |\iota_NJ^N|([t,t+h]\times \bbT)\d t<T \eps \,.
\]
Indeed, because $\iota_NJ^N$ is uniformly integrable, we find for all $\eps>0$ there is a $\delta>0$ such that for all $h\in]0,\delta[$ (note: $|\bbT|=1$) we have $\sup_{N}|\iota_NJ^N|([t,t+h]\times \bbT)< \eps$. Hence, $ \sup_{N\in\N}\int_0^{T-h} |\iota_NJ^N|([t,t+h]\times \bbT)\d t<T \eps$. 

Applying \cite[Thm.~2]{RoSa03}, we conclude that the sequence $\rho^N$ is relatively compact in $\mathcal{M}(0,T;\L^1(\bbT))$. Since $\rho^N\in\mathcal{M}(0,T;\L^1(\bbT))$ is uniformly integrable in $\L^1([0,T]\times\bbT)$ by the bound on the energy, we conclude by \cite[Prop.\,1.10]{RoSa03} that $\rho^N$ is relatively compact in $\L^1(0,T;\L^1(\bbT))$, which finishes the proof.
\end{proof}

\subsubsection*{Proof of EDP convergence: liminf estimates}%

In this part of the section, we make use of the just derived compactness to obtain $\liminf$ estimates of the entropy and dissipation potentials.
For doing so, we consider a sequence $(c^N,J^N)\in \CE_N$ such that 
\begin{equation}\label{eq:curves:finiteLE}
	\sup_{N} L_{\alpha,N}(c^N,J^N)<\infty \qquad\text{ and } \qquad 
	\sup_{N} E_N(c^N(0)) <\infty \,.
\end{equation}
In particular, the sequence $(c^N,J^N)$ satisfies the a priori bounds~\eqref{eq:AprioriBounds}.

We start with the proof of the $\liminf$-estimate for the energy in Result~\ref{result:EDP} and Theorem~\ref{thm:ConvergenceResult}.
\begin{lemma}[Liminf of energy]
Let $(c^N,J^N)\in \CE_N$ satisfy~\eqref{eq:curves:finiteLE}. 
Let $\rho$ be the limit from Proposition \ref{p:StrongCompactness}. Then, we have for all $t\geq 0$ the following liminf estimate 
\[
\liminf_{N\to\infty}E_{N}(c^N(t))\geq{\cal E}(\rho(t)).
\]
\end{lemma}
\begin{proof}
Since $(c^N,J^N)$ satisfy the a priori estimates \eqref{eq:AprioriBounds}, we conclude from Proposition \ref{p:StrongCompactness}, that $\iota_{N}c^{N}=\rho^{N}\to\rho$ in $\L^{1}([0,T],\L^{1}(\mathbb{T}))$.  Moreover, we observe that 
\begin{align*}
{\cal E}(\iota_{N}c^{N}(t)) & =\int_{\mathbb{T}}\bigl(\rho^{N}(t)\log\rho^{N}(t)-\rho^{N}(t)+1\bigr)\dx x\\
&=\sum_{k=1}^N\int_{k/N}^{(k+1)/N}\bigl(c_{k}^{N}(t)\log(c_{k}^{N}(t))-c_{k}(t)^{N}+1\bigr)\dx{x} \\
 &=\frac{1}{N}\sum_{k=1}^N \bigl( c_{k}^{N}(t)\log(c_{k}^{N}(t))-c_{k}^{N}(t)+1 \bigr) = E_{N}(c^{N}(t)).
\end{align*}
Hence, the liminf estimate follows by the lower semicontinuity of the entropy ~${\cal E}$. 
\end{proof}
\begin{proposition}[Liminf for primal dissipation potential]%
Let $(c^N,J^N)\in \CE_N$ satisfy~\eqref{eq:curves:finiteLE}. 
Let $\rho$ be the limit from Proposition \ref{p:StrongCompactness} and $j\in\L^1([0,T],\bbT)$ be the weak-* limit from Lemma~\ref{lem:UniIntFluxes}. Then, the primal dissipation potentials satisfy the $\liminf$ estimate
\[
\liminf_{N\to\infty}\int_0^TR_{\alpha,N}(c^N,J^N)\d t \geq\int_0^T{\cal R}_\alpha (\rho,j)\d t \,.
\]
\end{proposition}

\begin{proof}
For all $N\in\N$, we have $(c^N,J^N)\in\CE_N$, and hence thanks to the embeddings~\eqref{eq:ContinuityEquations} also $(\rho^N,j^N)\in\CE$, which holds in the sense of distributions. Since $\rho^N\to\rho$ in $\L^1\L^1$ and $j^N\wstarlim j$ in $\mathcal{M}([0,T]\times\bbT)$, we conclude that $(\rho,j)\in\CE$. Hence, we only need to show
\begin{align}\label{eq:liminfestimate}
\liminf_{N\to\infty}\int_{0}^{T}\frac{1}{N}\sum_{k=1}^N\CCC\bigl(N^2J_{k}^{N}|N^{4}m_\alpha(c_{k-1},c_{k},c_{k+1})\bigr)\d t & \geq \frac{1}{2} \int_{0}^{T}\int_{\mathbb{T}}\frac{j^{2}}{\rho^\alpha}\d x\d t.
\end{align}
We are going to show~\eqref{eq:liminfestimate} by duality. To do so, we recall from the a priori estimates \eqref{eq:AprioriBounds} and the improved integrability from Lemma \ref{lem:ImprovedRegularity}, that $\rho^N\to\rho\in\L^1([0,T]\times\bbT)$ and that $\rho^N$ is uniformly bounded in $\L^{\alpha+1}([0,T]\times\bbT)$. In the first step, we show that for all
$\xi\in C^{\infty}([0,T]\times\mathbb{T})$ we have that
\begin{align*}%
\limsup_{N\to\infty}\int_{0}^{T}R_{\alpha,N}^{*}(c^{N},\Delta^N\iota_{N}^{*}\xi)\d t\leq\int_{0}^{T}{\cal R}_\alpha^{*}(\rho,\partial_{xx}\xi)\d t,
\end{align*}
where we recall 
\begin{align*}
R_{\alpha, N}^{*}(c,\Delta^N\iota_{N}^{*}\xi)& =\!\frac{1}{N}\sum_{k=1}^N N^{4}m_\alpha(c_{k-1},c_{k},c_{k+1})\CC^{*}\!\!\Bigl(\!\tfrac{1}{N^2}\!\bigl(\Delta^{N}\iota_N^*\xi\bigr)_{k}\Bigr), \\
 \cR_\alpha^*(\rho,\partial_{xx} \xi) & =\! \frac 1 2\int_\bbT\!\rho^\alpha |\partial_{xx} \xi|^2  \d x.
\end{align*}
Using the symmetry of $\CC^{*}$, the monotonicity of $]0,\infty]\ni r\mapsto\CC^{*}(r)$ and the commutator estimate~\eqref{eq:CommutatorEstimate},
we compute 
\begin{align*}
R_{\alpha,N}^{*}(c^{N},\Delta^N \iota_{N}^{*}\xi) & \leq\frac{1}{N}\sum_{k=1}^N N^{4}m_\alpha(c^N_{k-1},c^N_{k},c^N_{k+1})\CC^{*}\!\!\Bigl(\tfrac{1}{N^2}\bigl(\Delta^{N}\iota_{N}^{*}\xi\bigr)_{k}\Bigr)\\
 & \leq\frac{1}{N}\sum_{k=1}^NN^{4}m_\alpha(c_{k-1}^N,c^N_{k},c^N_{k+1})\CC^{*}\!\!\Bigl(\!\tfrac{1}{N^2}\iota_{N}^{*}(|\partial_{xx}\xi|)_{k}+\frac{1}{3N^{3}}\|\xi'''\|_{\L^\infty}\!\!\Bigr).
\end{align*}
Using the elementary inequality $N^{4}\CC^{*}(r/N^{2})\leq\frac{r^{2}}{2}\cosh(r/N^{2})$,
we get 
\[
R_{\alpha,N}^{*}(c^{N},\Delta^N\iota_{N}^{*}\xi)\leq\omega_{N}\frac{1}{N}\sum_{k=1}^Nm_{\alpha}(c^N_{k-1},c^N_{k},c^N_{k+1})\frac{\Bigl(\iota_{N}^{*}(|\partial_{xx}\xi|)_{k}+\frac{1}{3N}\|\xi'''\|_{\infty}\Bigr)^{2}}{2},
\]
where we have introduced the positive factor 
\[
\omega_{N}:=\max_{k}\cosh\Bigl(\frac{1}{N^{2}}\iota_{N}^{*}(|\partial_{xx}\xi|)_{k}+\frac{1}{3N^3}\|\xi'''\|_{\L^\infty}\Bigr),
\]
which converges to one as $N\to\infty$ since $\xi\in C^{\infty}(\mathbb{T})$.
Moreover, using \eqref{eq:MobilityInequality} we have
\[
m_\alpha(c_{k-1},c_{k},c_{k+1})\!=\!\sigma_{\alpha}(c_{k-1},c_{k},c_{k+1})\bigl(c_{k}\sqrt{c_{k-1}c_{k+1}}\bigr)\!\leq\!\max\{c_{k}^{\alpha},\sqrt{c_{k-1}c_{k+1}}^\alpha\}\!=:\!A_{\alpha,k}^N(c),
\]
which provides the estimate 
\begin{align*}
R_{\alpha,N}^{*}&(c^{N},\Delta^N\iota_{N}^{*}\xi) \\
\leq& \omega_{N}\frac 1 2\frac{1}{N}\sum_{k=1}^NA_{\alpha,k}(c)\Bigl\{ \iota_{N}^{*}(|\partial_{xx}\xi|)_{k}^{2}+\frac{2}{3N}\|\xi'''\|_{\L^\infty}\iota_{N}^{*}(|\partial_{xx}\xi|)_{k}+\frac{1}{9N^{2}}\|\xi'''\|_{\L^\infty}^{2}\Bigr\} \\
= & I^N_{1}+I^N_{2}+I^N_{3},
\end{align*}
with
\begin{align*}
I_{1}^N & :=\omega_{N}\frac 1 2 \frac{1}{N}\sum_{k=1}^NA_{\alpha,k}(c)\iota_{N}^{*}(|\partial_{xx}\xi|)_{k}^{2},\\
I_{2}^N & :=\omega_{N}\frac 1 2 \frac{1}{N}\sum_{k=1}^NA_{\alpha,k}(c)\frac{2}{3N}\|\xi'''\|_{\L^\infty}\iota_{N}^{*}(|\partial_{xx}\xi|)_{k},\\
I_{3}^N & :=\omega_{N}\frac 1 2 \frac{1}{N}\sum_{k=1}^NA_{\alpha,k}(c)\frac{1}{9N^{2}}\|\xi'''\|_{\L^\infty}^{2}.
\end{align*}
Since $\xi\in C^{\infty}([0,T]\times\mathbb{T})$ and $\iota_{N}c^{N}$ is uniformly bounded in  $\L^{\max\{\alpha,1\}}([0,T]\times\mathbb{T})$,
we easily see that $\int_0^TI_{2}^N\d t,\int_0^TI_{3}^N\d t\to0$ as $N\to\infty$ by dominated convergence.

For $I_{1}^N$, we use Jensen's inequality and the strong convergence $(\iota_Nc^N)^{\alpha}\to\rho^\alpha$ from Proposition~\ref{p:StrongCompactness} to conclude
\begin{align*}
\limsup_{N\to \infty} \int_0^T& R^*_{\alpha,N}(c^N,\Delta^N\iota_N^*\xi)\d t=\limsup_{N\to \infty} \int_0^TI_1^N\d t\\
 \leq  & \frac 1 2\int_0^T\int_\bbT \rho^\alpha |\partial_{xx}\xi|^{2}\d x\d t = \int_0^T\cR_\alpha^*(\rho,\partial_{xx}\xi).
\end{align*}

With that estimate, we can now show the desired liminf-estimate \eqref{eq:liminfestimate}. Exploiting the duality of $\CC-\CC^{*}$, we have $
\langle J^{N},\Delta^{N}\iota_{N}^{*}\xi\rangle_{N}\leq R_{\alpha,N}^{*}(c^{N},\Delta^N\iota_{N}^{*}\xi)+R_{\alpha,N}(c^{N},J^{N})$. 
Hence, we conclude for $\xi\in C^{\infty}([0,T]\times\mathbb{T})$
that 
\begin{align*}
\int_{0}^{T}\langle j,\partial_{xx}\xi\rangle-&{\cal R}^{*}_\alpha(\rho,\partial_{xx}\xi)\d t \\
& \leq\liminf_{N\to\infty}\int_{0}^{T}\langle j^{N},\partial_{xx}\xi\rangle\d t-\limsup_{N\to\infty}\int_{0}^{T}R_{\alpha, N}^{*}(c^{N},\Delta^N\iota_{N}^{*} \xi)\d t\\
 & \leq\liminf_{N\to\infty}\int_{0}^{T}\langle J^{N},\Delta^{N}\iota_{N}^{*}\xi\rangle-R_{\alpha, N}^{*}(c^{N},\Delta^N\iota_{N}^{*}\xi)\d t\\
 & \leq\liminf_{N\to\infty}\int_{0}^{T}R_{\alpha, N}(c^{N},J^{N})\d t.
\end{align*}
Now, the desired liminf-estimate follows by taking the supremum over test functions for the
quadratic functional $
\xi\mapsto\int_{0}^{T}
\bra[\big]{ \langle j,\partial_{xx}\xi\rangle-{\cal R}_\alpha^{*}(\rho,\partial_{xx}\xi) } \dx t$.
\end{proof}

\begin{proposition}[Liminf estimate for the slope] 
If $\iota_N c^N \to \rho$ in $\rmL^{\max\{\alpha,1\}}(\bbT)$, then
\begin{align}
    \label{eq:LiminfStatSlope}
    \liminf_{N\to \infty} S_{\alpha,N}(c^N) \ \geq \ \cS_\alpha(\rho).
\end{align}
Moreover, if $(c^N,J^N)\in \CE_N$ satisfies~\eqref{eq:curves:finiteLE}\texttt{}
and $\rho$ is the limit from 
Proposition \ref{p:StrongCompactness}, then, it holds
\begin{align}
    \label{eq:LiminfSlopeT}
\liminf_{N\to\infty} \int_0^T S_{\alpha,N}(c^N)\d t\geq \int_0^T 
\mathcal{S}_\alpha(\rho)\d t.
\end{align}
\end{proposition}
\begin{proof}
We first show that \eqref{eq:LiminfStatSlope} implies \eqref{eq:LiminfSlopeT}. 
For this we use $S_{\alpha,N}(c)\geq 0$ and apply Fatou's lemma:
\begin{align*}
\liminf_{N\to \infty} \int_0^T S_{\alpha,N}(c^N(t))\d t 
 \overset{\text{Fatou}}\geq  
 \int_0^T \liminf_{N\to \infty} S_{\alpha,N}(c^N(t))\d t  
 \overset{\text{\eqref{eq:LiminfStatSlope}}}\geq  
 \int_0^T \cS_\alpha(\rho(t))\d t  ,
\end{align*}
where for the last estimate we used that Proposition \ref{p:StrongCompactness} 
implies $\iota_N c^N(t) = \rho^N(t) \to \rho(t) $ for a.e.~$t\in [0,T]$ in 
$\rmL^{\max\{\alpha,1\}}(\bbT)$. 

To establish the static liminf estimate \eqref{eq:LiminfStatSlope}, we use the lower bound from Assumption \ref{ass:Activity}:
\begin{align*}
   & S_{\alpha,N}(c) = \frac1N \sum_ {k=1}^N 2N^4 \,\sigma_\alpha(c_{k-1},c_k,c_{k+1})
     \, \bigl( c_k-\sqrt{c_{k-1}c_{k+1}}\bigr)^2
  \\
   & \geq \frac1N \sum_{k=1}^N N^4 \frac{2}{\alpha^2} 
      \frac{\bigl( c_k^\alpha {-} (c_{k-1}c_{k+1})^{\alpha/2}\bigr)^2}
            { \max\{c_k^\alpha, (c_{k-1}c_{k+1})^{\alpha/2}\}}
    =  \frac{2}{\alpha^2} \int_\bbT   \frac{\bigl(N^2\iota_N ( c_{k}^{\alpha} 
 -(c_{k-1}c_{k+1})^{\alpha/2}) \bigr)^2}{\iota_N\bigl(\max\{c_k^\alpha, (c_{k-1}c_{k+1})^{\alpha/2}\} \bigr)} \dx x \,,
\end{align*}
where the last equality holds because the functions in the fraction are piecewise constant. 

Introducing $w_k:=c_k^{\alpha/2}$, we can now rewrite
\begin{align*}
 N^2 \Bigl( \sqrt{c_{k-1}^{\alpha} c_{k+1}^{\alpha}}-c_{k}^\alpha \Bigr) 
 &= N^2(w_{k-1}w_{k+1}-w_k^2) \\
 &=N^2\bigl( w_{k} (w_{k-1}-2w_{k}+w_{k+1})-(w_{k}-w_{k-1})(w_{k+1}-w_{k})\bigr)\\
 &=w_{k}\bigl(\Delta^{N}w\bigr)_{k}-\bigl(\partial^N_{-}
    w\bigr)_{k}\bigl(\partial^N_{+}w\bigr)_{k}.
\end{align*}
We now replace $c$ by $c^N$ and $w$ by $w^N$, respectively, and assume $\iota_N c^N 
\to \rho$ in $\rmL^{\max\{\alpha,1\}}(\bbT)$ and $\liminf_{N\to \infty} S_{\alpha,N}(c^N) =:\beta$.  
In the case $\beta= \infty$ there is nothing to be shown. We thus consider the 
case $\beta < \infty$ and further assume (after extracting a subsequence, not 
relabeled) that $ S_{\alpha,N}(c^N)\to \beta$. With this, we are going to show that
\begin{subequations}
    \label{eq:ConvDeNu}
\begin{alignat}{2}
\iota_N \Bigl(\max\{c_k^\alpha, (c_{k-1}c_{k+1})^{\alpha/2}\}  \Bigr)
&\ \to\  \rho^\alpha &&\text{ in } \L^1(\bbT) 
\label{eq:ConvDenom}\\
\iota_N \Bigl(w_{k}^N\bigl(\Delta^{N}w^N\bigr)_{k}-
 \bigl( \partial^N_{-}w^N\bigr)_{k} \bigl(\partial^N_{+}w^N\bigr)_{k} \Bigr)
 &\ \rightharpoonup \ \omega \partial_{xx} \omega - (\partial_x \omega)^2 
  \quad &&\text{ in }  \L^1(\bbT) \label{eq:ConvNumer}.
\end{alignat}
\end{subequations}
Since $(\iota_Nc^N)_N$ converges  to $\rho$ in $\L^{\max\{\alpha,1\}}(\bbT)$ the  convergence of  
\eqref{eq:ConvDenom} follows by dominated convergence.

To show \eqref{eq:ConvNumer}, we treat both terms separately. By Lemma 
\ref{lem:SpatReg}, we have that $(\iota_N\Delta^Nw^N)_N$ is uniformly bounded 
in $\L^2(\bbT)$, and hence it converges weakly in $\L^2(\bbT)$ to a 
function that can be identified by $\partial_{xx}\omega$ (because the differential 
operator is outside). Moreover, since $\iota_Nc^N$ converges to $\rho$ in $\L^{\max\{\alpha,1\}}(\bbT)$, we have the strong convergence  
$\iota_N w^N \to \omega$ in $\L^2(\bbT)$. Hence, for the product of a weakly and a strongly 
converging sequence, we obtain the weak convergence $\iota_N(w^N\Delta^Nw^N)
= (\iota_Nw^N ) (\iota_N\Delta^Nw^N) \rightharpoonup \omega \partial_{xx} \omega$.

Corollary \ref{cor:StaticBounds} implies that the sequence
$(\iota_N\partial_\pm^N w^N)_{N\in\N}$  is compact in $\L^2(\bbT)$. 
Hence, they converge strongly and the limit is $\partial_x\omega$ (because the
differential operator is outside). So also their product converges strongly in 
$\L^1(\bbT)$. Hence, the convergence \eqref{eq:ConvNumer} is established.

To show the liminf estimate, we exploit the joint convexity of the function 
$(u,v)\mapsto\frac{u^{2}}{v}$ and the  convergences \eqref{eq:ConvDeNu} 
and conclude that 
\begin{align*}
\liminf_{N\to\infty}S_{\alpha,N}(c^N) \geq  
 \frac{2}{\alpha^2}
  \int_{\mathbb{T}}\frac{\bigl(\omega\partial_{xx}\omega
  -\bigl(\partial_{x}\omega\bigr)^{2}\bigr)^{2}}{\rho^\alpha}\dx x  = \mathcal{S}_\alpha(\rho),
\end{align*}
by the rewriting of the slope \eqref{eq:RelaxedSlope.BenamouBrenier}.
This proves the static liminf estimate \eqref{eq:LiminfStatSlope}. 
\end{proof}

\section{Chain rule and weak solutions}

Our approach to the chain rule is stimulated by the usage of the Hilbert space
$\rmL^2\rmL^2=\rmL^2([0,T]{\times}\bbT)$ as in
\cite{FehGes23NELD,GesHey25PMEL}, but we use a much more direct (and probably
less general) approach by introducing the \emph{modified flux $ V $} and
the \emph{modified slope $\Sigma$} given by
\begin{equation}
  \label{eq:Modi.Flux.Slope}
   V  = \rho^{-\alpha/2} j \quad \text{and} \quad \Sigma = -\frac2\alpha
  \Bigl(\partial_{xx}\bigl( \rho^{\alpha/2}\bigr) - 4 \bigl| \partial_x
  (\rho^{\alpha/4})\bigr|^2 \Bigr).
\end{equation}
They are chosen such that the dissipation functional $\mathcal{D}_{\alpha}$ in~\eqref{eq:ContinunousDissipationFunctional} takes, thanks to~\eqref{eq:RelaxedSlope.a}, the form
\[
\mathcal{D}_\alpha(\rho,j)= \int_0^T\int_\bbT\Bigl[ \frac{j^2}{2\rho^\alpha} +
     \frac{\rho^\alpha}2 ( \partial_{xx} \log \rho)^2 \Bigr] \dx x \dx t 
 =  \int_0^T\int_\bbT\Bigl[ \frac12 | V |^2 + \frac12 |\Sigma|^2 \Bigr] \dx x
   \dd t .
\] 
The advantage is that $\mathcal{D}_\alpha(\rho,j) <\infty$ gives a clear control of $ V $ and 
$\Sigma$ in $\rmL^2\rmL^2$.  

The desired chain rule then takes the form 
\[
 \cE(\rho(s))-\cE(\rho(r)) = -\int_r^s \int_\bbT \Sigma\  V  
  \dd x \dd t, 
\]
see Proposition \ref{pr:ChainRule} below. However, to achieve this goal for
general $\rho$ with $\mathcal{D}_\alpha(\rho,j)< \infty$, we exploit the relation first for smooth
approximations $\rho_{\eps,\delta}$, $\Sigma_{\eps,\delta}$, and
$V_{\eps,\delta}$ and need to control the passages to the limit for $\eps\to 0$ and
$\delta \to 0$. 

The following result shows that using mollifications $\rho_\eps$ and $V_\eps$ allows
a control of $V_\eps$ and $\Sigma_\eps$ for suitable $\alpha$. In other cases,
we can only proceed by assuming additionally boundedness of $\rho$, namely 
$  \rho\in \rmL^\infty([0,T] {\times} \bbT)$. Choosing a non-negative
mollifier $\phi\in \rmH^2(\R)$ with $\operatorname{supp}(\phi)\subset [-1,1]$, $\int_\R
\phi\dd y =1$ and setting 
$\phi_\eps(y) = \eps^{-1}\phi(y/\eps)$, we define the smoothed density and flux
\begin{equation}
  \label{eq:Mollification}
  \rho_\eps (t,x) =\bigl(\rho(t,\cdot)* \phi_\eps\bigr)(x) \quad
\text{and} \quad j_{\eps}(t,x) = \bigl( j(t,\cdot)*\phi_\eps\bigr)(x).
\end{equation}

The following result is classical for convex functionals, but our proof for
$\alpha\in (1,3/2)$ is slightly more general, because we only know convexity
of the slope $\rho\mapsto \cS_\alpha(\rho) $ for $\alpha\in [3/2,2]$, see
\eqref{eq:RelaxedSlope.c}.  

\begin{lemma}[Bounds for convolutions]\label{lem:BoundsConvol}
Assume that $(\rho,j)$ satisfy $\mathcal{D}_\alpha(\rho,j)<\infty$ and that $\sup_{t\in[0,T]}\cE(\rho(t))<\infty$. Define
$(\rho_\eps,j_\eps)$ via \eqref{eq:Mollification}. 

(A) If $\alpha \in [0,1]$,  then $V_\eps =\rho_\eps^{-\alpha/2}j_\eps \to
V=\rho^{-\alpha/2}j$ in $\rmL^2\rmL^2$. 

(B) If $\alpha\in [1,2]$, then $\Sigma_\eps = - 
\frac2\alpha\bigl(\partial_{xx}\rho_\eps^{\alpha/2} -4 |\partial_x
\rho_\eps^{\alpha/4}|^2\bigr) \to \Sigma $  in $\rmL^2\rmL^2$. 
\end{lemma}
\begin{proof} 

For part (A) we use $\alpha\in [0,1]$ giving the convexity of the primal dissipation potential
$(\rho,j)\mapsto \cR(\rho,j) = \frac12\iint \rho^{-\alpha} j^2 \dd x
\dx t$. By Jensen's inequality for convolutions we have
$\cR(\rho_\eps,j_\eps) \leq\cR(\rho,j)\leq \mathcal{D}_\alpha(\rho,j)<\infty$ and conclude
boundedness of $V_\eps =\rho_\eps^{-\alpha/2} j_\eps$, namely
$\|V_\eps\|_{\rmL^2\rmL^2} \leq  \|V\|_{\rmL^2\rmL^2} <\infty$. Hence, along a
subsequence (not relabeled) we have $V_\eps \rightharpoonup W$ in
$\rmL^2\rmL^2$, and $\|W\|_{\rmL^2\rmL^2}\leq \liminf_{\eps\to 0}
\|V_\eps\|_{\rmL^2\rmL^2} \leq \|V\|_{\rmL^2\rmL^2}$. To identify the limit, we 
recall that the uniform discrete a priori bound from Lemma~\ref{lem:ImprovedRegularity} passes to the limit, and we have that 
$\rho_\eps^{\alpha/2}\in\L^r\L^r$ with $r>2$, and thus obtain that 
$j_\eps =\rho_\eps^{\alpha/2}V_\eps$ is bounded in $\rmL^{q}\rmL^{q}$ with 
$q>1$. Passing to the weak limit in the last relation we obtain 
$j=\rho^{\alpha/2} W$, which implies $W= V$. Hence, part (A) is established.

For part (B) we define the convex functional 
\[
\mathfrak S(\rho):=\int_0^T\!\int_\bbT\!  \Bigl(\frac{(\partial_{xx}
   \rho)^2}{\rho^{2-\alpha}} + \frac{|\partial_x \rho|^4}{\rho^{4-\alpha} }
  \Bigr)\dd x \dd t.
\]
The convexity for $\alpha\in [1,2]$ follows from the convexity of
$\R{\times}]0,\infty[\ni (a,b)\mapsto a^\beta/b^\gamma $ for $\gamma\geq 0$ and
$\beta\geq 1{+}\gamma$.  
Moreover, the slope representations \eqref{eq:RelaxedSlope.c} implies
$\mathfrak S(\rho)\leq c_\alpha \int_0^T\cS_\alpha(\rho)\dd t\leq c_\alpha
\mathcal{D}_\alpha(\rho,j)$.  

Arguing as for part (A) with Jensen's inequality for convolutions we find  
\begin{align*}
\partial_x \rho_\eps^{\alpha/4} &=\tfrac\alpha4  \rho_\eps^{\alpha/4-1} \partial_x \rho_\eps 
  \to  \tfrac\alpha4  \rho^{\alpha/4-1}\partial_x \rho = \partial_x \rho^{\alpha/4} 
 \text{ in }\rmL^4\rmL^4 \quad  \text{and} \quad 
\\
\partial_{xx} \rho_\eps^{\alpha/2} 
 &= \tfrac\alpha2 \rho_\eps^{\alpha/2-1}\partial_{xx} \rho_\eps
  {+} \tfrac{\alpha(\alpha{-}2)}{4} \rho^{\alpha/2-2}_\eps(\partial_x\rho_\eps)^2\\
  & \quad \to \tfrac\alpha2 \rho^{\alpha/2-1}\partial_{xx} \rho
  {+} \tfrac{\alpha(\alpha{-}2)}{4} \rho^{\alpha/2-2}(\partial_x\rho)^2 
    =\partial_{xx} \rho^{\alpha/2} \text{ in }\rmL^2\rmL^2.
\end{align*}
Inserting this into the definition of $\Sigma_\eps$ we arrive at $\Sigma_\eps
\to \Sigma$, which is part (B).
\end{proof}

We are now ready to establish our chain rule, where the function $\rho$ with
$\mathcal{D}_\alpha(\rho,j)<\infty$ has to satisfy additional bounds depending on
$\alpha$, see \eqref{eq:CR.AddiCond}. Only for $\alpha=1$ we obtain the full result without further
conditions, in all other cases we need boundedness of $\rho$ or even strict
positivity. The proof relies on approximation via smooth and positive
functions.

\begin{proposition}[Chain rule]\label{pr:ChainRule}
For $\alpha>0$ consider a pair $(\rho,j)\in \CE$, i.e.~$\partial_t\rho=\partial_{xx} j$ in the sense of distributions, satisfying the a priori bounds 
  $\mathcal{D}_\alpha(\rho,j)<\infty$ and $\sup_{t\in[0,T]}\cE(\rho(t))<\infty$. Moreover, assume one of the following three additional
  conditions: 
\begin{align}
\tag*{\eqref{eq:CR.AddiCond.a}}
& \alpha=1;
\\
\tag*{\eqref{eq:CR.AddiCond.b}}
&\alpha\in {]0,2]} \ \text{ and } \  \rho \in \rmL^\infty([0,T]{\times} \bbT);
\\
\tag*{\eqref{eq:CR.AddiCond.c}}
& \alpha>0,  \ \rho \in \rmL^\infty([0,T]{\times} \bbT), \text{ and } \
\exists\,\delta>0: \ \rho(t,x) \geq \delta \text{ a.e.}
\end{align}
Then, for all subintervals
  $[r,s]\subset [0,T]$ we have the identity
\begin{equation}
  \label{eq:ChainRule.st}
  \cE(\rho(s))-\cE(\rho(r)) = -\int_r^s \int_\bbT \Sigma\  V  
  \dd x \dd t.
\end{equation}
where the modified flux $ V $ and the modified slope
$\Sigma$ are defined in \eqref{eq:Modi.Flux.Slope}.
\end{proposition}
\begin{proof}
\underline{\emph{Step 1. Smooth and positive case}:} We first consider the case that
$\rho \in \rmW^{1,2}([0,T];\rmL^2(\bbT))\cap \rmL^2\rmH^2$ with
$\rho(t,x) \geq \delta >0$. In this
case, we can first apply the classical chain rule for convex functionals in
$\rmL^2(\bbT)$, see e.g.\ \cite[Lem.\,3.3]{Brez73OMMS}, and then integrate by
parts to obtain
\begin{align}
\cE(\rho(s))-\cE(\rho(r))  \nonumber
& = \int_r^s\int_\bbT \log \rho \,\partial_t\rho \dd x \dd t 
  =  \int_r^s\int_\bbT \log \rho \:(\partial_{xx} j) \dd x \dd t
\\ 
\label{eq:ChainRuleSmooth}
&= \int_r^s\int_\bbT \partial_{xx} \bigl(\log \rho) \,j \dd x \dd t 
 = \int_r^s\int_\bbT  \Bigl(\frac{\partial_{xx} \rho}\rho -\frac{|\partial_x
   \rho|^2}{\rho^2} \Bigr) \, j   \dd x \dd t 
\\ \nonumber
&  = \int_r^s\int_\bbT  \rho^{\alpha/2} \Bigl(\frac{\partial_{xx} \rho}\rho -\frac{|\partial_x
   \rho|^2}{\rho^2} \Bigr) \, \frac{j}{\rho^{\alpha/2}}   \dd x \dd t 
   = -  \int_r^s\int_\bbT  \Sigma \,  V  \dd x \dd t. 
\end{align}
This establishes the desired identity for the smooth case. 

\smallskip
\noindent
\underline{\emph{Step 2. Smoothing of positive $\rho$}:} We start now with a
general pair $(\rho,j)$ satisfying $\mathcal{D}_\alpha(\rho,j) < \infty$ and
$\rho(t,x) \in [\delta,M]$ with $0<\delta<M=\|\rho\|_{\rmL^\infty}<\infty$. 

Using bound $| V |^2=\rho^{-\alpha}|j|^2\in \rmL^1(0,T;\rmL^1(\bbT))$ and
$\rho\leq M$, we obtain $j \in \rmL^2\rmL^2$. Using the mollification from 
\eqref{eq:Mollification} the pair 
$(\rho_\eps,j_\eps)$ still satisfies the linear continuity
equation $\partial_t\rho=\partial_{xx} j$ and the upper and lower bound
$\rho_\eps\in [\delta, M]$.  For $\eps>0$ we have $\rho_\eps \in
\rmL^2\rmH^2$. Moreover, $V_\eps\in \rmL^2\rmL^2$ yields $\rho_\eps \in
\rmW^{1,2}\rmL^2$. Hence, Step 1 can 
be applied, i.e.\ \eqref{eq:ChainRuleSmooth} holds: 
	\begin{gather}
  \label{eq:CR.eps}
  \cE(\rho_\eps(s))-\cE(\rho_\eps(r)) 
   = -  \int_r^s\int_\bbT  \Sigma_\eps \,  V _\eps \dd x \dd t \\
 \ \text{  with } \Sigma_\eps = -\frac2\alpha\bigl( \partial_{xx}
 (\rho_\eps)^{\alpha/2} -4| \partial_x (\rho_\eps^{\alpha/4})|^2\bigr) 
\text{ and }  V _\eps= \rho_\eps^{-\alpha/2}j_\eps. \nonumber
\end{gather}

\smallskip
\noindent
\underline{\emph{Step 3. Limit $\eps \to 0$}:} We keep $\delta>0$ fixed and
consider the limit $\eps\to 0$ in \eqref{eq:CR.eps}. As $\cE$ is convex we have
$\cE(\rho_\eps (t))\leq \cE(\rho(t))$ by Jensen's
inequality, and using the lower semicontinuity $\liminf_{\eps\to
  0}\cE(\rho_\eps(t)) \geq \cE(\rho(t))$ we conclude $\cE(\rho_\eps (t)) \to 
\cE(\rho(t))$ as $\eps \to 0$, for all $t\in [0,T]$.  

Moreover, for $\eps \to 0$ we obtain $ V _\eps \to  V $ strongly in
$\rmL^2\rmL^2$ by using the pointwise convergence
$\rho_\eps(t,x)^{-\alpha/2} \to \rho(t,x)^{-\alpha/2}
\in \bigl[M^{-\alpha/2}, \delta^{-\alpha/2}\bigr] $ a.e.\ in
$[0,T]{\times} \bbT$ and the strong $\rmL^2$ convergence $j_\eps \to j$.

For the slope we use $\Sigma \in \rmL^2\rmL^2$ which implies $\partial_x
\rho^{\alpha/4} \in \rmL^4\rmL^4$ and $\partial_{xx} \rho^{\alpha/2} \in
\rmL^2\rmL^2$. For $\alpha>4$ we exploit $\rho\geq \delta>0$ and for $\alpha\in
]0,4]$ we use $\rho\leq M$ to conclude $\partial_x \rho\in \rmL^4\rmL^4$. 
Similarly, we find $\partial_{xx} \rho\in \rmL^2\rmL^2$ by using $\rho\in [\delta,M]$.  
Hence, we have $\partial_x \rho_\eps \to \partial_x \rho$ in $\rmL^4\rmL^4$ and 
$\partial_{xx} \rho_\eps \to \partial_{xx}\rho$ in $\rmL^2\rmL^2$ and find 
\[
\partial_x \rho_\eps^{\alpha/4}= \tfrac\alpha4\rho_\eps^{\alpha/4 -1} \partial_x \rho_\eps \to
\tfrac\alpha4\rho^{\alpha/4-1}\partial_x \rho= \partial_x \rho^{\alpha/4} \ 
\text{ strongly in }\rmL^4\rmL^4.
\] 
Similarly, we obtain $\partial_{xx} \rho_\eps^{\alpha/2}\to \partial_{xx} \rho^{\alpha/2}$ strongly in
$\rmL^2\rmL^2$ and hence, $\Sigma_\eps \to \Sigma $ strongly in $\rmL^2\rmL^2$.

Now the limit $\eps\to 0$ in \eqref{eq:CR.eps} yields the chain rule
\eqref{eq:ChainRule.st} under the conditions in
\eqref{eq:CR.AddiCond.c}. 

\smallskip
\noindent
\underline{\emph{Step 4. The case $\alpha\leq 2$}:}  
In this case, we have to show that the lower bound $\rho\geq \delta>0$ is not
needed. We do this by considering $\rho_\delta=\delta{+}\rho$ and taking the
limit $\delta\to 0$. For $\delta>0$ we can apply Step 3 and easily see that
$\cE(\rho_\delta(t))\to \cE(\rho(t))$ as $\delta \to 0$, for all $t\in
[0,T]$.  Moreover, $ V _\delta=\rho_\delta^{-\alpha/2} j$ satisfies $| V _\delta|
\leq | V |$ such that $ V _\delta \to  V $ strongly in
$\rmL^2\rmL^2$ by dominated convergence. This part works for all $\alpha >0$. 

For the convergence $\Sigma_\delta \to \Sigma$, we observe the explicit
representation
\begin{equation}
  \label{eq:Sigma.delta}
  -\Sigma_\delta = (\rho{+}\delta)^{\alpha/2}\Bigl(\frac{\partial_{xx}\rho}{\rho{+}\delta} -
\frac{|\partial_x \rho|^2}{({\rho{+}\delta})^2}\Bigr) = 
\Bigl(\frac{\rho}{\rho{+}\delta}\Bigr)^{1-\alpha/2}\Bigl( - \Sigma
    + \frac{16\,\delta}{\alpha^2(\rho{+}\delta)} \bigl|\partial_x \rho^{\alpha/4}|^2
    \Bigr) .
\end{equation}
Thus, using $\alpha \in ]0,2]$ we find $|\Sigma_\delta|\leq |\Sigma| + 
C_\alpha |\partial_x \rho^{\alpha/4}|^2$, which is an integrable pointwise majorant
for $|\Sigma_\delta|^2$. Moreover, we easily see $\Sigma_\delta(t,x)\to \Sigma(t,x)$
a.e.\ for $\delta\to 0$. Thus, $\Sigma_\delta \to \Sigma$ strongly in~$\rmL^2\rmL^2$. Hence, the final chain rule \eqref{eq:ChainRule.st} for $\rho$
follows from that for $\rho{+} \delta$, under the conditions~\eqref{eq:CR.AddiCond.b}.

\smallskip
\noindent
\underline{\emph{Step 5. The case $ \alpha = 1 $}:} We now exploit the result
of Lemma \ref{lem:BoundsConvol}, which shows $V_\eps\to V$ and $\Sigma_\eps \to
\Sigma$ in $\rmL^2\rmL^2$ without any upper or lower bound. Again starting from 
the smooth and positive case \eqref{eq:CR.eps} we can pass to the limit to
obtain the chain rule \eqref{eq:ChainRule.st} for $\alpha=1$ and general
$(\rho,j)$ with $\mathcal{D}_\alpha(\rho,j)<\infty$. 
\end{proof} 
We note that the restriction $\alpha\leq 2$ in Step 4 of the above proof
cannot be removed easily. For $\alpha>2$ we can choose $\gamma\neq 1$ with $3/\alpha
< \gamma < 3/2$ 
and consider $\rho$ with $\rho(x)=|x|^\gamma$ for $|x|\leq 1/4$ and smooth
otherwise. Then, $\Sigma\in \rmL^2\rmL^2$ with
$\Sigma(x)= \gamma|x|^{\alpha\gamma/2-2}$ for $|x|\leq 1/4$. However,
$\Sigma_\delta$ in~\eqref{eq:Sigma.delta} satisfies
\[
\Sigma_\delta(x) = (\delta{+}|x|^\gamma)^{\alpha/2-2}
\bigl( -\delta\gamma(\gamma{-}1)|x|^{\gamma-2} + \gamma|x|^{2\gamma-2}\bigr)
\qquad\text{for } |x|\leq 1/4,
\]
and hence does not lie in $\rmL^2\rmL^2$ for $\delta>0$. It remains open to
show the chain rule (under $\rmL^\infty$ bounds) for $\alpha >2$ when no
positivity bound is assumed.

We can now prove our main result on EDB and weak solutions (Result \ref{result:EDB-weaksol}), which again uses
the additional conditions \eqref{eq:CR.AddiCond} on $\rho$ if $\alpha\neq 1$.

\begin{proof}[Proof of Result \ref{result:EDB-weaksol}]
We first show the EDB, i.e.\ for all $r,s$ with $0\leq r<s\leq T$ we
have
\begin{equation}
  \label{eq:EDBSol}
  \cE(\rho(s)) + \int_r^s\!\! \int_\bbT \Bigl(\frac{j^2}{2\rho^\alpha} +
\frac{\Sigma^2}2\Bigr) \dd x \dd t = \cE(\rho(r)).
\end{equation} Applying Proposition  \ref{pr:ChainRule} for EDI solutions $(\rho,j)$ satisfying $\mathcal{D}_\alpha(\rho,j) < \infty$, we obtain 
with the energy-dissipation inequality  and the chain rule
\eqref{eq:ChainRule.st} that 
\[
0\geq \cE(\rho(T))-\cE(\rho(0))+\mathcal{D}_\alpha(\rho,j) \!=\!\!\int_0^T\!\!\!\!\int_\bbT
\Bigl(\! {-} \Sigma\, V  + \frac12| V |^2{+}\frac12 |\Sigma|^2\Bigr) \dd x \dd t
\!=\!\! \int_0^T\!\!\!\!\int_\bbT
\frac12 \bigl| V  {-} \Sigma \bigr|^2\dd x \dd t.
\]
Thus, we conclude $ V =\Sigma$ and hence $\Sigma V = \frac12 V ^2 + \frac12 \Sigma^2$. Thus,
\eqref{eq:EDBSol} follows from the chain rule \eqref{eq:ChainRule.st} and the
relation $V=\rho^{-\alpha/2}j$.

To show that $(\rho,j)$ is a weak solution with a well-defined flux 
$j=\rho^{\alpha/2}V$ we use higher integrability of $\rho$ 
(cf.\ Lemma \ref{lem:ImprovedRegularity} for the discrete case). Using 
$\int_0^T\cS_\alpha(\rho)\dd t \leq \mathcal{D}_\alpha(\rho,j) <\infty$ 
and $\cE(\rho(t))\leq \cE(\rho(0))<\infty$ we obtain that 
EDB solutions $\rho$  satisfy $\rho \in \rmL^{q_*}\rmL^{q_*}$ 
with $q_* = \max\{4{+}\alpha,2\alpha\}$. For this, we consider first 
$\alpha\in ]0,4]$ and use $u=\rho^{\alpha/4} \in \rmL^\infty 
\rmL^{4/\alpha} \cap \rmL^4 \rmW^{1,4}$. With interpolation 
and a version of the Gagliardo--Nirenberg estimate, see e.g.~\cite[Append.~C]{HeinzeMielkeStephan2025},
we obtain 
$u \in \rmL^{4+16/\alpha} \rmL^{4+16/\alpha} $,  and $\rho = u^{4/\alpha} 
\in \rmL^{4+\alpha}\rmL^{4+\alpha}$ follows. For $\alpha\geq 4$ we have
$u=\rho^{\alpha/4} \in \rmL^\infty \rmL^1 \cap \rmL^4 \rmW^{1,4} \subset
\rmL^8\rmL^8$ and conclude $\rho=u^{4/\alpha}\in \rmL^{2\alpha}\rmL^{2\alpha}$.

Moreover, we know $\rho^{\alpha/2} \in \rmL^2\rmH^2$ and
$\rho^{\alpha/4} \in \rmL^4\rmW^{1,4}$. Hence, for a.a.\ $t\in
[0,T]$ we can apply the product rule in Sobolev spaces:
\begin{align*}
&\partial_{xx} (\rho^\alpha)= \partial_{xx}\bigl(\rho^{\alpha/2}\, \rho^{\alpha/2}\bigr) =
2\rho^{\alpha/2} \partial_{xx}( \rho^{\alpha/2}) + 2 \bigl|\partial_x (\rho^{\alpha/2})\bigr|^2,
\\
&
\partial_x (\rho^{\alpha/2}) = \partial_x \bigl(\rho^{\alpha/4}\,\rho^{\alpha/4}\bigr) =
2 \rho^{\alpha/4} \,\partial_x( \rho^{\alpha/4}) \quad \Longrightarrow \quad
\bigl|\partial_x \rho^{\alpha/2}\bigr|^2 =  4 \rho^{\alpha/2} \bigl|\partial_x
\rho^{\alpha/4}\bigr|^2 \,.
\end{align*}
With this we find the identity 
\[
j \!=\! \rho^{\alpha/2} V  =  \rho^{\alpha/2} \Sigma
  \!=\!  -\rho^{\alpha/2}\,
    \frac2\alpha \Bigl( \partial_{xx}(\rho^{\alpha/2})
    - 4 \bigl|\partial_x(\rho^{\alpha/4})\bigr|^2\Bigr)
  \!=\!  -\frac1\alpha \partial_{xx}(\rho^\alpha)
    + \frac4\alpha \bigl|\partial_x(\rho^{\alpha/2})\bigr|^2.
\]
From $ \rho^{\alpha/2} \in \rmL^{2q_*/\alpha}$ and $\partial_{xx}(\rho^{\alpha/2}), 
|\partial_x(\rho^{\alpha/4})|^2 \in \rmL^2\rmL^2$ we conclude 
$j \in \rmL^{p_\alpha}\rmL^{p_\alpha} $ with exponent $ p_\alpha= 2q_*/(q_*{+}\alpha)  = \max\bigl\{(4{+}\alpha)/(2{+}\alpha),4/3 \bigr\} $.

Of course, if we additionally know that $\rho$ is bounded, then 
$ j \in \rmL^2\rmL^2$.

Combining this with the weak form of the continuity equation finishes the proof.
\end{proof}

\appendix
 {
\section{Incompatibility with the Otto-Wasserstein gradient flow structure}\label{appendix:OttoFisher}
In this section, we discuss whether equation \eqref{eq:DLSS:alpha} can also be understood as a formal Otto-Wasserstein gradient flow of a weighted Fisher-information type functional. For this, we consider a general Otto-Wasserstein gradient flow with a nonlinear power-type mobility of the form
\begin{equation}\label{eq:OttoFisher}
	\partial_t \rho = \partial_x \bra[\big]{ \rho^{\alpha-\beta+\gamma+1} \partial_x \mathcal{F}'_{\beta,\gamma}(\rho)}
\end{equation}
with a generalized Fisher information given in terms of $\beta,\gamma\in \R$ by
\[
\mathcal{F}_{\beta,\gamma}(\rho) = \int \frac{\abs[\big]{\partial_x \rho}^{\beta} }{\rho^{\gamma}} \dd x  \,.
\]
Note, that the homogeneity in~\eqref{eq:OttoFisher} is chosen such that the resulting right-hand side is $\alpha$ homogeneous in $\rho$ matching the homogeneity of~\eqref{eq:DLSS:alpha}.
The first variation is given by
\[
\mathcal{F'}_{\beta,\gamma}(\rho) = - \beta(\beta{-}1) \frac{\partial_{xx} \rho 
\,\bra[\big]{\partial_x \rho}^{\beta-2}}{\rho^\gamma} +  (\beta{-}1) 
\gamma  \frac{(\partial_x \rho)^\beta}{\rho^{\gamma+1}} \,.
\]
The transport flux inside of the divergence in~\eqref{eq:OttoFisher} has the form
\begin{align*}
	\rho^{\alpha-\beta+\gamma+1} \partial_x \mathcal{F}'_{\beta,\gamma}(\rho) 
	&= - \beta(\beta{-}1) \partial_{xxx} \rho \, 
      (\partial_x \rho)^{\beta-2} \rho^{\alpha-\beta+1} \\
	&\quad -\beta(\beta{-}1)(\beta{-}2) (\partial_{xx} \rho)^2 
     (\partial_x \rho)^{\beta-3} \rho^{\alpha -\beta+1} \\
	&\quad +2\beta(\beta{-}1)\gamma \partial_{xx} \rho \, 
    (\partial_x \rho)^{\beta{-}1} \rho^{\alpha-\beta} \\
	&\quad -(\beta{-}1) \gamma(\gamma{+}1) (\partial_x \rho)^{\beta+1} 
    \rho^{\alpha-\beta{-}1}  \, .
\end{align*}
The classical flux inside of~\eqref{eq:DLSS:alpha} is given by
\begin{align*}
	-\partial_x ( \rho^{\alpha} \partial_{xx} \log \rho) &=- \partial_x\bra[\bigg]{\rho^\alpha \bra[\Big]{\frac{\partial_{xx}\rho}{\rho} - \frac{(\partial_x \rho)^2}{\rho^2}}} \\
	&= -\partial_x\bra[\Big]{ \partial_{xx} \rho \, \rho^{\alpha-1} - (\partial_x \rho)^2 \rho^{\alpha -2}} \\
	&= -\partial_{xxx} \rho \,\rho^{\alpha-1} - (\alpha{-}3) \partial_{xx} \rho 
     \,\partial_x \rho \, \rho^{\alpha-2}  + (\alpha{-}2) (\partial_x \rho)^3 \rho^{\alpha-3}.
\end{align*}
Comparing leading terms, we obtain $\beta=2$ and an overall prefactor $1/2$. Hence, we arrive at
\[
\tfrac{1}{2}\rho^{\alpha-2+\gamma+1} \partial_x \mathcal{F}'_{2,\gamma}(\rho) 
= -\partial_{xxx}\rho \, \rho^{\alpha-1}  +2 \gamma \partial_{xx}\rho \,
\partial_x \rho\, \rho^{\alpha-2} -\tfrac{1}{2} \gamma(\gamma{+}1) 
(\partial_x \rho)^3 \rho^{\alpha -3}.
\]
For comparing the remaining coefficients, we end up with the two algebraic constraints
\[
3-\alpha = 2\gamma \quad\text{and}\quad 2-\alpha = \tfrac{1}{2}\gamma (\gamma+1).
\]
The system has exactly two solutions: $\alpha=1$, $\gamma=1$ and $\alpha=-1$, $\gamma=2$. The first 
case is the classic DLSS equation ($\alpha=1$) and the second case falls out of our analysis.
Let us still note, that for the second case the mobility for the Otto-Wasserstein metric tensor 
is constant ($\alpha-\beta+\gamma+1=0)$, thus the resulting gradient flow falls into the classical 
$H^{-1 }$ Hilbert-space framework. The equation is formally given by
\begin{equation}\label{eq:NegMobilityFisher}
	\partial_t \rho_t = -\partial_{xx}\bra[\big]{ \rho^{-1} \partial_{xx} \log \rho} = -\partial_{xx} \bra[\Big]{ \frac{\rho_{xx}}{\rho^2} - \frac{(\partial_x\rho)^2}{\rho^3}}
	= \tfrac12\partial_{xx} \mathcal{F}_{2,2}'(\rho)
\end{equation}
with driving energy $
\frac12\mathcal{F}_{2,2}(\rho) =  \frac{1}{2} \int |\partial_x \rho|^2/\rho^2 \dx{x} 
= \frac{1}{2}\int |\partial_x \log \rho|^2 \dx{x} $.}

\bibliography{bib}
\bibliographystyle{alpha_AMs}

\end{document}